\documentclass[11pt]{article}
\usepackage{fullpage}

\usepackage[utf8]{inputenc} 
\usepackage[T1]{fontenc}    
\usepackage{hyperref}       
\usepackage{url}            
\usepackage{booktabs}       
\usepackage{amsfonts}       
\usepackage{nicefrac}       
\usepackage{microtype}      
\usepackage{tcolorbox}

\usepackage{enumerate}
\usepackage{enumitem}
\usepackage{algorithmic}
\usepackage{algorithm}
\usepackage{subfigure}
\usepackage{graphicx} 
\usepackage{caption}
\usepackage{amsmath}
\usepackage{amsthm}
\usepackage{amssymb}
\usepackage{tikz}
\usepackage{tablefootnote}
\usepackage{multirow}
\usepackage{enumerate}
\usepackage{color}
\usepackage{xcolor}
\usetikzlibrary{arrows}

\allowdisplaybreaks[4]

\usepackage{mathrsfs}

\usepackage{hyperref}
\usepackage{bm,todonotes}

\def\A{\mathcal{A}}

\allowdisplaybreaks

\newtheorem{thm}{Theorem}[section]
\newtheorem{lem}{Lemma}[section]
\newtheorem{cor}{Corollary}[section]
\newtheorem{prop}{Proposition}[section]
\newtheorem{asmp}{Assumption}[section]
\newtheorem{defn}{Definition}[section]

\newtheorem{rem}{Remark}[section]

\usepackage{amsmath,amsfonts,bm}

\def\1{\bm{1}}

\def\eps{{\varepsilon}}

\DeclareMathAlphabet{\mathsfit}{\encodingdefault}{\sfdefault}{m}{sl}
\SetMathAlphabet{\mathsfit}{bold}{\encodingdefault}{\sfdefault}{bx}{n}

\def\A{{\bf A}}
\def\a{{\bf a}}
\def\B{{\bf B}}
\def\bb{{\bf b}}
\def\C{{\bf C}}
\def\c{{\bf c}}
\def\D{{\bf D}}

\def\e{{\bf e}}

\def\g{{\bf g}}

\def\I{{\bf I}}

\def\u{{\bf u}}

\def\v{{\bf v}}

\def\X{{\bf X}}

\def\ttx{{\widetilde{\bf x}}}
\def\tth{{\widetilde{\bf h}}}

\def\x{{\bf x}}

\def\y{{\bf y}}

\def\z{{\bf z}}
\def\0{{\bf 0}}
\def\1{{\bf 1}}

\def\AM{{\mathcal A}}
\def\CM{{\mathcal C}}
\def\DM{{\mathcal D}}
\def\EM{{\mathcal E}}
\def\GM{{\mathcal G}}

\def\IM{{\mathcal I}}

\def\NM{{\mathcal N}}
\def\OM{{\mathcal O}}
\def\PM{{\mathcal P}}

\def\RB{{\mathbb R}}
\def\EB{{\mathbb E}}

\def\PB{{\mathbb P}}

\def\TB{{\mathbb T}}
\def\GB{{\mathbb G}}
\def\SB{{\mathbb S}}

\def\tx{\tilde{\x}}

\def\bet{\mbox{\boldmath$\beta$\unboldmath}}

\def\tha{\mbox{\boldmath$\theta$\unboldmath}}

\def\argmin{\mathop{\rm argmin}}

\newcommand{\PA}{\mathcal{P}_{\A}}
\newcommand{\PAT}{\mathcal{P}_{\A^{\perp}}}
\newcommand{\xc}{\widetilde{\x}^*}
\newcommand{\sigmaa}{\sigma_{\A,*}^2}
\newcommand{\sigmaat}{\sigma_{\A^\perp, *}^2}

\title{
Delayed Projection Techniques for Linearly Constrained Problems: Convergence Rates, Acceleration, and Applications
}

\author{
	Xiang Li\thanks{School of Mathematical Sciences, Peking University; email: \texttt{lx10077@pku.edu.cn}. } \\
	\and
	Zhihua Zhang\thanks{School of Mathematical Sciences, Peking University; email: \texttt{zhzhang@math.pku.edu.cn}. } \\
}
\begin{document}

\maketitle

\begin{abstract}%
	In this work, we study a novel class of projection-based algorithms for linearly constrained problems (LCPs) which have a lot of applications in statistics,  optimization, and machine learning.
	Conventional primal gradient-based methods for LCPs call a projection after each (stochastic) gradient descent, resulting in that the required number of projections equals that of gradient descents (or total iterations). Motivated by the recent progress in distributed optimization, we propose the delayed projection technique that calls a projection once for a while, lowering the projection frequency and improving the projection efficiency. Accordingly, we devise a series of stochastic methods for LCPs using the technique, including a variance reduced method and an accelerated one. We theoretically show that it is feasible to improve projection efficiency in both strongly convex and generally convex cases. Our analysis is simple and unified and can be easily extended to other methods using delayed projections. When applying our new algorithms to federated optimization, a newfangled and privacy-preserving subfield in distributed optimization, we obtain not only a variance reduced federated algorithm with convergence rates better than previous works, but also the first accelerated method able to handle data heterogeneity inherent in federated optimization.
\end{abstract}

\section{Introduction}
The constrained optimization problem is an important ingredient in optimization literature~\cite{nesterov2013introductory,boyd2004convex,linaccelerated}.
It has a lot of applications such as linear programming~\cite{dantzig1998linear}, optimal transport~\cite{villani2008optimal}, reinforcement learning~\cite{sutton2018reinforcement} and distributed optimization~\cite{boyd2011distributed}.
In this paper, we focus on a specific constrained problem--- linearly (or linear equality) constrained problem (LCP), which has many applications in statistics and optimization on which we will give a brief introduction in the next section.
In a nutshell, LCP aims to minimize a (strongly) convex function $F(\x)$ subject to several linear equality constraints on the variable $\x$ which define a feasible regime $\CM = \{\x: \A^\top\x = \bb \}$.

Due to its linearity structure, many methods have been proposed to solve LCP.
If the constraint regime is simple enough like the unit ball or the simplex, a typical method to handle LCP is projected gradient descent, or more generally, proximal gradient descent~\cite{parikh2014proximal}.
Indeed, we can define an indicator function $h(\x) := 1_{\CM}(\x)$ that takes value zero if the variable locates in the feasible regime $\x \in \CM$ and takes infinity otherwise.
In particular, we can show that the indicator function is closed and proper~\cite{boyd2004convex}, which is often required by proximal gradient descent.
Then by adding the indicator function to the original objective, we arrive at a composite optimization problem that minimizes $F(\x) + h(\x)$.
The projected gradient descent (which is equivalent to proximal gradient descent here) iteratively updates the variable according to
\[
\x_{t+\frac{1}{2}} = \x_t - \eta_t \nabla F(\x_t)
\ \text{and} \
\x_{t+1} =  \PM_{\CM} \left( \x_{t+\frac{1}{2}}  \right)
=   \mathrm{prox}_{\eta_{t}h}\left(  \x_{t+\frac{1}{2}}   \right),
\]
where $\eta_t$ is the learning rate, $\PM_{\CM}(\cdot)$ is the projection onto the feasible regime defined by $\PM_{\CM}(\x) := \argmin_{\z \in \CM} \|\z-\x\|^2$, and $\mathrm{prox}_{h}(\cdot)$ is the proximal mapping of $h(\cdot)$ defined by
\[
\mathrm{prox}_{h}(\x)= \argmin_{\z} \left\{ h(\z) + \frac{1}{2} \|\z -\x\|^2 \right\}.
\]
To ensure the feasibility of maintained sequence $\{\x_t\}$,  projected gradient descent (PGD) typically calls a projection after a gradient descent iteration is performed~\cite{nesterov2013introductory}.
This implies that PGD performs the same number for both projection and iteration.
Given projection is cheap to call, PGD is practically feasible and well-understood.
Its convergence shares a lot of similarities with unconstrained optimization methods~\cite{nesterov2013introductory}.


When a projection is impossible to call, the linear structure of the constraint renders researchers an alternative to avoid the projection.
The most natural method is to eliminate equality constraints by reformulating the feasible regime and then solving the resulting unconstrained problem by methods for unconstrained minimization~\cite{boyd2004convex,james2020penalized}.
Indeed, we can rewrite a feasible point $\x$ that satisfies $\A^\top\x = \bb$ as $\x = \y^* + (\A^{\perp}) \z$ where $\A^\top\y^* = \bb$ and the columns of $\A^{\perp}$ locate in the kernel of $\A$.
Then the objective becomes a function of $\z$, which is unconstrained.
If one wants to solve it by gradient methods, however, it will introduce a massive number of matrix-vector products to compute gradients for the new variable $\z$, whose cost will be much higher than that of a hard projection, if the rank of $\A$ is much smaller than its nullity (i.e., the dimension of the kernel of $\A$).
Another celebrated method is to use dual  or primal-dual methods~\cite{komodakis2015playing}.
For example, the Lagrange multiplier method and its augmented extension are  typical approaches for solving LCPs~\cite{birgin2014practical}.
However, these methods need to maintain additional dual iterations and thus require more memories.

In the paper, we focus on the intermediate case between the two extremes, where a projection is possible but expensive to call.
The intermediate case includes many important and interesting problems, such as linearly constrained quadratic programming (LCQP) and the global consensus problem in distributed optimization, which we will introduce in the next subsection.
The central concern in the paper is whether we can use projections the number of which is much less than the total iterations to obtain an $\eps$-suboptimal solution for LCPs.
This question is meaningful only if we focus on the primal perspective because no projection is needed by dual methods for LCPs.
We will give an affirmative answer to the question.

\subsection{Examples of Linearly Constrained Problems (LCPs) 
}
We now present  several important applications from which the interest
for LCPs stems.

\paragraph{Example 1: Linearly Constrained Quadratic Programming (LCQP).}
LCQP  considers the following optimization problem
\[
\min_{\x \in \RB^p} F(\x) :=\x^\top \C \x + \g^\top \x  
\ \text{s.t.} \ \A^\top\x = \bb, 
\]
where $\C \in \RB^{p \times p}$ is a positive definite matrix and $\A \in \RB^{p \times m}$ is the problem-dependent matrix for linear constraints.
Assume that $\C$ has a finite sum structure, i.e., $\C = \frac{1}{N}\sum_{i=1}^N \c_i \c_i^\top$ for a set of $N$ vectors $\{ \c_i \}_{i=1}^N$.
Noting the gradient involves the full evaluation of $\C \x$, which is time-consuming.
If $N$ is quite large, we can estimate $\C$ by sampling a small portion of $\{ \c_i \}_{i=1}^N$, yielding a stochastic gradient.

\paragraph{Example 2: Generalized Lasso Problem.}
Variable selection has received great attention in statistics and machine learning.
~\cite{she2010sparse,tibshirani2011solution} introduced the generalized Lasso problem as follows:
\begin{equation}
\label{eq:g_lasso}
\argmin_{\tha} \frac{1}{2} \| \y - \X \tha \|_2^2 + \lambda \|\D \tha\|_1,
\end{equation}
where $\X \in \RB^{N \times p}$ is the design matrix, $\y \in \RB^{N}$ is the response, and  $\D \in \RB^{r \times p}$ is a fixed, user-specified regularization matrix. 
We choose $\D$ so that sparsity of $\D\tha$ corresponds to some other desired behavior depending on the application, including the fused lasso and trend filtering~\cite{tibshirani2011solution}. 
When $\mathrm{rank}(\D) = r$  ($r \le p$), ~\cite{tibshirani2011solution} showed the generalized lasso can be converted to the classical lasso problem.
When $r > p$, such a reformulation is not possible.
However, in this case,~\cite{james2020penalized} showed  that it can be formulated as an instance of LCP.
\begin{lem}
If $r > p$ where $\mathrm{rank}(\D) = p$, we can find matrices $\A, \C$ and $\widetilde{\X}$ with appropriate dimensions such that the solution of~ Eq.\eqref{eq:g_lasso} is equal to $\tha = \C \bet$, where $\bet$ is given by
\begin{equation}
\label{eq:lcp_lassp}
\argmin_{\bet} \frac{1}{2} \| \y - \widetilde{\X}\bet \|_2^2 + \lambda \|\bet\|_1
\ \mathrm{s.t.} \
\A^\top \bet = \0.
\end{equation}
\end{lem}
\cite{deng2020efficient} developed a semismooth Newton-based augmented Lagrangian method to solve the above LCP.
Such a reformulation considers only linear equality constraints.
As an extension,~\cite{gaines2018algorithms,james2020penalized} additionally considered linear inequality constraints and proposed a more general framework named as constrained Lasso.
When $N$ is quite large, the finite-sum structure of~Eq.\eqref{eq:lcp_lassp} renders us able to use stochastic gradients generated by uniformly sampling a small batch of data points from the $N$-size training set to save the expensive computation.

\paragraph{Example 3: Network Flow Optimization Problem.}
Consider a network represented by a directed graph $\GM = (\NM, \EM)$ with node set $\NM = [n]$ and edge set $\EM = [e]$.
The network is deployed to support a single information flow specified by incoming rates $b_i > 0$ at source nodes and outgoing rates $b_i < 0$ at sink nodes.
We collect the rate requirements in a vector $\bb = (b_1, b_2, \cdots, b_n)^\top$ that satisfy $\sum_{i=1}^n b_i  = 1_n^\top \bb = 0$ in order to ensure problem feasibility.
The goal of a network flow optimization problem is to determine a flow vector $\x = (x_1, \cdots, x_e)^\top \in \RB^{e}$ with $x_l$ denoting the amount of flow on edge $l = (i, j)$.
Flow conservation is enforced as a linear equality constraint $\A^\top \x = \bb$, where $\A \in \RB^{e \times n}$ is the edge-node incidence matrix defined by
\[
\A(l, i) = 
\left\{\begin{array}{ll}
1 & \text { if edge } l \text { leaves node } i,\\
-1 & \text { if edge } l \text { enters node } i,\\
0 & \text { otherwise. }
\end{array}\right.
\]
Then convex min-cost flow network optimization problem~\cite{zargham2013accelerated} is defined as
\[
\min_{\x \in \RB^e} F(\x) := \sum_{l=1}^e f_l(x_l)
\ \text{s.t.} \ \A^\top\x = \bb.
\]

\paragraph{Example 4: (Global Consensus of) Distributed Optimization.}
In typical distributed optimization, we want to find a global vector that minimizes the average of $n$ local objective function, i.e., $\min_{\x \in \RB^d} \frac{1}{n} \sum_{k=1}^{n} f_{k}(\x) $, where $f_{k}(\x) =  \EB_{\xi\sim \DM_k} f(\x; \xi)$ is the $k$-th local loss function evaluated on local data distribution $\DM_k$.
The global consensus of it~\cite{boyd2011distributed,parikh2014proximal} is then formulated as
\begin{equation}
\label{eq:global_consensus}
\min_{\x^{(1)}, \x^{(2)}, \cdots, \x^{(n)} \in \RB^d} 
\frac{1}{n} \sum_{k=1}^{n} f_{k}(\x^{(k)})  
\quad \mathrm{s.t.} \quad
\x^{(1)} = \x^{(2)} = \cdots = \x^{(n)},
\end{equation}
where the $\x^{(i)} \in \RB^d$ are the local parameters.
If we concatenate all the local variables as $\x = [(\x^{(1)})^\top, (\x^{(2)})^\top, \cdots, (\x^{(n)})^\top]^\top \in \RB^{nd}$, we can rewrite~Eq.\eqref{eq:global_consensus} in a simpler form 
\begin{equation}
\label{eq:constraint}
\min_{\x \in \RB^{nd}} 
F(\x) := \sum_{k=1}^{n} f_{k}(\x^{(k)})  
\quad \mathrm{s.t.} \quad
\A^\top \x = \0, 
\end{equation}
where
\begin{equation}
\label{eq:B}
\A = \I_d \otimes \B^\top \in \RB^{dn \times d(n-1)}
\ \text{and} \
\B = \left(
\begin{matrix}
1 & -1 & 0 & \cdots &0 & 0\\
0 & 1 & -1 & \cdots & 0 & 0 \\
\vdots & \vdots & \vdots & \ddots & \vdots &\vdots \\
0 & 0 & 0 & \cdots & 1 & -1
\end{matrix}
\right) \in \RB^{(n-1) \times n}.
\end{equation}
From~Eq.\eqref{eq:constraint}, the distributed optimization is essentially a constrained optimization problem where the variable lies in a higher dimension space (because it is the concatenate of all local variable $\x = [(\x^{(1)})^\top, (\x^{(2)})^\top, \cdots, (\x^{(n)})^\top]^\top \in \RB^{nd}$) and the constraint is an equality $\A^\top \x=\0$. 
Although this consensus formulation involves more variables, it is more amenable to the analysis of distributed procedures~\cite{boyd2011distributed}.

The global consensus problem has been studied extensively in the literature.
Primal methods include distributed subgradient method~\cite{nedic2009distributed} and the EXTRA method~\cite{shi2015extra}, dual methods include distributed dual averaging~\cite{duchi2011dual}, and primal-dual based methods such as the Alternating Direction Method of Multipliers (ADMM)~\cite{boyd2011distributed} and its variants~\cite{zhang2017distributed}.
Recently, a new distributed computing paradigms named Federated Learning (FL) becomes quite famous for its privacy-preserving property~\cite{kairouz2019advances}.
Though it also can be reformulated as an instance of problem~\eqref{eq:constraint}, it faces more challenges, including expensive communication costs, unreliable connection, massive scale, and privacy constraints~\cite{li2020federated}.
We will introduce FL formally and detailedly in Section~\ref{sec:FL} and give two novel primal methods that overcome the statistical heterogeneity inherent in FL.

Except for the methods introduced previously, other optimization methods (not exclusively) have been proposed for LCPs.
~\cite{fletcher1972algorithm,murtagh1978large} applied quasi-Newton methods to solve large scale non-linear LCPs, which, however, suffer great computation complexity.
~\cite{best1975feasible} used conjugate directions to minimize a nonlinear function subject to linear inequality constraints.
\cite{mahdavi2012stochastic} developed a primal-dual stochastic optimization algorithm that only requires one projection at the last iteration to produce a feasible solution in the given domain. 
~\cite{hong2018gradient,lu2020finding} proposed primal-dual algorithms that converge to the second-order stationary solutions for nonconvex LCPs.
~\cite{armand2017globally} presented a Newton-like method applied to a perturbation of the optimality system that follows from a reformulation of the initial problem by introducing an augmented Lagrangian function.
Our method is from the primal perspective and is mainly motivated by the recent progress in distributed optimization.

\subsection{Motivation}
For simplicity, we assume $\bb=\0$ and thus the equality constraint becomes $\A^\top \x = \0$.
In the latter section, we will explain the importance and feasibility of the assumption.
A typical operation making $\x$ satisfy the linear constraint is projection. 
Let $\PA$ be the projection onto the column space of $\A$ (denoted $\mathcal{R}(\A)$) and $\PAT$ the projection onto the null space of $\A^\top$ (denoted  $\mathcal{R}(\A^\perp)$).
Then, the constraint $\A^\top \x = \0$ is equivalent to requiring $\x$ has no component on $\mathcal{R}(\A)$\footnote{This is because $\PA(\x) = \A(\A^\top \A)^{\dag}\A^\top \x$ with $\dag$ the pseudo inverse.}.

In our last example of distributed optimization, an interesting observation is that $\PAT(\x)$ forms in $\bar{\x} \otimes \1_n$ with $\bar{\x}$ being the average of $n$ block components of $\x$, as show in Lemma~\ref{lem:proj}.
\begin{lem}
	\label{lem:proj}
	For $\x = [(\x^{(1)})^\top, (\x^{(2)})^\top, \cdots, (\x^{(n)})^\top]^\top  \in \RB^{nd}$ and $\A$ is given in~Eq.\eqref{eq:B}, then we have 
	\begin{equation*}
	\PAT(\x) = \bar{\x} \otimes \1_n \ \text{with} \
	\bar{\x} = \frac{1}{n} \sum_{k=1}^n \x^{(k)}.
	\end{equation*}
\end{lem}
With Lemma~\ref{lem:proj}, we have a novel interpretation of communication that is inevitable in distributed optimization.
A centralized communication typically synchronizes all devices with the average of all local parameters, which exactly has the same effect of $\PAT$.
Indeed, the original distributed optimization is now formulated as an instance of LCPs, thus projection is the synonyms for synchronization.
The observation bridges distributed optimization and single-machine LCP.
As a result, one could apply methods for LCPs to solve distributed optimization problems and vice versa.
The former idea is stale and has been used to design new distributed algorithms. 
For example, \cite{alghunaim2020linear,hong2018gradient} proposed and analyzed an incremental implementation of the primal-descent dual-ascent gradient method used for the solution of LCPs, and applied it to solve distributed optimization problems.
\cite{pathak2020fedsplit} proposed a new algorithm named as {FedSplit} 
by applying deterministic methods for monotone inclusion problems (of which Problem~\eqref{eq:constraint} is an instance), showing that {FedSplit} converges to optima of the original distributed optimization problem with linear convergence rate.
However, the latter is rarely considered.

\begin{algorithm}[tb]
	\caption{Local SGD}
	\label{alg:local_sgd}
	\begin{algorithmic}
		\STATE {\bfseries Input:} function $\{f_k\}_{k=1}^n$, initial point $\x_0$, step size $\eta_t$, communication set $\IM_T$ with $\mathrm{gap}(\IM_T) = E$.
		\STATE {\bfseries Initialization:} let $\x_{0}^{(k)} = \x_0$ for all $k$.
		\FOR{$t=1$ {\bfseries to} $T$}
		\FOR {each device $k=1$ {\bfseries to} $n$}
		\STATE  $\x_t^{(k)} = \x_{t-1}^{(k)} - \eta_{t-1} \nabla f_k(\x_{t-1}^{(k)}; \xi_{t-1}^{(k)})$  
		\IF {$t \in \IM_T$}
		\STATE {$\x_t^{(k)} \gets \frac{1}{n}\sum_{j=1}^n \x_t^{(j)}  $ \quad \# synchronization}
		\ENDIF
		\ENDFOR
		\ENDFOR
		\STATE $\hat{\y} \gets  \frac{1}{nW_T}\sum_{j=0}^{T-1} (1-\mu \eta)^{T-j-1} \sum_{k=1}^n \x_j^{k}$ where $W_T = \sum_{j=0}^{T-1}(1-\mu \eta)^{T-j-1}$.
		\STATE {\bfseries Return: $\hat{\y}$.} 
	\end{algorithmic}
\end{algorithm}

A representative distributed optimization method is distributed SGD~\cite{zinkevich2010parallelized} that synchronizes local parameters after each device performs one step of stochastic gradient descent (SGD).
In the literature of constrained optimization,  the most famous method is perhaps the projected (stochastic) gradient descent (P-SGD), where the constraint is enforced in a separate step by projecting onto the  constraint space after an update is performed~\cite{nesterov2013introductory}.
Viewing synchronization and projection equivalently, the counterpart  of P-SGD in the context of LCPs is distributed SGD.
Recent progresses in distribution optimization find that lowering the frequency of communication is able to improve communication frequency substantially.
The most famous optimization method is Local SGD~\cite{lin2018don,stich2018local,bayoumi2020tighter,woodworth2020local,woodworth2020minibatch} or Federated Average ~\cite{mcmahan2016communication,li2019convergence,konevcny2017stochastic}.
Local SGD (Algorithm~\ref{alg:local_sgd}) runs SGD independently in parallel on
different workers and averages the sequences only once in a while, lowering the communication frequency.
Here $\IM_T \subset [T]$ denotes the set of iterations that calls a communication, and $\mathrm{gap}(\IM_T)$ is the largest interval between two ordered sequential elements in $\IM_T$.
Typically, we set $\IM_T = \{0, E, 2E,\cdots, E\lfloor \frac{T}{E}\rfloor  \}$, which implies we perform a communication for synchronization every $E$ iterations.

The effectiveness of Local SGD inspires us that we can similarly modify P-SGD to improve projection efficiency.
The current question is its feasibility, i.e., whether it is possible to call projections after several (or constant) steps of unconstrained SGDs rather than alternating between one unconstrained SGD and one projection, and whether such a method is projection efficient, which is measured by the required number of projections to obtain a solution with satisfactory accuracy.
To answer these questions, we are motivated to propose and analyze the delayed projection technique that performs a projection once in a while rather than at each iteration.

\subsection{Our Contribution}
\begin{table}[t]
	\newcommand{\tabincell}[2]{\begin{tabular}{@{}#1@{}}#2\end{tabular}}
	\centering
	\resizebox{\textwidth}{39mm}{
	\begin{tabular}{|c|c|c|}
		\hline
		Methods & Generally Convex $(\mu =0)$ & Strongly Convex  $(\mu > 0)$\\
		\hline
		P-SGD 
		& $\OM\left( \frac{L}{\eps}  + \frac{ \sigma^2}{\eps^2}\right) \cdot \Delta^2$     
		& $\widetilde{\OM} \left(  \kappa + \frac{\sigma^2}{\mu \eps}  \right)$       \\  
		\hline
		\tabincell{c}{	DP-SGD (Alg~\ref{alg:multi}) \\ Theorem~\ref{thm:simple} and~\ref{thm:complicate}}
		& $\OM\left( \frac{L}{\eps}  + \frac{ \sigmaat}{E\eps^2} +  \frac{\sqrt{(E-1)L}\widetilde{\sigma}_{\A,*}}{E\eps^{1.5}} \right) \cdot \Delta^2$     
		& $\widetilde{\OM} \left(  \kappa + \frac{\sigmaat}{E \mu \eps} +  \frac{\sqrt{(E-1)L}\widetilde{\sigma}_{\A,*}}{E \mu \eps^{0.5}}  \right)$       \\  
		\hline
		P-SVRG~\cite{xiao2014proximal} &  $\widetilde{\OM}\left( \frac{1}{\eps} \right) \cdot L\Delta^2$   &  
		$\widetilde{\OM}\left( \kappa \right)$     \\
		\hline
		\tabincell{c}{	DP-SVRG (Alg~\ref{alg:multi_SVRG}) \\ Theorem~\ref{thm:svrg}}
		& 
		 \tabincell{c}{If $m=E$, $\OM\left( \frac{1}{\eps} \right) \cdot L\Delta^2$;\\
			 If $m=N$, $\OM\left( \frac{1}{\eps} \right) \cdot L\Delta^2\cdot \max\{ 1, \frac{N}{E} \}$.}
		 &  
		\tabincell{c}{If $m=\kappa$ or $m=E$, $\widetilde{\OM}\left( \kappa \right) $; \\If $m=N$, $\widetilde{\OM}\left( \max\left\{ \kappa, \frac{N}{E} \right\} \right)$.}  \\
		\hline
		P-ASVRG~\cite{shang2018asvrg}
		&\tabincell{c}{ $\widetilde{\OM}\left(  N \sqrt{\frac{L}{\eps} }\Delta  \right)$ \\
		$\widetilde{\OM}\left(  N + \sqrt{N \cdot\frac{L}{\eps} }\cdot \Delta \right)$ (by~\cite{allen2016optimal})
	}
		&  $\widetilde{\OM}\left( N + \sqrt{N \kappa}  \right)$   \\
		\hline
 \tabincell{c}{	DP-ASVRG (Alg~\ref{alg:multi_acc_SVRG}) \\  Theorem~\ref{thm:acc_svrg_general} and~\ref{thm:acc_svrg_strong}}
 & \tabincell{c}{
 	If $m=E=1$ and using~\cite{allen2016optimal}: $\widetilde{\OM}\left(\sqrt{\frac{L}{\eps}} \right);$\\
 	If $m=E=N$ and using~\cite{allen2016optimal}:$\widetilde{\OM}\left(\left( \frac{L}{\eps} \right)^{2/3} \right);$\\
 	If $m=(\frac{L}{\eps})^{1/4} \Delta^{1/2}E$,
 	$ \widetilde{\OM}\left(\left(\frac{L}{\eps}\right)^{3/4} \Delta^{3/2}\right)$.
}
 & \tabincell{c}{If $m=E=1: \widetilde{\OM}\left(\sqrt{\kappa} \right);$\\
	If $m=E=N: \widetilde{\OM}\left(\kappa^{2/3} \right);$\\
	If $E =\max\left\{ \sqrt{ N/\kappa} , 1\right\},$\\
	and $m=\kappa E: \widetilde{\OM}\left( \kappa\right).$}\\
		\hline
	\end{tabular}}
	\caption{
		Projection complexity of different algorithms, defined as the required number of projections to achieve an $\eps$-suboptimal solution (i.e., $\EB\left[F(\hat{\y}) - F(\xc)\right] \le \eps$).
		$L$ is the smoothness modulus, $\mu$ is the modulus for strong convexity.
		$\kappa = \frac{L}{\mu}$ is the condition number and $\Delta^2 = \EB\|\x_0 - \xc\|^2$.
		$E (E \ge 1)$ is the maximum interval between consecutive projections. 
		When $E=1$, we perform a projection each iteration.
		$\sigma^2, \sigma_{\A^{\perp},*}^2$ and $\sigma_{\A,*}^2$ are gradients variances.
		$m$ is the number of inner loops for variance reduced methods and $N$ is the total training size.
		$\widetilde{\OM}$ hides logarithmic dependence.
	}
	\label{table:1}
\end{table}%

\begin{table}[t]
	\newcommand{\tabincell}[2]{\begin{tabular}{@{}#1@{}}#2\end{tabular}}
	\centering
		\resizebox{\textwidth}{28mm}{
	\begin{tabular}{|c|c|c|}
		\hline
		Methods & Generally Convex $(\min_{k=1}^n \mu_k =0)$& Strongly Convex$(\min_{k=1}^n \mu_k >0)$ \\
		\hline
		D-SGD~\cite{zinkevich2010parallelized,stich2019unified} &  $\OM\left(  \frac{L}{\eps}  + \frac{\sigma^2}{n\eps^2}  \right)\cdot B^2$    &   $\widetilde{\OM}\left( \kappa  + \frac{\sigma^2 }{\mu n \eps} \right)$    \\
		\hline
		Local SGD~\cite{koloskova2020unified} & $ \OM\left(\frac{L}{\eps} + \frac{\sigma^2}{nE\eps^2} + \frac{\sqrt{L} \left( \zeta_*   +  \sigma_*/\sqrt{E} \right) }{\eps^{1.5}}\right)\cdot B^2$  
		&  $\widetilde{\OM}\left( \kappa + \frac{\sigma_*^2 }{\mu n E \eps} + \frac{\sqrt{L}\left( \zeta_*   + \sigma_* /\sqrt{E} \right) }{\mu\eps^{0.5}} \right)$  \\
		Corollary~\ref{cor:local_sgd} 
		& $\OM\left( \frac{L}{\eps}  + \frac{\sigma^2}{nE\eps^2} +  \frac{\sqrt{(E-1)L\left(  E \zeta_*^2 + \frac{n-1}{n} \sigma_*^2 \right)}}{E\eps^{1.5}} \right) \cdot B^2$ 
		& $\widetilde{\OM}\left( \kappa  + \frac{\sigma_*^2 }{\mu n E \eps} + \frac{\sqrt{(E-1)L\left(  E \zeta_*^2 + \frac{n-1}{n} \sigma_*^2 \right)}}{\mu E \eps^{0.5}}\right)$      \\  
		\hline
		SCAFFOLD~\cite{karimireddy2019scaffold} 
		& $\widetilde{\OM}\left(  \frac{LB^2}{\eps} + F+\frac{\sigma^2B^2}{nE\eps^2}  \right)$
		&  $\OM\left( \kappa  + \frac{\sigma^2}{\mu n E \eps} \right)$ \\
		\hline
		\tabincell{c}{	Local SVRG\\
		Corollary~\ref{cor:local_SVRG}}
	&  $ \OM\left( \frac{LB^2 }{\eps} + \frac{m F}{E \eps }\right)$   
	&  $ \widetilde{\OM}\left( \max\left\{ \kappa, \frac{m}{E} \right\} \right)$      \\
		\hline
		\tabincell{c}{	Local ASVRG\\
			Corollary~\ref{cor:local_ASVRG}}&      
		$	\widetilde{\OM}\left( \frac{m}{E}\frac{\sqrt{L}B}{\sqrt{\eps}}  + \left(1-\frac{1}{E^2}\right)^{\frac{1}{3}}\left(\frac{m}{E}\right)^{\frac{1}{3}} \frac{ L^\frac{2}{3}B^\frac{4}{3}}{\eps^{\frac{2}{3}}} \right)$
		&   $	\widetilde{\OM}\left( \max\left\{ \frac{m}{E}, \sqrt{\frac{m \kappa}{E}} \right\} + \kappa^{\frac{2}{3}}\left(1-\frac{1}{E^2}\right)^{\frac{1}{3}}\left(\frac{m}{E}\right)^{\frac{1}{3}} \right)$   \\
		\hline
	\end{tabular}}
	\caption{
		Communication complexity of different algorithms, defined as the required number of communications to achieve an $\eps$-suboptimal solution (i.e., $\EB\left[F(\hat{\y}) - F(\xc)\right] \le \eps$ where $F(\cdot) = \frac{1}{n}\sum_{k=1}^n f_k(\cdot)$).
		$L$ is the maximum smoothness modulus among all $f_k$'s and $\mu$ is the minimum modulus for strong convexity among all $f_k$'s, and $\kappa = \frac{L}{\mu}$ is the condition number.
		$B^2 =  \EB\|\x_0 - \xc\|^2$  and $F=\EB\left[F(\x_0) - F(\xc)\right]$.
		$E (E \ge 1)$ is the maximum interval between consecutive communications.
		When $E=1$, we communicate at each iteration.
		$\sigma^2$ is the gradient variance, $\sigma_*^2$ is that at the constrained optima, and $\zeta_*^2$ characterizes the degree of data heterogeneity (see Lemma~\ref{lem:sigma} for a specific definition).
	}
	\label{table:2}
\end{table}%

In this paper we propose the delayed projection technique and analyze methods using it for LCPs.
From a high-level idea, the delayed projection technique aims to lower the frequency of projection in order to improve the projection efficiency.
\begin{itemize}
	\item In particular, we generalize Local SGD and accordingly propose delayed projected SGD (DP-SGD) that performs a projection after a constant number of steps of SGD.
	See Algorithm~\ref{alg:multi} for more details.
	From derived theories, we find that delayed projection helps reduce the statistical error brought by stochastic gradients but makes DP-SGD suffer from an additional error, termed as a residual error, which slows down the convergence rate.
	\item Viewing the residual error as an another form of variance, we are motivated to eliminate it by variance reduction techniques and thus propose delayed projected SVRG (DP-SVRG) shown in Algorithm~\ref{alg:multi_SVRG}.
	The major difference between P-SVRG~\cite{xiao2014proximal}, a famous algorithm, and our DP-SVRG  is that the former calls projections right after each inner loop to ensure feasible gradients (i.e., in $\mathcal{R}(A^{\perp})$), while the latter performs amortized projection that is called after several inner loops.	
	DP-SVRG is successful in eliminating the statistical error and residual error.
	As a result, it converges much faster than DP-SGD (see Table~\ref{table:3}).
	However, we note that P-SVRG and DP-SVRG have the same projection complexity in the large $\kappa$ regime, making us wonder whether delayed projection can be useful in diminishing variance settings.
	\item Hence, we investigate the fastest convergence rate that methods with delayed projections could achieve.
	We propose and analyze accelerated variants of DP-SVRG (see Algorithm~\ref{alg:multi_acc_SVRG}).
	The delayed projected accelerated SVRG (DP-ASVRG) incorporates two acceleration techniques: one is Nesterov’s acceleration~\cite{nesterov2013introductory}, and the other is  variance reduction for the stochastic gradient~\cite{xiao2014proximal}.		
	We find that DP-ASVRG converges more quickly and efficiently than DP-SVRG and has advantages over its non-delayed-projected counterparts like P-ASVRG~\cite{nitanda2014stochastic,shang2018asvrg} in terms of the required number for projection.   
	In particular, in the case of $N$ finite sum minimization, when the sample size $N$ is larger than the condition number $\kappa$ and hyperparameters are well set, DP-ASVRG only needs $\widetilde{\OM}\left(\kappa\right)$ projections to obtain an $\eps$-suboptimal solution, while P-ASVRG needs $\widetilde{\OM}\left( N\right)$ projections, though the two algorithms have the same gradient complexity.
	This result implies that it is possible and provable to solve LCPs using projections less than the total iterations.
	In addition, the delayed projection method combined with variance reduction techniques can benefit from delayed projection.
	\item When the number of inner loops $m$ and the projection interval $E$ are well set, DP-SVRG and DP-ASVRG are reduced to previously known algorithms, like P-SVRG, P-ASVRG, and Nesterov Accelerated Gradient (NAG).
	Our analysis is so flexible that it provides convergence results for them in a unified way.
	In particular, we decompose the variable $\x_t$ into the sum of two orthogonal iterates $\y_t = \PAT(\x_t)$ and $\z_t = \x_t - \y_t$, where $\PA$ is the projection onto the column space of $\A$.
	Then, we analyze each iteration incrementally and recur the error vector to obtain convergence results.
	We illustrate the analysis procedure in Section~\ref{sec:convergence}.
	Moreover, we give  analysis in both strongly convex and generally convex cases to further characterize their convergence behaviors.
\end{itemize}

The proposed methods can be applied to solve any instance of LCPs.
For example, as shown in the last section, we can reduce distributed optimization into an instance of LCP.
Therefore, it is handy to parallelize delayed projection methods to distributed optimization algorithms as their counterparts and derive theories for them.
In particular, we generalize DP-SVRG and DP-ASVRG to Local SVRG and Local Accelerated SVRG, respectively (see Algorithm~\ref{alg:local_SVRG} and~\ref{alg:local_acc_SVRG}).
In this way, we obtain a better convergence rate than previous work (see Table~\ref{table:2} for a brief comparison and more details in Section~\ref{sec:FL}).
For example, both using variance reduction techniques, Local SVRG is able to eliminate both the statistical error and residual error, while the previous SCAFFOLD~\cite{karimireddy2019scaffold} fails to remove the statistical error.

\section{The Problem Setup and Notation}
\label{sec:notation}
In this paper, we focus on the following affine constrained stochastic optimization problem
\begin{equation}
\label{eq:formal_loss}
\min_{\x \in \RB^p} F(\x) := \EB_{\xi \sim \DM} F(\x; \xi) \quad \mathrm{s.t.} \quad
\A^\top \x = \0,
\end{equation}
where $\A$ is a general matrix whenever $\A^\top \x = \0$ has non-trivial solutions and $\xi$ (possibly lying in a high dimensional space) is generated according to $\DM$.
The constraint can be also reformulated as $\x \in \mathcal{R}(\A^{\perp})$, where $\mathcal{R}(\A)$ denotes the space spanned by the columns of $\A$ and $\mathcal{R}(\A^{\perp})$ denotes the orthogonal complement space of $\mathcal{R}(\A)$.
One may be more interested in the case the constraint is $\A^\top \x = \bb$ for a general $\bb$. 
We argue this is a special case of~\eqref{eq:formal_loss}.
Indeed, we could always first solve a feasible solution $\y^*$ of the linear system $\A^\top \y^* = \bb$ and then replace $F(\x)$ with a new function $\widetilde{F}(\x):= F(\x + \y^*)$.
By change of variables, we still arrive at~Eq.\eqref{eq:formal_loss}.

The main reason for using $\mathcal{R}(\A^{\perp})$ as the domain rather than  $\mathcal{R}(\A^{\perp}) + \y^* := \{ \z + \y^*:  \z \in \mathcal{R}(\A^{\perp})  \}=\{ \z : \A^\top \z = \bb \}$ is the nice property inherent in the projection into $\mathcal{R}(\A^{\perp})$, because $\mathcal{R}(\A^{\perp})$ is a linear space, while $\mathcal{R}(\A^{\perp}) + \y^*$ is not.
The projection into a linear space has a lot of nice properties, including linearity, non-expansiveness,  and orthogonality.
Such properties are crucial for deriving convergence theories.

\begin{prop}
	\label{prop:proj}
	Let $\PA$ be the projection onto $\mathcal{R}(\A)$ and $\PAT$ the projection onto $\mathcal{R}(\A^\perp)$. Then
	\begin{enumerate}
		\item Linearity: $\PA(\alpha\x + \beta\y) = \alpha\PA(\x) + \beta\PA(\y)$ for any $\x, \y \in \RB^p$ and $\alpha, \beta \in \RB$;
		\item Non-expansiveness: $\max\{\|\PA(\x)-\PA(\y)\|, \|\PAT(\x)-\PAT(\y)\|  \} \le \|\x-\y\|$ for any $\x, \y \in \RB^p$;
		\item Orthogonality: any $\x \in \RB^p$ can be decomposed uniquely into $\x = \u + \v$ where $\u = \PA(\x)$ and $\v = \PAT(\x)$ satisfying $\langle \u, \v \rangle = 0$.
	\end{enumerate}
\end{prop}

We assume that $F(\cdot)$ is well behaved, namely smoothness and (strong) convexity.
Such an assumption is quite common in the machine learning literature~\cite{nesterov2013introductory,bottou2018optimization,gower2020variance,linaccelerated}.
For example, we consider the finite sum minimization problem where $F(\x)$ has a finite sum structure~\cite{mohri2018foundations}.
In particular, $F(\x) = \frac{1}{N}\sum_{i=1}^N \ell(\x, \varsigma_i) = \EB_{\xi} \ell(\x, \xi)$ where $N$ is the number of total samples, and $\xi$ denotes the uniform distribution on the collected training samples $\{\varsigma_i\}_{i=1}^N$ (that is also the empirical distribution of $\DM$, according to which $\{\varsigma_i\}_{i=1}^N$ is i.i.d. sampled).
Then Assumption~\ref{asmp:smooth} requires $F(\x; \xi) = \ell(\x, \xi)$ to be smooth in $\x$ for any collected samples, while Assumption~\ref{asmp:strong} requires the finite sum function $F(\x) = \frac{1}{N}\sum_{i=1}^N \ell(\x, \varsigma_i)$ to be (strongly) convex in $\x$.

\begin{asmp}[Smoothness]
	\label{asmp:smooth}
	For~\eqref{eq:constraint}, we that assume $F(\x)$ is $L$-smooth convex, i.e.,
	\[
	\| \nabla F(\y; 
	\xi) - \nabla F(\x; \xi)  \| \le L \| \x - \y\|,
	\;     \forall \x, \y \in \RB^p \ \text{and} 
	\ \xi.
	\]
\end{asmp}

\begin{asmp}[Convexity]
	\label{asmp:strong}
	For~\eqref{eq:constraint}, we assume that $F(\x)$ is $\mu$-strongly convex ($\mu \ge 0$), i.e.,
	\[
	F(\y) - F(\x) \ge \langle\nabla F(\x), \y-\x \rangle + \frac{\mu}{2} \| \x-\y\|^2,
	\;      \forall \x, \y \in \RB^p.
	\]
\end{asmp}

\begin{cor}
	\label{cor:xc}
	Let Assumption~\ref{asmp:strong} hold with $\mu > 0$. Then the solution of~\eqref{eq:formal_loss} is unique (denoted $\xc$). Moreover, we have $\PAT (\nabla F(\xc))=\0$.
\end{cor}

We focus on solving~\eqref{eq:formal_loss} from the primal perspective where the constraint is enforced in a separate step by projecting onto the constraint space.
The most natural approach is the projected (stochastic) gradient-based algorithm that alternates between one step of (stochastic) gradient descent and one step of projection.
~\cite{nesterov2013introductory} used the technique of gradient mapping to analyze such an algorithm.
However, projection is often expensive and time-consuming to manipulate.
For example, as we discussed in the introduction, as the counterpart of projection in distributed optimization, communication is often the bottleneck of distributed optimization and should be reduced as much as possible.
In this paper, we would reduce the number of projections rather than reduce computation.
Motivated by Local SGD~\cite{lin2018don,stich2018local,kairouz2019advances,koloskova2020unified,woodworth2020local,woodworth2020minibatch}, which is a distributed algorithm that alternates between multiple steps of SGD and one step of communication, we try to reduce the frequency of projection when solving~\eqref{eq:formal_loss} and propose Delayed Projected SGD (see Algorithm~\ref{alg:multi}).

\section{Delayed Projected SGD for Linearly Constrained Problems}
We first analyze the novel algorithm named Delayed Projected SGD (DP-SGD) for LCPs~\eqref{eq:formal_loss}.
We will also show how to derive convergence results for methods using delayed projection techniques in a unified way, the procedure going through all our analysis.
This theoretical framework makes theoretical analysis easier and cleaner and sheds light on the design of new algorithms.

\subsection{The Algorithm}
DP-SGD (Algorithm~\ref{alg:multi}) shares the same philosophy of Local SGD, which is to reduce the frequency of projection (or, equivalently, communication, in the context of distributed optimization).
Let $T$ be the number of total iterations.
Let $\IM_T$, a subset of $[T]:=\{1, \cdots, T\}$, index the iterations that perform a projection and the cardinality $|\IM_T |$ denotes the total number of projections.
If $\IM_T = [T]$, projection happens at every iteration, and DP-SGD degenerates  into projected SGD.
To quantize the frequency of projection, we define the gap of $\IM_T$ as the largest interval between two consecutive elements in $\IM_T$ when we sort all elements in a decrease order and denotes it by $\mathrm{gap}(\IM_T)$.
Hence, $\mathrm{gap}(\IM_T) = 1$ characterizes the case of $\IM_T = [T]$ .

\begin{algorithm}[tb]
	\caption{Delayed Projected SGD (DP-SGD)}
	\label{alg:multi}
	\begin{algorithmic}
		\STATE {\bfseries Input:} function $F$, initial point $\x_0$, step size $\eta_t$, projection set $\IM_T$ with $\mathrm{gap}(\IM_T) = E  (E \ge 1)$.
		\FOR{$t=1$ {\bfseries to} $T$}
		\STATE  $\x_t = \x_{t-1} - \eta_{t-1} \nabla F(\x_{t-1}; \xi_{t-1})$   \quad \# $\xi_{t-1}$ is the selected random sample
		\IF {$t \in \IM_T$}
		\STATE {$\x_t \gets \PAT(\x_t)$ \quad \# project $\x_t$ into $\mathcal{R}(A^{\perp})$ and make it satisfy the linear constraints.}
		\ENDIF
		\ENDFOR
		\STATE $\hat{\y} \gets  \PAT(\frac{1}{W_T}\sum_{j=0}^{T-1} (1-\mu \eta)^{T-j-1} \x_j)$ where $W_T = \sum_{j=0}^{T-1}(1-\mu \eta)^{T-j-1}$.
		\STATE {\bfseries Return: $\hat{\y}$.} 
	\end{algorithmic}
\end{algorithm}

\subsection{Convergence Analysis}
\label{sec:convergence}

In this section, we provide convergence guarantees for Algorithm~\ref{alg:multi} that produces the optimum parameter for affine constrained optimization problems.
We give an outline of our analysis for a glance.
We first decompose the iterate $\x_t$ into two components $\x_t = \y_t + \z_t$ where $\y_t = \PAT(\x_t)$ and $\z_t = \PA(\x_t)$.
From Proposition~\ref{prop:proj}, $\y_t$ is orthogonal to $\z_t$.
Then we derive one-step descent analysis for the separate two iterates $\EB\|\y_t-\xc\|^2$ and $\EB \|\z_t\|^2$.
We concatenate them together and denote by $L_t = (\EB\|\y_t-\xc\|^2, \EB \|\z_t\|^2)^\top$.
We formulate an error propagation that depicts how $L_{t+1}$ evolves with $L_t$ and other factors (like gradients variance and residual terms), by which and using a standard recursion argument, we give convergence analysis for DP-SGD.
This analysis procedure goes through all of our analyses.

\begin{asmp}[Bounded variance at the optimum]
	\label{assum:bounded}
	Let $\xc = \argmin_{\x \in \mathcal{R}(A^\perp)}F(\x)$, and then define $\sigmaat = \EB\| \PAT (\nabla F(\xc; \xi) )\|^2$ and $\sigmaa= \EB\|\PA(\nabla F(\xc;\xi))\|^2$.
\end{asmp}
\begin{rem}
	\label{rem:sigma}
	$\sigmaat$ can be rewritten as $\sigmaat = \EB\| \PAT (\nabla F(\xc; \xi) - \nabla F(\xc))\|^2$, which measures the stochastic gradient variance of $F(\cdot; \xi)$ within the space $\mathcal{R}(A^\perp)$ at the constrained optimum $\xc$.
\end{rem}

\begin{lem}
	\label{lem:y_descent}
	Under Assumptions~\ref{asmp:smooth},~\ref{asmp:strong} and~\ref{assum:bounded},
	letting $\y_t = \PAT (\x_t)$ be the projection onto $\mathcal{R}(A^\perp)$ and $\z_t = \PA(\x_t)$ the projection onto $\mathcal{R}(A)$, and if $\eta_t \le \frac{1}{10L}$, we have
	\begin{equation}
	\label{eq:y_descent}
	\EB\|\y_{t+1} - \xc\|^2
	\le (1-\mu \eta_t)  \EB \| \y_t-\xc\|^2 -\eta_t \EB\left[ F(\y_t) - F(\xc)\right]  + 3  \eta_t^2 \sigmaat + 2L\eta_t\EB \| \z_t\|^2.
	\end{equation}
\end{lem}

In Lemma~\ref{lem:y_descent}, we show that $\y_{t+1}$, the projection on $\mathcal{R}(\A^{\perp})$ of one step of SGD from $\x_t$, behaves similarly to traditional SGD.
This makes sense; since as long as $\nabla F(\x_t)$ has a non-trivial component on $\mathcal{R}(\A^{\perp})$, $\y_t$ could make use of that to move further towards $\xc$.
However, due to the decayed projection, typically $\nabla F(\x_t)$ also has a non-trivial component on $\mathcal{R}(\A)$ that pushes $\x_t$ to go beyond $\mathcal{R}(\A^{\perp})$.
Hence, one step descent of $\y_t$ also depends on the value of $\z_t$, the projection of $\x_t$ on $\mathcal{R}(\A)$.
The larger $\EB \|\z_t\|^2$, the smaller decent $\y_t$ would make, because $\EB \|\z_t\|^2$ measures the difference between $\nabla F(\x_t)$ and $\nabla F(\y_t)$ (due to the smoothness assumption).
When $\EB \|\z_t\|^2$ vanishes,~Eq.\eqref{eq:y_descent} recovers the result of one-step descent of projected gradient descent.
Next, we are going to bound $\EB \|\z_t\|^2$ that is typically non-zero. 
To warm up, we first make an idealized assumption that $\nabla F(\x) \in \mathcal{R}(\A^\perp)$ uniformly over $\x \in \mathcal{R}(\A^\perp)$ and remove this assumption to derive similar bounds on $\EB \|\z_t\|^2$.

\begin{asmp}[Almost unconstrained gradients]
	\label{asmp:g_unconstrained}
	Assume that the expected gradient on any element of $\mathcal{R}(\A^{\perp})$ is always in $\mathcal{R}(\A^{\perp})$. Thus we have $\PA(\nabla F(\x)) = \0$ for all $\x \in \mathcal{R}(\A^{\perp})$.
\end{asmp}

Assumption~\ref{asmp:g_unconstrained} means that gradients descent is closed under the constraint space in expectation; once the algorithm reaches a feasible point $\x_{t_0}$ in $\mathcal{R}(\A^{\perp})$, the next iterate $\x_{t_0+1}$ produced by one step of expected gradient descent is not going to violate the constraint, implying $\x_{t_0+1} = \x_{t_0} - \eta_{t_0}\nabla F(\x_{t_0}) \in \mathcal{R}(\A^{\perp})$ still holds.
Then we easily find that all the following iterates $\{\x_{t}\}_{t \ge t_0}$ are feasible.
This gives an illusion that the linear-equation constraint disappears, where the so-called almost-unconstrained-gradient name comes from.
However, such an ideal case is not practical, because we often make use of stochastic gradients rather than expected gradients.
Randomness inherent in stochastic gradients is going to provoke violation of linear constraints. 
Assumption~\ref{asmp:g_unconstrained} simplifies the situation: such deviation is purely caused by randomness.
Once the almost-unconstrained-gradient assumption fails, an additional factor $\PA(\nabla F(\x))$ will also affect the generation of iterates; it will complicate the situation and deteriorate the convergence.

\begin{rem}
	If both Assumptions~\ref{asmp:strong} and~\ref{asmp:g_unconstrained} hold, we know that $\nabla F(\xc) = \0$ because $\nabla F(\xc) = \PA(\nabla F(\xc)) +  \PAT(\nabla F(\xc)) =  \0$.
	It means $\argmin_{\x \in \RB^p} F(\x) = \argmin_{\x \in \mathcal{R}(A^\perp)} F(\x)$.
\end{rem}

\begin{lem}
	\label{lem:z_descent}
Under Assumptions~\ref{asmp:smooth},~\ref{asmp:strong},~\ref{assum:bounded} and~\ref{asmp:g_unconstrained}, and if $\eta_t \le \frac{1}{2L}$, we have
	\begin{equation}
	\label{eq:z_descent_0}
	\EB \|\z_{t+1}\|^2 \le 	(1 {-}\mu \eta_t) \EB \|\z_{t}\|^2  +  2\eta_t\EB\left[  F(\y_t)  - F(\xc) \right] + 2\eta_t^2 \sigmaa. 
	\end{equation}
	Without Assumptions~\ref{asmp:g_unconstrained},  if $\eta_t \le \frac{1}{L(2E+3)}$ and $\mathrm{gap}(\IM_T) \le E$, then~\eqref{eq:z_descent_0} becomes
	\begin{equation*}
	\label{eq:z_descent_1}
	\EB \|\z_{t+1}\|^2	\le\left(1 {-} \mu \eta_t+\frac{1-\mu\eta_{t}}{E}+L\eta_t\right) \EB \|\z_{t}\|^2  + 2\eta_t\EB\left[ F(\y_t) - F(\xc) \right] +  3\eta_t^2  \left[\sigmaa  + E\|\nabla F(\xc)\|^2 \right].
	\end{equation*}
\end{lem}

Lemma~\ref{lem:z_descent} shows that $\EB\|\z_{t+1}\|^2$ behaves differently from $\EB\|\y_{t+1}-\xc\|^2$.
Essentially, we expect $\EB\|\y_{t+1}-\xc\|^2$ decreases with $t$ because the negative term $-\eta_t \EB\left[ F(\y_t) - F(\xc)\right]$ always decays the right hand side of~Eq.\eqref{eq:y_descent}.
Instead, $\eta_t \EB\left[ F(\y_t) - F(\xc)\right]$ that appears on the right hand side of~Eq.\eqref{eq:z_descent_0} is positive and looses the bound, indicating $\EB \|\z_{t+1}\|^2$ is gradually increasing in $t$.
Latter on, we will illustrate this with a simple example.
Assumption~\ref{asmp:g_unconstrained} affects the relative magnitudes of $\EB \|\z_{t+1}\|^2$ with respect to $\EB \|\z_{t}\|^2$.
When the gradient is almost unconstrained, $\EB \|\z_{t+1}\|^2$ shrinks $\EB \|\z_{t}\|^2$ by a factor of $1-\mu\eta_t$ and then suffers an additive error $\OM(\eta_t \EB\left[ F(\y_t) - F(\xc)\right] + \eta_t^2 \sigmaa )$.
By contrast, when the almost-unconstrained-gradient assumption disappears, $\EB \|\z_{t+1}\|^2$ enlarges $\EB \|\z_{t}\|^2$ by a factor of $1+\frac{1}{E}-\mu\eta_t +3L\eta_t$ and then suffers a larger additive error $\OM(\eta_t \EB\left[ F(\y_t) - F(\xc)\right] + \eta_t^2 \left[\sigmaa + \|\nabla F(\xc)\|^2\right] )$.
Indeed, when $\nabla F(\xc) = \PA(\nabla F(\xc)) \neq \0$, $\x_t$ always moves away from $\xc$ because $\PA(\nabla F(\x_t)) \neq \0$ even if $\x_t$ is quite close to $\xc$, implying $\EB\|\z_t\|^2$ accumulates faster than before.
Fortunately, $\EB \|\z_{t}\|^2$ is set as zero periodically at an interval no larger than $E$, so such exponential enlargement will not last for a long time.

Let $L_t = (\EB\|\y_t-\xc\|^2, \EB \|\z_t\|^2)^\top \in \RB^2$.
Lemma~\ref{lem:y_descent} together with Lemma~\ref{lem:z_descent} depict how $L_{t+1}$ evolves with $L_t$.
It is in  form of $L_{t+1} \le A L_{t} -\eta \delta_t\bb + \eta^2 \c$ where $A \in \RB^{2 \times 2}$, $\bb, \c \in \RB^2$ are some problem-dependent factors and the inequality holds element-by-element.
Using a standard recursion argument (see Lemma~\ref{lem:error_prop} in  Appendix), we can give a convergence analysis for DP-SGD.
We find that the almost-unconstrained-gradient assumption does not affect the convergence and only increases the residual error by an additional term $\|\nabla F(\xc)\|^2$.

\begin{thm}[Simple case]
	\label{thm:simple}
	Suppose that Assumptions~\ref{asmp:smooth},~\ref{asmp:strong},~\ref{assum:bounded} and~\ref{asmp:g_unconstrained} hold, and consider a constant learning rate such that $\eta_t = \eta \le \min\{\frac{1}{10L}, \frac{1}{\mu + 8L(E-1)}\}$.
	Let $\Delta^2 = \EB\|\PAT(\x_{0}) -\xc\|^2$. Then it follows that 
	\begin{equation}
	\label{eq:y_error}
	\EB \left[ F(\hat{\y})-F(\xc)\right]
	\le   
	\underbrace{	\OM\left(\min\left\{\frac{1}{T}, (1-\mu\eta)^T \right\}  \frac{\Delta^2}{\eta}\right)}_{\text{optimization error} }  + 
	\underbrace{\vphantom{ \left(\frac{a^{0.3}}{b}\right) } \OM\left( \eta \sigmaat \right)}_{\text{statistic error}} + 
	\underbrace{\vphantom{ \left(\frac{a^{0.3}}{b}\right) }  \OM\left((E-1) L \eta^2\sigmaa\right)}_{\text{residual error}}. 
	\end{equation}
	Choosing an appropriate step size $\eta$,  we have that for convex case $(\mu = 0)$, 
	\begin{equation}
		\label{eq:y_mu0}
	\EB \left[F(\hat{\y})-F(\xc)\right] = \OM\left( \frac{LE\Delta^2}{T} +\frac{\Delta}{\sqrt{T}} \cdot \sigma_{\A^{\perp},*} + \frac{\sqrt[3]{(E-1)L\Delta^4}}{T^{\frac{2}{3}}} \cdot \sigma_{\A,*}^{\frac{2}{3}}\right) , 
	\end{equation}
	and that for strongly convex case $(\mu > 0)$,
	\begin{equation}
	\label{eq:y_mu>}
	\EB \left[F(\hat{\y})-F(\xc)\right] = \widetilde{\OM}\left( LE\Delta^2 \cdot\exp\left(-\Theta\left(\frac{ T}{\kappa E}\right) \right) + \frac{\sigmaat}{\mu T}  + \frac{(E-1)L\cdot\sigmaa}{\mu^2 T^2}  \right),
	\end{equation}
	where $\kappa = \frac{L}{\mu}$ is the condition number.
\end{thm}

\begin{thm}[Complicated case]
	\label{thm:complicate}
	Under the same conditions of Theorem~\ref{thm:simple} but without Assumption ~\ref{asmp:g_unconstrained}, setting a constant learning rate $\eta_t = \eta \le \min\{\frac{1}{L(E+9)}, \frac{1}{\mu + 25L(E-1)}\}$, Algorithm~\ref{alg:multi} has a similar error decomposition~\eqref{eq:y_error} except that $\sigmaa$ is replaced by 
	\begin{equation}
	\label{eq:new_sigmaa}
		\widetilde{\sigma}_{\A,*}^2 = \sigmaa + (E-1)\|\nabla F(\xc)\|^2.
	\end{equation}
	With a similar choice of learning rate in Theorem~\ref{thm:simple}, the bounds~\eqref{eq:y_mu0} and~\eqref{eq:y_mu>} still hold by replacing $\sigmaa$ with $\widetilde{\sigma}_{\A,*}^2$.
\end{thm}

\section{Removing Residual Errors via Variance Reduction}
The residual error, though with a positive dependence on the projection interval $E$, still forms like a variance.
To remove the dependence, we are motivated to use variance reduction methods~\cite{johnson2013accelerating,allen2016improved,gower2020variance}.

\subsection{Delayed Projected SVRG}
Delayed projected SVRG (DP-SVRG), shown in Algorithm~\ref{alg:multi_SVRG}, is divided into $S$ epochs, each consisting of $m$ inner iterations.
Like Algorithm~\ref{alg:multi}, we call a projection only when $t+1 \in \IM_m$ where $\IM_m$ is the projection set with $\mathrm{gap}(\IM_m) = E$.
Typically, we can use $\IM_m^0 = \{ 0, E, 2E, \cdots \} \cap [m]$ that means we call a projection after every $E (E \ge 1)$ inner iterations are finished.

There are several important features that should be highlighted.
First, the stochastic gradient makes use of control variate that is known as the main ingredient for variance reduction.
The gradient $\g_t^s = \nabla F(\x_{t}^s; \xi_{t}^s) - \nabla F(\ttx_s; \xi_{t}^s) + \PAT(\nabla F(\ttx_{s}))$ consists of two parts: (i) the stochastic part $\nabla F(\x_{t}^s; \xi_{t}^s) - \nabla F(\ttx_s; \xi_{t}^s)$ is projection-free, and the randomness mainly comes from a randomly generated sample\footnote{Here we don't consider the minibatch setting  for simplicity where multiple samples are used to form stochastic gradients.
Besides, it is quite easy to extend our result to that setting.}; and (ii) the deterministic part $\PAT(\nabla F(\ttx_{s}))$ is evaluated at $\ttx_{s}$ at the beginning of an epoch, which is the counterpart of the full gradient if we consider finite-sum minimization.

Second,  the gradient $\g_t^s$ is not an unbiased estimator for $\PAT(\nabla F(\x_t^s))$ and even may not lie within $\mathcal{R}(\A^{\perp})$.
It implies in expectation the updated iterate $\x_t^s - \eta_{t}^s \g_t^s$ may violate the affine constraint.
However, the algorithm has two mechanisms to ensure convergence even with biased inner updates.
The most obvious one is we force the feasibility by delayed projections and repeated restarts.
For one thing, we call a projection at an interval no more than $E$ iterations to remove the infeasible part of $\x_{t+1}^s$ (i.e., $\PA(\x_{t+1}^s)$).
For another thing, we set the starting vector $\x_0^{s+1}$ as the projected ending vector of the previous stage $\PAT(\x_m^s)$.
The second one is implicit in our theory; that is we set a sufficiently small step size, typically $\eta_t^s = \Theta\left(\frac{1}{LE}\right)$.
In this way, the effect of multiple inner loops between two consecutive projections is similar to one  step feasible update with a larger step size, which is very important to convergence.

Third, the snapshot $\ttx_{s+1}$ is a projected weighted average of $\{\x_t^s\}_{t=0}^{E-1}$ in the most recent stage.
The projection ensures $\ttx_{s+1}$ is feasible (i.e., in $\mathcal{R}(\A^{\perp})$) and is invoked after the weighted average is computed.
When $\mu > 0$, the weight decreases geometrically with $t$, implying more recent iterate has a larger weight and thus is much more important.
When $\mu = 0$, the geometrically weighted average is reduced to a simple average, the latter having been used by many previous algorithms and shown to work well in practice~\cite{johnson2013accelerating,xiao2014proximal,allen2016improved,shang2018vr}.

Finally, the output $\hat{\y}$ is a weighted average of all snapshots $\{\ttx_{s}\}_{s=1}^S$.
Noting the structure of snapshot points, we have $\hat{\y} =  \PAT(\frac{1}{W_T}\sum_{j=0}^{T-1} (1-\mu \eta)^{T-j-1} \x_j)$ where $W_T = \sum_{j=0}^{T-1}(1-\mu \eta)^{T-j-1}$.

\begin{algorithm}[tb]
	\caption{Delayed Projected SVRG (DP-SVRG)}
	\label{alg:multi_SVRG}
	\begin{algorithmic}
		\STATE {\bfseries Input:} function $F$, initial point $\x_0$ (and let $\ttx_0 = \x_0^0= \PAT(\x_0)$), step size $\eta_t^s$, stage number $S$, loop iteration $m$,
		projection set $\IM_m \subset [m]$ with $\mathrm{gap}(\IM_m) = E  (E \ge 1)$.
		\FOR{$s=0$ {\bfseries to} $S-1$}
		\STATE {$\tth_{s} \gets \PAT(\nabla F(\ttx_{s}))$ }
		\FOR{$t=0$ {\bfseries to} $m-1$}
		\STATE  $\g_{t}^s \gets \nabla F(\x_{t}^s; \xi_{t}^s) - \nabla F(\ttx_s; \xi_{t}^s) + \tth_s$  with  $\xi_{t}^s$ sampled independently 
		\STATE  $\x_{t+1}^s \gets \x_{t}^s - \eta_{t}^s\g_t^s $  
		\IF {$(t+1)\in \IM_m$}
		\STATE {$\x_{t+1}^s \gets \PAT(\x_{t+1}^s)$ \quad \# project $\x_{t+1}^s$ into $\mathcal{R}(A^{\perp})$ and make it satisfy the linear constraints.}
		\ENDIF
		\ENDFOR
		\STATE {$\x_0^{s+1} \gets \PAT(\x_m^s), \ttx_{s+1} \gets \PAT(\sum_{i=0}^{m-1}(1-\mu\eta)^{i}\x_{m-i-1}^s/\sum_{j=0}^{m-1}(1-\mu\eta)^{j})$}
		\ENDFOR
		\STATE If $\mu = 0$, $\hat{\y} \gets \frac{1}{S}\sum_{s=1}^S \ttx_{s}$; otherwise $\hat{\y} \gets \ttx_{S}$.
		\STATE {\bfseries Return:} $\hat{\y}$.
	\end{algorithmic}
\end{algorithm}

\subsection{Convergence Analysis}
\begin{thm}
	\label{thm:svrg}
	Assume Assumption~\ref{asmp:smooth} and~\ref{asmp:strong} hold.
	Let $\y_0 = \PAT(\x_0), \Delta^2 =  \EB \|\y_0 - \xc \|^2$ and $T=mS$.
	Run Algorithm~\ref{alg:multi_SVRG}  for $S$ stages, each stage has $m$ inner loops and performs projections at $\IM_m$ with gap $E$.
	By choosing an appropriate constant step size $\eta_t^s = \eta =\Theta(\frac{1}{LE})$, for the convex case $(\mu = 0)$,
	\begin{equation}
	\label{eq:y_mu0_svrg}
	\EB \left[F(\hat{\y})-F(\xc)\right] 
	= \OM\left(  \frac{LE\Delta^2}{T} + \frac{\EB[F(\y_0) - F(\xc)]}{S} \right).
	\end{equation}
	and for the strongly convex case $(\mu > 0)$, 
	\begin{equation}
	\label{eq:y_mu>_svrg}
	\EB \left[F(\hat{\y})-F(\xc)\right] 
\le \OM \left( \left[   L E\Delta^2 + \EB[F(\y_0) - F(\xc)] \right] \cdot 
\exp\left( -  \Theta \left(\frac{T}{\max\{ \kappa E,  m \}} \right) \right)\right)
	\end{equation}
		where $\kappa = \frac{L}{\mu}$ is the condition number.
\end{thm}

Note that $\EB[F(\y_0) - F(\xc)]= \OM(L\Delta^2)$.
By comparing the convergence bounds \eqref{eq:y_mu0_svrg} and~\eqref{eq:y_mu>_svrg} with those for DP-SGD~\eqref{eq:y_mu0} and~\eqref{eq:y_mu>}, we find the biggest difference is that both the statistical error and residual error are eliminated.
The statistical error is eliminated as expected due to the control variates we use.
Indeed, when the iterates approach the optimum, the difference between $\x_t^s$ and $\ttx_{s}$ is on the decline, implying the fluctuation caused by random samples is diminishing.
The control variates also account for the disappearance of residual errors.
When the algorithm starts to converge, the gradient $\g_t^s = \nabla F(\x_{t}^s; \xi_{t}^s) - \nabla F(\ttx_s; \xi_{t}^s) + \PAT(\nabla F(\ttx_{s}))$ would be dominated by $\PAT(\nabla F(\ttx_{s}))$ and thus how $\ttx_{s}$ converges determines the performance of the algorithm.
The fact that $\ttx_{s}$ is always feasible (i.e., in $\mathcal{R}(\A^\perp)$) and is always closer to $\xc$ than its prior iterate $\ttx_{s-1}$ (see Lemma~\ref{lem:error_svrg_stage} in the appendix) explains the whole story.

As a result, DP-SVRG achieves a convergence rate of $\OM(\frac{1}{S})$ for generally convex functions and of $\OM(\exp(-c \min \left\{   \frac{m}{\kappa E} ,1 \right\}\cdot S) )$ for strongly convex functions.
These are the same rates on $S$ achieved by gradient descent under these assumptions~\cite{nesterov2013introductory}, and are much faster than the $\OM(\frac{1}{\sqrt{mS}})$ and $\OM(\frac{1}{mS})$ rate of DP-SGD in the corresponding settings.

Finally, let us compare DP-SVRG with another competitive baseline method, Proximal SVRG~\cite{xiao2014proximal} (P-SVRG) thoroughly.
They are comparable for two reasons.
First, when we set the regularization function as $h(\x) = 1_{\x \in \mathcal{R}(\A^{\perp})} = \infty $ if $ \x \notin \mathcal{R}(\A^{\perp}) $ otherwise $=0$ for P-SVRG, it is also able to solve LCPs.
Second, P-SVRG also uses a multi-stage scheme to progressively reduce the variance of the stochastic gradient.
The biggest difference between P-SVRG and our DP-SVRG is how each projection is performed.
P-SVRG performs immediate projections; it calls a projection right after each gradient descent, so it uses feasible stochastic gradients (i.e., $\PAT(\g_t^s)$) each step.
DP-SVRG performs amortized projections or delayed projections; it performs several non-feasible stochastic gradients (i.e., $\g_t^s \notin \mathcal{R}(\A^{\perp})$ generally) and then rectifies the bias periodically.

\begin{table}[t]
	\newcommand{\tabincell}[2]{\begin{tabular}{@{}#1@{}}#2\end{tabular}}
	\centering
	\begin{tabular}{|c|c|c|}
		\hline
		Items &\tabincell{c}{	P-SVRG \\ $m=\kappa, E=1$} & DP-SVRG \\
		\hline
		\tabincell{c}{	Iteration  \\ $(\TB)$ } 
		& $\OM\left( \kappa \ln \frac{L\Delta^2}{\eps} \right)$
		& $\OM\left( \max\{ \kappa E, m \} \ln \frac{LE\Delta^2}{\eps} \right)$ \\
		\hline
		\tabincell{c}{	Stage  \\ $(\SB = {\TB}/{m})$ } 
		& $\OM\left( \ln \frac{L\Delta^2}{\eps} \right)$
		& $\OM\left( \max\{ \frac{\kappa E}{m}, 1 \} \ln \frac{LE\Delta^2}{\eps} \right)$ \\
		\hline
		\tabincell{c}{	Projection  \\  $(\PB = \SB + {\TB}/{E})$ } 
		& $\OM\left( \kappa \ln \frac{L\Delta^2}{\eps} \right)$
		& $\OM\left( \max\{ \kappa , \frac{m}{E} \} \ln \frac{LE\Delta^2}{\eps} \right)$ \\
		\hline
		\tabincell{c}{	Gradient  \\  $(\GB = \TB + N\SB)$  } 
		& $\OM\left(  (N + \kappa) \ln \frac{L\Delta^2}{\eps} \right)$
		&		\tabincell{c}{	$\OM\left( \max\left\{ \kappa E +  \frac{\kappa EN}{m}, N + m \right\} \ln \frac{LE\Delta^2}{\eps} \right)$  \\  minimum is $\OM\left(  \left(N + \kappa E\right) \ln \frac{LE\Delta^2}{\eps} \right)$ } 
		 \\
\hline
	\end{tabular}
	\caption{Compare P-SVRG and DP-SVRG with $\IM_m^0$ in four aspects under the strongly convex setting.
	P-SVRG is a special case of DP-SVRG when $m=\kappa$ and $E=1$.
	We use $\TB, \SB, \PB, \GB$ to stand for the four complexities (see Definition~\ref{def:complexity}).
	We give their relations in the brackets of the first column.
	$N$ is the number of total samples.
	}
	\label{table:3}
\end{table}%

Let $F(\x) = \frac{1}{N} \sum_{i=1}^N F(\x; \xi_i)$ be the empirical form of~\eqref{eq:formal_loss}, where $\{\xi_i\}_{i=1}^N$ is generated independently and $N$ is the training datasize.
We investigate four kinds of complexity for P-SVRG and DP-SVRG (see the following definition), and show the results in Table~\ref{table:3}.

\begin{defn}
	\label{def:complexity}
	In the process of obtaining an $\eps$-optimal solution, four complexities are taken into account to evaluate the optimization efficiency of considered algorithms, namely 
	\begin{enumerate}
		\item iteration complexity $(\TB)$: how many inner iterations are used;
		\item stage complexity $(\SB)$: how many stages are used;
		\item projection complexity $(\PB)$: how many projections are performed;
		\item gradient complexity $(\GB)$: how many stochastic gradient computations are used.
	\end{enumerate}
\end{defn}

For simplicity, we assume $\IM_m^0 = \{0, E, 2E, \cdots \} \cap [m]$ is used.
As discussed, P-SVRG is a special case of DP-SVRG when $E=1$ and $m=\kappa$.
Thus, we can derive convergence analysis for P-SVRG by letting $E=1$ and $m=\kappa$ in bounded~\eqref{eq:y_mu0_svrg} and~\eqref{eq:y_mu>_svrg}, which give the results for the second column of Table~\ref{table:3}, which is consistent with previous analysis~\cite{johnson2013accelerating}.
It is worth to mention that our result gives a unified analysis for P-SVRG under both generally convex and strongly convex cases.
Previous works do that mainly by designing new algorithms~\cite{schmidt2017minimizing,defazio2014saga,allen2016improved} or using reduction methods~\cite{xiao2014proximal,allen2016optimal}.
Besides, our analysis allows $E$ and $m$ to vary, illustrating the flexibility and expansibility of our analysis.

From Table~\ref{table:3}, if we set $m=N$ for DP-SVRG, its projection complexity becomes $\widetilde{\OM}(\max\{\kappa, N/E\})$.
When $\kappa \le N/E$, i.e., the case of big data scenario, the projection complexity decreases as $E$ increases, which can be verified by our experiments.
If we set $m=\kappa$ for DP-SVRG, then its iteration complexity, stage complexity, and gradient complexity are all $E$ times larger than those of P-SVRG, while its projection complexity remains unchanged.
It indicates that delayed projections seem not so useful for variance reduced methods, at least for P-SVRG.
Recall that delayed projection is quite useful in stochastic optimization where gradient variances are often the bottleneck of optimization.
It is unknown and quite interesting to see whether delayed projection technique combined with the variance reduced technique could better trade-off the four complexities or improve the projection complexity.
We explore the question in the next section.

\section{Accelerating Delayed Projected Methods}

\subsection{Accelerated Delayed Projected SVRG}
Accelerating stochastic gradient methods was a hot topic in literatures~\cite{mairal2013optimization,su2014differential,linaccelerated}.
For deterministic optimization, \cite{nesterov2013introductory} proposed his accelerated gradient descent (AGD) for convex optimization that respectively achieves $\OM\left(\frac{1}{\sqrt{\eps}}\right)$ and $\OM\left(\sqrt{\kappa}\ln\frac{1}{\eps}\right)$ under generally convex and strongly convex smooth problems.
For stochastic optimization, using the variance reduction method, ~\cite{allen2017katyusha} proposed the first truly accelerated stochastic algorithm, named Katyusha.
Many other papers also work on that topic~\cite{nitanda2014stochastic,allen2017katyusha,shang2018asvrg}.
To accelerate Algorithm~\ref{alg:multi_SVRG}, we borrow the acceleration technique proposed in~\cite{shang2018asvrg}, which is much simpler than other stochastic acceleration momentum~\cite{allen2017katyusha,allen2018katyusha}.
We summarize the algorithm in Algorithm~\ref{alg:multi_acc_SVRG}.
It is quite similar to P-ASVRG~\cite{shang2018asvrg} except for the following features.

First, as discussed in the description of DP-SVRG, we also use delayed projections here, which results (possibly) biased updates between two consecutive projection iterations.
As a remedy, $\mathrm{gap}(\IM_m) = E$ limits the growth rate of residual errors caused by infrequent projections.

Second, we maintain two sequences $\{ \x_t^s \}$ and $\{\u_t^s\}$, and always initialize them by $\u_0^{s+1} \gets \PAT(\u_m^s)$ and $\x_0^{s+1} \gets \ttx_{s+1}$ at the beginning of each stage to ensure constraint feasibility, no matter whether $\mu > 0$ or not.

Finally, the only momentum parameter $\theta_{s}$ in the algorithm is tuned more carefully, since we need to control the residual error.
Let $\delta = 9(E^2-1)\eta^2 L^2$ and we will choose the learning rate $\eta$ sufficiently small such that $\delta \in [0, 1)$.
Under the strongly convex case $(\mu>0)$, we set all $
\theta_{s}$ as a constant $\theta = 2\delta + \sqrt{4\delta^2 + \eta\mu m} \in (2\delta, 1+\delta)$.
Under the generally convex case $(\mu=0)$, we define the sequence $\{\theta_{s} \}$ recursively: let $\theta_{0}  = 1 - \frac{2\eta L}{1-\eta L} \in (2\delta, 1+\delta)$ and $\theta_{s+1}  = \sqrt{ \frac{1+\delta}{1-\delta}\theta_{s}^2 + \frac{\theta_{s}^4}{4(1-\delta)^2}} - \frac{\theta_{s}^2}{2(1-\delta)} $.
Note that $\theta_{s+1}$ is the positive root of $\frac{1-\theta_{s+1}+\delta}{1-\delta} \cdot \frac{1}{\theta_{s+1} ^2} = \frac{1}{\theta_{s}^2}$.
When $\delta=0$, the second-order equation is reduced to $\frac{1-\theta_{s+1}}{\theta_{s+1} ^2} = \frac{1}{\theta_{s}^2}$, which is crucial for the acceleration methods that aim to solve convex but not strongly convex smooth problems~\cite{nesterov2013introductory,nitanda2014stochastic,lin2015universal}.
The reason why we incorporate $\delta$ in the second-order equation is to control the residual error.

\begin{algorithm}[tb]
	\caption{Delayed Projected Accelerated SVRG (DP-ASVRG)}
	\label{alg:multi_acc_SVRG}
	\begin{algorithmic}
		\STATE {\bfseries Input:} function $F$, initial point $\x_0$ (and $\ttx_0 = \u_0^0 = \x_0^0 = \PAT(\x_0)$), stage number $S$, loop iteration $m$, 			projection set $\IM_m \subset [m]$ with $\mathrm{gap}(\IM_m) = E (1 \le E \le m )$, (sufficiently small) step size $\eta$, $\delta = 9(E^2-1)\eta^2 L^2 \in [0, 1)$.
		\STATE {\bfseries An auxiliary sequence:}  {If $\mu > 0$, $\theta_{s}\equiv \theta = 2\delta + \sqrt{4\delta^2 + \eta\mu m} \in (2\delta, 1+\delta)$ for all $s \ge 0$.}
		\STATE {If $\mu =0$, let $\theta_{0}  = 1 - \frac{2\eta L}{1-\eta L} \in (2\delta, 1+\delta)$ and $\theta_{s+1}  = \sqrt{ \frac{1+\delta}{1-\delta}\theta_{s}^2 + \frac{\theta_{s}^4}{4(1-\delta)^2}} - \frac{\theta_{s}^2}{2(1-\delta)} $.}
		\FOR{$s=0$ {\bfseries to} $S-1$}
		\STATE {$\tth_{s} \gets \PAT(\nabla F(\ttx_{s}))$ }
		\FOR{$t=0$ {\bfseries to} $m-1$}
		\STATE  $\g_{t}^s \gets \nabla F(\x_{t}^s; \xi_{t}^s) - \nabla F(\ttx_s; \xi_{t}^s) + \tth_s$  with  $\xi_{t}^s$ sampled independently 
		\STATE  $\u_{t+1}^s \gets \u_{t}^s - \frac{\eta}{\theta_s} \cdot \g_{t}^s $  
		\STATE  $\x_{t+1}^s \gets \ttx_{s} + \theta_s(\u_{t+1}^s - \ttx_{s}) $  
		\IF {$(t+1)\in \IM_m$}
		\STATE {$\x_{t+1}^s \gets \PAT(\x_{t+1}^s), \u_{t+1}^s \gets \PAT(\u_{t+1}^s)$ }
		\ENDIF
		\ENDFOR
		\STATE {$\u_0^{s+1} \gets \PAT(\u_m^s),\ttx_{s+1} \gets \PAT(\frac{1}{m}\sum_{i=1}^{m}\x_{i}^s), \x_0^{s+1} \gets \ttx_{s+1}$}
		\ENDFOR
		\STATE If $\mu > 0$, $\hat{\y} \gets \frac{1}{S}\sum_{s=1}^S \ttx_{s}$; otherwise $\hat{\y} \gets \ttx_{S}$.
		\STATE {\bfseries Return:} $\hat{\y}$.
	\end{algorithmic}
\end{algorithm}

\subsection{Analysis for Strongly Convex Objectives}

\begin{table}[t]
	\newcommand{\tabincell}[2]{\begin{tabular}{@{}#1@{}}#2\end{tabular}}
	\centering
	\begin{tabular}{|c|c|c|c|}
		\hline
		Items &  \tabincell{c}{	P-ASVRG~\cite{shang2018asvrg} \\ $m=N$ \\ $E=1$}
        &\tabincell{c}{	DP-ASVRG (Ours) \\  $m=\kappa^p, 0 \le p \le 1$ \\ $ E =1$}
		& \tabincell{c}{	DP-ASVRG (Ours) \\  $m=\kappa^p E, 0 \le p \le 1$ \\ $ E =\sqrt{ N/\kappa^q}, 0 \le q < \log_{\kappa} N$}\\
		\hline
		\tabincell{c}{	Stage\\ $(\SB)$ }
		& $\widetilde{\OM}\left(  \max\left\{ 1, \sqrt{\frac{\kappa}{N}} \right\}  \right)$
      & $ \widetilde{\OM}\left(  \kappa^{\frac{1}{2}(1-p)} \right)$
		&   $ \widetilde{\OM}\left(  \kappa^{\frac{2}{3}(1-p)} \right)$\\
		\hline	
		\tabincell{c}{	Iteration\\ $(\TB=m\SB)$ }
		& $\widetilde{\OM}\left( N+ \sqrt{N\kappa}   \right)$
        & $ \widetilde{\OM}\left(  \kappa^{\frac{1}{2}(1+p)} \right)$
		& $ \widetilde{\OM}\left( \sqrt{N} \kappa^{\frac{1}{3}(2+p)-\frac{q}{2}} \right)$\\
		\hline
		\tabincell{c}{	Projection\\ $(\PB = \SB + {\TB}/{E})$  }
		& $\widetilde{\OM}\left ( N+ \sqrt{N\kappa}  \right)$
       & $ \widetilde{\OM}\left(  \kappa^{\frac{1}{2}(1+p)} \right)$
		& $ \widetilde{\OM}\left(  \kappa^{\frac{1}{3}(2+p)} \right)$\\
		\hline
		\tabincell{c}{	Gradient \\ $(\GB = (m+N)\SB)$ }
		& $\widetilde{\OM}\left(  N + \sqrt{N\kappa} \right)$
		&  $\widetilde{\OM}\left(  N \kappa^{\frac{1}{2}(1-p)}+ \kappa^{\frac{1}{2}(1+p)} \right)$
		& $\widetilde{\OM}\left(  N \kappa^{\frac{2}{3}(1-p)}+
		 \sqrt{N} \kappa^{\frac{1}{3}(2+p)-\frac{q}{2}}\right)$\\
		\hline
	\end{tabular}
	\caption{Compare P-ASVRG and DP-ASVRG with $\IM_m^0$ in four aspects under the strongly convex setting.
		$N$ is the number of total samples and $\widetilde{\OM}(\cdot)$ omits a factor of $\ln \frac{L\Delta^2}{\eps}$.
		Note that P-ASVRG is a special case of DP-ASVRG when $m=N$ and $E=1$.
		We use $\TB, \SB, \PB, \GB$ to stand for the four complexities (see Definition~\ref{def:complexity}).
		We give their relations in the brackets of the first column.
		All results are obtained by plugging corresponding parameters in Corollary~\ref{cor:strong_DP_ASVRG}.
	}
	\label{table:4}
\end{table}%

To apply DP-ASVRG to solve the strongly convex problems, we restart the algorithm repeatedly and initialize each new restart with the output parameter $\hat{\y}$ produced in the last restart.
The following theorem specifies the number of stages each restart needs to half the optimization error, implying only $\OM\left(\ln \frac{E\left[F(\x_0)-F(\xc)\right]}{\eps} \right)$ restarts are needed.

\begin{thm}[Strongly convex case]
	\label{thm:acc_svrg_strong}
	Assume Assumption~\ref{asmp:smooth} and~\ref{asmp:strong} hold and let $\y_0 = \PAT(\x_0)$ and $\Delta^2 =  \EB \|\y_0 - \xc \|^2$.
	For strongly convex case $(\mu > 0)$, run Algorithm~\ref{alg:multi_acc_SVRG}  for $S$ stages and each stage has $m$ iterations with projection set as $\IM_m$ where $\mathrm{gap}(\IM_m) = E$ (so $m \ge E$).
	By choosing an appropriate constant step size $\eta = \Theta\left(\frac{1}{LE\kappa^{\frac{1}{3}}}\right)$, 
	when 
	\[
	S= \OM\left(  \max\left\{ 1, \sqrt{\frac{\kappa E}{m}} \right\} + \kappa^{\frac{2}{3}} \sqrt[3]{\frac{E^2-1}{m^2}}   \right),
	\]
	with $\kappa = \frac{L}{\mu}$ the condition number, we have
	\[
		\EB \left[F(\hat{\y}) - F(\xc)\right] \le \frac{1}{2} \cdot \EB \left[F(\y_0) - F(\xc)\right].
	\]
\end{thm}

\begin{cor}
	\label{cor:strong_DP_ASVRG}
	Under the setting of Theorem~\ref{thm:acc_svrg_strong}, in order to obtain an $\eps$-suboptimal solution (i.e., $\EB\left[F(\hat{\y}) - F(\xc)\right] \le \eps$), we can restart Algorithm~\ref{alg:multi_acc_SVRG} with $\IM_m^0$ for $\log_2\frac{\EB \left[ F(\y_0) - F(\xc) \right]}{\eps} $ times. 
	So the stage complexity is 
	\[
	\SB = \OM\left(  \left[  \max\left\{ 1, \sqrt{\frac{\kappa E}{m}} \right\} + \kappa^{\frac{2}{3}} \sqrt[3]{\frac{E^2-1}{m^2}}    \right] \ln\frac{\EB \left[ F(\y_0) - F(\xc) \right]}{\eps}   \right).
	\]
	As a result, its iteration complexity is $
	\TB=m\SB = \widetilde{\OM}\left(  \max\left\{ m, \sqrt{m\kappa E} \right\} + \kappa^{\frac{2}{3}}(E^2-1)^{\frac{1}{3}}m^{\frac{1}{3}} \right)$, projection complexity is $\PB = \SB + \frac{\TB}{E} = \widetilde{\OM}\left( \max\left\{ \frac{m}{E}, \sqrt{\frac{m\kappa}{E}} \right\} + \kappa^{\frac{2}{3}}\left(1-\frac{1}{E^2}\right)^{\frac{1}{3}}\left(\frac{m}{E}\right)^{\frac{1}{3}} \right)$, and gradient complexity is $\GB=\widetilde{\OM}\left(  \max\left\{ 1, \sqrt{\frac{\kappa E}{m}} \right\} + \kappa^{\frac{2}{3}} \sqrt[3]{\frac{E^2-1}{m^2}}   \right)\cdot(m + N)$, where $\widetilde{\OM}(\cdot)$ omits a factor of $\ln \frac{ \EB \left[ F(\y_0) - F(\xc) \right]}{\eps}$ for simplicity.
\end{cor}

When $E$ and $m$ are set correspondingly, DP-ASVRG is reduced to many previous algorithms (with slight differences).
Our theorem not only allows $E$ and $m$ to vary, but also recovers previous analysis for those reduced algorithms (see Table~\ref{table:4}).
There are some interesting observations:
\begin{itemize}
	\item When $m=N$ and $E=1$, DP-ASVRG is reduced to P-ASVRG.
	Its projection complexity and iteration complexity are the same since a  projection is performed at each iteration.
	Its gradient complexity is optimal in the sense that it achieves the lower bound on the number of gradient oracle accesses needed to find an $\eps$-suboptimal solution~\cite{woodworth2016tight}.
	There are other works trying to accelerate SVRG in other ways.
	For example, \cite{nitanda2014stochastic} propose accelerated proximal SVRG that uses Nesterov’s acceleration method and shows that with an appropriate mini-batch size, it achieves lower overall gradient complexity than proximal SVRG and accelerated proximal gradient descent.
	\item When $m=E=1$, DP-ASVRG is reduced to Nesterov Accelerated Gradient (NAG).
	Setting $p=0$ in the second column of Table~\ref{table:4}, NAG obtains the optimal iteration complexity and better projection complexity (i.e., $\widetilde{\OM}(\sqrt{\kappa})$), however, has the worst gradient complexity (i.e., $\widetilde{\OM}(N\sqrt{\kappa})$).
	\item When $E > 1$ (which is equivalent to $q < \log_{\kappa}N$ in the rightest column of Table~\ref{table:4}), delayed projections start to involve in optimization.
	 Additionally assuming $m/E =\kappa^p$, its projection complexity increases to $\widetilde{\OM}(\kappa^{\frac{1}{3}(2+p)})$.
	 The achievable smallest projection complexity is $\widetilde{\OM}(\kappa^{\frac{2}{3}})$, though larger that NAG but much smaller than DP-SVRG.
	 It implies DP-SVRG indeed can be accelerated, with statistical errors eliminated and projection complexity reduced.
	\item When $N \ge \kappa$ and we set $m=\kappa E = \sqrt{N\kappa}$ (i.e., $p=q=1$ in the rightest column of Table~\ref{table:4}), DP-ASVRG only needs $\widetilde{\OM}\left( \kappa\right)$ projections to obtain an $\eps$-suboptimal solution, while P-ASVRG needs $\widetilde{\OM}\left( N\right)$ projections, though the two algorithms has a same gradient complexity (i.e., $\widetilde{\OM}(  N + \sqrt{N\kappa} )$).
	Hence, it is possible to solve LCPs using projections much less than total iterations, and delayed projection technique will hnot cancel with variance reduction techniques.
\end{itemize}

\subsection{Analysis for General Convex Objectives}

\begin{thm}[Generally convex case]
	\label{thm:acc_svrg_general}
	Assume Assumption~\ref{asmp:smooth} and~\ref{asmp:strong} hold and let $\y_0 = \PAT(\x_0)$ and $\Delta^2 =  \EB \|\y_0 - \xc \|^2$.
	For convex case $(\mu = 0)$, run Algorithm~\ref{alg:multi_acc_SVRG} for $S$ stages and each stage has $m$ iterations with projection set as $\IM_m$ where $\mathrm{gap}(\IM_m) = E$ (so $m \ge E$).
	By choosing an appropriate constant step size $\eta = \widetilde{\Theta}\left(\frac{1}{EL\sqrt{S}}\right)$, we have
	\begin{equation*}
	\EB \left[F(\hat{\y}) - F(\xc)\right] 
	=\widetilde{\OM}\left(   \frac{F(\x_{0}) - F(\xc)}{S^2} + \frac{LE\Delta^2}{mS^2} + \frac{\sqrt{E^2-1}L\Delta^2}{mS^{1.5}}  \right).
	\end{equation*}
	If $E=1$, we can safely replace the above $\widetilde{\OM}(\cdot)$ with ${\OM}(\cdot)$.
\end{thm}

\begin{cor}
	\label{cor:general_DP_ASVRG}
	Under the setting of Theorem~\ref{thm:acc_svrg_general}, for Algorithm~\ref{alg:multi_acc_SVRG} using $\IM_m^0$ to obtain an $\eps$-suboptimal solution (i.e., $\EB\left[F(\hat{\y}) - F(\xc)\right] \le \eps$), the stage complexity is 
	\begin{align*}
	\SB
	&=\widetilde{\OM}\left(  \frac{\sqrt{L}\Delta}{\sqrt{\eps}} +\sqrt[3]{\frac{E^2-1}{m^2}}\cdot  \frac{ L^\frac{2}{3}\Delta^\frac{4}{3}}{\eps^{\frac{2}{3}}}  \right).
	\end{align*}
	As a result, its iteration complexity is $\TB=m\SB = \widetilde{\OM}\left(  m\frac{\sqrt{L}\Delta}{\sqrt{\eps}} + 
	(E^2-1)^{\frac{1}{3}}m^{\frac{1}{3}}	\frac{ L^\frac{2}{3}\Delta^\frac{4}{3}}{\eps^{\frac{2}{3}}} \right)$, projection complexity is $\PB = \frac{\TB}{E} = \widetilde{\OM}\left( \frac{m}{E}\frac{\sqrt{L}\Delta}{\sqrt{\eps}}  + \left(1-\frac{1}{E^2}\right)^{\frac{1}{3}}\left(\frac{m}{E}\right)^{\frac{1}{3}} \frac{ L^\frac{2}{3}\Delta^\frac{4}{3}}{\eps^{\frac{2}{3}}} \right)$, and gradient complexity is $\GB = \widetilde{\OM}\left(  m\frac{\sqrt{L}\Delta}{\sqrt{\eps}} + 
	(E^2-1)^{\frac{1}{3}}m^{\frac{1}{3}}	\frac{ L^\frac{2}{3}\Delta^\frac{4}{3}}{\eps^{\frac{2}{3}}}   \right)\cdot(m + N)$, where $\widetilde{\OM}(\cdot)$ omits a factor of $\ln \frac{1}{\eps}$ for simplicity.
	If $E=1$, we can safely replace the above $\widetilde{\OM}(\cdot)$ with ${\OM}(\cdot)$.
\end{cor}

We have the following observations:
\begin{itemize}
	\item When $m=N$ and $E=1$, DP-ASVRG is reduced to P-ASVRG that has ${\OM}(N \sqrt{\frac{L}{\eps}} \Delta)$ projection complexity and ${\OM}( N \sqrt{\frac{L}{\eps}} \Delta)$ gradient complexity, consistent with previous analysis~\cite{shang2018asvrg}.
	\item When $E > 1$, DP-SAVRG converges with rate $\widetilde{\OM}(1/S^{1.5})$, much faster than DP-SVRG's  $\widetilde{\OM}(1/S)$, implying acceleration works.
	However, one can see that DP-ASVRG rate fails to match the best known result in~\cite{allen2017katyusha,nguyen2019accelerated,shang2018asvrg} that attains the optimal convergence rate $\OM(1/S^2)$ due to the residual error.
	\item We find that DP-ASVRG has advantage over P-ASVRG on projection complexity at the large $N$ regime or low accuracy (large $\eps$) regime.
	For example, when $m=(\frac{L}{\eps})^{1/4} \Delta^{1/2}E$, the projection complexity for DP-ASVRG is $ \widetilde{\OM}\left((\frac{L}{\eps})^{3/4} \Delta^{3/2}\right)$, which will be tremendously smaller than P-ASVRG's ${\OM}( N \sqrt{\frac{L}{\eps}} \Delta )$ when $\eps \ge \frac{L \Delta^2}{N^4}$.
	Fortunately, typical machine learning tasks don't need high accuracy solution.
	For example,~\cite{anonymous2021sharper} suggests that if $F(\cdot)$ satisfies the Polyak-Lojasiewicz (PL) inequality~\cite{polyak1963gradient} at the optima, typically $\eps = \Theta(1/N)$ is enough for a good generalization.
\end{itemize}

The best know gradient complexity for generally convex smooth finite sum minimization is $\widetilde{\OM}\left( N + \sqrt{N\frac{L}{\eps}}  \right) $.
As argued by~\cite{allen2017katyusha,shang2018vr}, to achieve that lower bound, we can use the adaptive regularization technique proposed in~\cite{allen2016optimal} to the original non-strongly convex optimization.
In particular, we aim to minimize $F(\x) + \frac{\mu_l}{2} \| \x - \x_0\|^2$ with a exponentially decreasing value $\sigma_l$ (e.g., $\sigma_l = \sigma_{l-1}/2$) and we will decrease the value of $\mu_l$ at an appropriate time until it reaches around $\Theta(\eps)$.
As a result, we can improve Corollary~\ref{cor:general_DP_ASVRG} to the following.
As a thumb of rule, one can derive Corollary~\ref{cor:general_DP_ASVRG_1} by replacing $\mu$ with $\eps$ in Corollary~\ref{cor:strong_DP_ASVRG}.
Hence, the discussion in the last subsection can apply here.
For example, when $N \ge \frac{L}{\eps}$ and we set $m= \frac{L}{\eps} E = \sqrt{N \frac{L}{\eps}}$ for DP-ASVRG, DP-ASVRG only needs $\widetilde{\OM}\left(  \frac{L}{\eps} \right)$ projections to obtain an $\eps$-optimal solution, while P-ASVRG needs $\widetilde{\OM}\left( N + \sqrt{N \frac{L}{\eps}}\right)$ projections, though the two algorithms has a same gradient complexity.

\begin{cor}
	\label{cor:general_DP_ASVRG_1}
	Under the same condition of Theorem~\ref{thm:acc_svrg_general}, in order to obtain an $\eps$-suboptimal solution (i.e., $\EB\left[F(\hat{\y}) - F(\xc)\right] \le \eps$), we use the adaptive regularization technique in~\cite{allen2016optimal} to the original non-strongly convex optimization.
	Then, the required stage complexity is 
	\[
	\SB = \widetilde{\OM}\left(  \left[  \max\left\{ 1, \sqrt{\frac{L E}{\eps m}} \right\} + \left( \frac{L}{\eps} \right)^{\frac{2}{3}} \sqrt[3]{\frac{E^2-1}{m^2}}    \right]  \right).
	\]
	As a result, its iteration complexity is $
	\TB=m\SB = \widetilde{\OM}\left(  \max\left\{ m, \sqrt{m\frac{L}{\eps} E} \right\} + \left(\frac{L}{\eps}\right)^{\frac{2}{3}}(E^2-1)^{\frac{1}{3}}m^{\frac{1}{3}} \right)$, projection complexity is $\PB = \SB + \frac{\TB}{E} = \widetilde{\OM}\left( \max\left\{ \frac{m}{E}, \sqrt{\frac{mL}{E\eps}} \right\} + \left(\frac{L}{\eps}\right)^{\frac{2}{3}}\left(1-\frac{1}{E^2}\right)^{\frac{1}{3}}\left(\frac{m}{E}\right)^{\frac{1}{3}} \right)$, and gradient complexity is $\GB=\widetilde{\OM}\left(  \max\left\{ 1, \sqrt{\frac{ EL}{m\eps}} \right\} + \left(\frac{L}{\eps}\right)^{\frac{2}{3}} \sqrt[3]{\frac{E^2-1}{m^2}}   \right)\cdot(m + N)$, where $\widetilde{\OM}(\cdot)$ omits a factor of $\ln \frac{ 1}{\eps}$ for simplicity.
\end{cor}

\section{Applications in Federated Optimization}
\label{sec:FL}

Federated Learning (FL) emerges as a new distributed computing paradigms that try to perform private distributed optimization in large-scale networks of remote clients~\cite{kairouz2019advances}.
In particular, we have the following distributed optimization problem across $n$ worker nodes:
\begin{equation}
\label{eq:loss}
f(\x)
= \frac{1}{n} \sum_{k=1}^{n} f_{k}(\x) 
:= \frac{1}{n} \sum_{k=1}^{n} \EB_{\xi \sim \DM_k} f(\x; \xi).
\end{equation}
In conventional distribute learning, a distribute system evenly allocates the whole dataset into $n$ worker nodes and often periodically shuffles the data to make sure each worker node has access to the underlying data distribution.
Therefore, $\DM_1 = \cdots = \DM_n = \DM$.
However, in Federated Learning, for the sake of privacy protection, data are generated locally and are prohibitive to be uploaded to the data center, which incurs a discrepancy among local data distributions, i.e., $\{\DM_i\}_{i=1}^n$ are not necessarily identical anymore.
What's more, any third party including the center has no access to data instances generated by any worker node.
Apart from data heterogeneity, FL systems also present other challenges characterized by expensive communication costs, unreliable connection, massive scale, and privacy constraints~\cite{li2020federated}.

In the section, we show how to apply our new methods with delayed projections to federated optimization and how derived theories help understand their convergence behaviors.
We assume that there are $n$ machines and denote its parameter by $\x_t^{(k)}$ or $\x_{t,s}^{(k)}$ with $t, s$ denoting the inner and outer iterations when two loops are used.
The selected sample at that iteration is denoted by $\xi_t^{(k)}$ or $\xi_{t,s}^{(k)}$.
Each device holds a objective $f_k(\x)$ in a form $f_k(\x) = \EB_{\xi \sim \DM_k} f(\x; \xi)$ or $f_k(\x) = \frac{1}{n_k} \sum_{i=1}^{n_k} f(\x; \varsigma_i)$ where $\varsigma_i$ are generated independently from $\DM_k$.

In this section, the notation will be slightly different from that in the introduction, since we need additional superscripts to distinguish different devices. 
Let $\x_t = [(\x_t^{(1)})^\top, \cdots, (\x_t^{(n)})^\top]^\top \in \RB^{nd}$ be the concatenated variable at iteration $t$
and $\xi_t =[(\xi_t^{(1)})^\top, \cdots, (\xi_t^{(n)})^\top]^\top \in \RB^{n}$ the concatenated samples selected at iteration $t$.
Let $F(\x_t) = \sum_{k=1}^n f_k(\x_t^{(k)}) $ the objective function of it.
Assuming each local function $f_k(\x) = \EB_{\xi \sim \DM_k}f(\x; \xi)$ is $L_k$-smooth and $\mu_k$-strongly convex, one can show that $F(\cdot)$ is $L$-smooth with modulus $L=\max_{k \in [n]} L_k$ and $\mu$-strongly convex with modulus $\mu=\min_{k \in [n]} \mu_k$ by definition, satisfying Assumption~\ref{asmp:smooth} and~\ref{asmp:strong}.
Let $\x^* = \argmin_{\x \in \RB^{d}} \sum_{k=1}^n f_k(\x)$ and $\xc = \argmin_{\x \in \mathcal{R}(\A^{\perp})} \sum_{k=1}^n f_k(\x^{(k)})$ with $\A$ given in~\eqref{eq:B}, then obviously $\xc = \x^* \otimes 1_n$.

\subsection{Recover the analysis for Local SGD}
In this section, we show how derived Theorem~\ref{thm:simple} and~\ref{thm:complicate} help to give theoretical results for Local SGD (Algorithm~\ref{alg:local_sgd}).
This is a typical procedure of reducing a distribution optimization problem to an LCP.

The stochastic gradient of $F(\x)$ is given by $\nabla F(\x; \xi) = [\nabla f(\x^{(1)}; \xi^{(1)})^\top, \cdots, \nabla f(\x^{(n)}; \xi^{(n)})^\top]^\top \in \RB^{nd}$ where $\xi = [\xi^{(n)}, \cdots, \xi^{(n)}]^\top$ denotes by the selected samples used to generate stochastic gradients.
Here each $\xi^{(k)} \sim \DM_k$ is generated independently but may conform to different distributions.
By assuming each $f_k(\cdot)$ has bounded stochastic gradient variance on $\x^*$, $F(\cdot)$ meets Assumption~\ref{assum:bounded} with parameters satisfying the following relation:

\begin{lem}
	\label{lem:sigma}
	Define  
	\begin{equation}
	\label{eq:local_sigma}
	\sigma_*^2 = \frac{1}{n} \sum_{k=1}^n \EB_{\xi^{(k)}\sim \DM_k} \| \nabla f(\x^*; \xi^{(k)}) -\nabla f_k(\x^*) \|^2 
	\ \text{and} \
	\zeta_*^2 = \frac{1}{n} \sum_{k=1}^n \| \nabla f_k(\x^*) \|^2.
	\end{equation}
	Then $F(\x) = \sum_{k=1}^n f_k(\x_k)$ satisfies Assumption~\ref{assum:bounded} with parameters
	\[
	\sigmaat  =   \sigma_*^2
	\ \text{and} \
	\sigmaa = n 	\zeta_*^2 + (n-1) \sigma_*^2.
	\]
\end{lem}

DP-SGD (Algorithm~\ref{alg:multi}) is the synonyms for Local SGD (Algorithm~\ref{alg:local_sgd}) in the context of single-machine LCP.
Each machine performs SGD locally via $\x_t = \x_{t-1} - \eta_{t-1} \nabla F(\x_{t-1}; \xi_{t-1})$ and periodically synchronizes local model parameters with global average that is equivalent to projection here $\x_t \gets \PAT(\x_t)$.
When all local data distribution are identical ($\DM_1 = \DM_2 = \cdots = \DM_n := \DM$), each device has the access to the underlying data distribution.
As a result, $f_1(\x) = f_2(\x) =\cdots = f_n(\x) = f(\x) := \EB_{\xi\sim \DM} f(\x; \xi)$.
With $\A$ given in~\eqref{eq:B}, Lemma~\ref{lem:proj} shows any $\y_0 \in \mathcal{R}(\A^{\perp})$ has $n$ identical block of coordinates. 
Hence, $\nabla F(\y_0) = \EB_{\xi}\nabla F(\y_0; \xi)$ also has $n$ identical block of coordinates and thus belongs to $\mathcal{R}(\A^{\perp})$, implying Assumption~\ref{asmp:g_unconstrained} holds.
Once local data distribution varies (i.e., there exists a pair $i \neq j$ such that $\DM_i \neq \DM_j$), Assumption~\ref{asmp:g_unconstrained} might not hold.

\begin{cor}[Local SGD on identical and heterogeneous data]
	\label{cor:local_sgd}
	Assume each $f_k(\cdot)$ is $L_k$-smooth, $\mu_k$-strongly convex, and has bounded gradient variance at the optimum $\x^*$ with parameters defined in~\eqref{eq:local_sigma}.
	Start from $\x_0$ that $\|\x_0 - \x^*\| \le B$, run Local SGD for $T$ iterations with $\mathrm{gap}(\IM_T) = E \ (E \ge 1)$, and tune the constant learning rate $\eta \le \Theta(\frac{1}{\mu +LE})$ where $L=\max_{k \in [n]} L_k$ and $\mu=\min_{k \in [n]} \mu_k$.
	Then Local SGD produces a global $\hat{\x}$ satisfying: $\EB \left[f(\hat{\x}) - f(\x^*) \right]=$
	\begin{equation}
	\label{eq:y_mu0_local}
	\OM\left( \frac{LEB^2}{T} +\frac{B\sigma_*}{\sqrt{nT}} + \frac{\sqrt[3]{(E-1)LB^4}}{T^{\frac{2}{3}}} \cdot {   \left[E\zeta_*^2 + \frac{n-1}{n} \sigma_*^2\right] }^{\frac{1}{3}}\right) , 
	\end{equation}
	for convex case $(\mu = 0)$ and
	\begin{equation}
	\label{eq:y_mu>_local}
	\widetilde{\OM}\left( LEB^2 \cdot\exp\left(-\Theta\left(\frac{\mu T}{LE}\right) \right) + \frac{\sigma_*^2}{n\mu T}  + \frac{(E-1)L}{\mu^2 T^2}\cdot \left[E\zeta_*^2 + \frac{n-1}{n} \sigma_*^2\right] \right).
	\end{equation}
	for strongly convex case $(\mu > 0)$, no matter whether each machine obtains samples from an identical data distribution or not.
\end{cor}

In the above corollary, we derive the convergence result for Local SGD easily from Theorem~\ref{thm:simple} and~\ref{thm:complicate} and the result is is finer than the state-of-the-art analysis~\cite{koloskova2020unified,woodworth2020minibatch}.
Local SGD suffers an additional term named as the residual error, the third term of~\eqref{eq:y_mu0_local} and~\eqref{eq:y_mu>_local}, than traditional SGD~\cite{stich2019unified}.
When $E=1$ (no local updates) or $n=1$ (no other participants\footnote{This follows since the constraint vanishes when $n=1$ and thus $\zeta_* = 0$.}), the residual error vanishes.
Otherwise, distributed methods with local updates inevitably suffer the residual error due to delayed communication and periodic synchronization.
Many previous works including ours prove the residual error should form in a function of $\OM(E\sigma_*^2 + E^2 \zeta_*^2)$~\cite{li2019communication,stich2019error,bayoumi2020tighter,koloskova2020unified,woodworth2020local,woodworth2020minibatch,koloskova2020unified}.
In particular,~\cite{woodworth2020local} and~\cite{woodworth2020minibatch} present lower bounds on the performance of local SGD with
\[
\Omega\left(  \frac{\sigma B}{\sqrt{nT}} + \frac{(LB^4\sigma^2)^{\frac{1}{3}}}{T^{\frac{2}{3}}} + \min\left\{ \frac{LEB^2}{T}, \frac{(LE^2B^4\zeta_*^2)^{\frac{1}{3}}}{T^{\frac{2}{3}}}\right\} \right)
\]
for generally convex cases $(\mu=0)$ and 
\[\Omega\left( \frac{\sigma B}{\mu n T} + \min\left\{ \Delta_0,   \frac{L\sigma^2}{\mu^2T^2} \right\} + \min\left\{ \Delta_0\exp\left(-\frac{6\mu}{L} \frac{T}{E}  \right), \frac{L E^2\zeta_*^2 }{\mu^2 T^2}   \right\}  \right)  \]
for strongly convex cases $(\mu>0)$ where $\sigma^2$ is the uniform bound on stochastic gradients and $\Delta_0^2 = \EB \left[F(\x_0) - F(\x^*) \right]$ is the initial error in function values (note that we almost have $\Delta_0^2 \approx L B^2$ by smoothness).
We can see that those lower bounds almost match the upper bounds~\eqref{eq:y_mu0_local} and~\eqref{eq:y_mu>_local} except that the second term is not matched up to a factor of $E^{\frac{1}{3}}$ and $E$ respectively
It still remains an open problem to close the gap.

\subsection{Remove Statistical Errors and Residual Errors}
\begin{algorithm}[tb]
	\caption{Local SVRG}
	\label{alg:local_SVRG}
	\begin{algorithmic}
		\STATE {\bfseries Input:} functions $\{f_k\}_{k=1}^n$, initial point $\x_0$, step size $\eta_t^s$, stage number $S$, loop iteration $m$,
		communication set $\IM_m \subset [m]$ with $\mathrm{gap}(\IM_m) = E  (E \ge 1)$.
		\STATE {\bfseries Initialization:} let $\ttx_0 = \x_{0,0}^{(k)} = \x_0$ for all $k$.
		\FOR{$s=0$ {\bfseries to} $S-1$}
		\STATE {$\tth_{s} \gets \frac{1}{n} \sum_{k=1}^n \nabla f_k(\ttx_{s})$ }
		\FOR{$t=0$ {\bfseries to} $m-1$}
		\FOR {each device $k=1$ {\bfseries to} $n$}
		\STATE  $\g_{t, s}^{ (k)} \gets \nabla f_k(\x_{t,s}^{(k)}; \xi_{t, s}^{(k)}) - \nabla f_k(\ttx_s; \xi_{t,s}^{(k)}) + \tth_s$  with  $\xi_{t,s}^{(k)}$ sampled independently on device $k$
		\STATE  $\x_{t+1,s}^{(k)} \gets \x_{t,s}^{(n)} - \eta_{t}^s\g_{t,s}^{(k)} $  
		\IF {$(t+1)\in \IM_m$}
		\STATE {$\x_{t+1,s}^{(k)} \gets  \frac{1}{n} \sum_{j=1}^n \x_{t+1,s}^{(j)}  $ \quad \# synchronization }
		\ENDIF
		\ENDFOR
		\ENDFOR
		\STATE {$\x_{0, s+1}^{(k)} \gets  \frac{1}{n} \sum_{k=1}^n \x_{m,s}^{(k)}$ for all $k$ \quad \# initialization for the next stage}
		\STATE {$\ttx_{s+1} \gets \sum_{i=0}^{m-1}(1-\mu\eta)^{i}\frac{1}{n}\sum_{k=1}^n\x_{m-i-1, s}^{(k)}/\sum_{j=0}^{m-1}(1-\mu\eta)^{j}$}
		\ENDFOR
		\STATE If $\mu = 0$, $\hat{\y} \gets \frac{1}{S}\sum_{s=1}^S \ttx_{s}$; otherwise $\hat{\y} \gets \ttx_{S}$.
		\STATE {\bfseries Return:} $\hat{\y}$.
	\end{algorithmic}
\end{algorithm}

When applying DP-SVRG  to solve the specific distributed  problem~\eqref{eq:constraint}, we obtain a novel distributed algorithm, named Local SVRG (Algorithm~\ref{alg:local_SVRG}).
Similar to previous federated optimization methods~\cite{sahu2018convergence,karimireddy2019scaffold,yuan2020federated}, Local SVRG periodically synchronizes local parameters with their average and the synchronization interval is no larger than $E$.
However, Local SVRG can remove both statistical errors and residual errors, achieved by no previous works.

Typically, the residual error is often believed to come from the data heterogeneity (i.e., $\zeta_* > 0$) and local updates (i.e., $E > 1$). 
It is well known that such data heterogeneity degrades the performance of the global model and may even result in divergence~\cite{zhao2018federated,li2019convergence,zhang2020fedpd}.
Previous researchers want to alleviate or even try to remove the effect of data discrepancy via an impractical data sharing strategy~\cite{zhao2018federated} or making use of control variates \cite{karimireddy2019scaffold,liang2019variance} and primal-dual methods~\cite{zhang2020fedpd}.
\cite{liang2019variance} sets the control variate as the accumulated difference of each individual local parameter and the global parameter.
However, their analysis gives guarantees in the non-convex world, and, though, without dependence on the residual error, their algorithm still suffers from the statistical error.
The most related work is~\cite{karimireddy2019scaffold}.
It proposes SCAFFOLD and uses the same type of control variates as we do.\footnote{~\cite{karimireddy2019scaffold} considers a more general situation where each device participates in the training with probability at each communication round. Here we discuss the special case where all devices participate in each round.} 
To achieve an $\eps$-suboptimal solution, SCAFFOLD needs $\widetilde{\OM}\left(  \frac{\sigma^2 B^2}{nE\eps^2} +  \frac{L B^2}{\eps} + F \right)$ and $\widetilde{\OM}\left( \frac{L}{\mu} +\frac{\sigma^2}{\mu n E \eps}  \right)$ communication rounds respectively for generally convex and strongly convex problems, where $\sigma^2$ is the uniform bound on gradient variance and $F = \EB \left[f(\x_0) - f(\x^*)\right]$.
By contrast, Local SVRG only needs $\widetilde{\OM}\left(  \frac{ L B^2}{\eps}  +  \frac{ m F}{E \eps}  \right)$ and $\widetilde{\OM}\left( \max \left\{ \frac{L}{\mu}, \frac{m}{E} \right\}  \right)$ communication rounds for corresponding cases, which are much smaller quantities if $\frac{m}{E} = \Theta(1)$.

\begin{cor}[Local SVRG]
	\label{cor:local_SVRG}
	Assume each $f_k(\cdot)$ is $L_k$-smooth, $\mu_k$-strongly convex.
	Start from $\x_0$ that $\|\x_0 - \x^*\| \le B$, run Local SVRG for $S$ stages, each stages has $m$ iterations with $\mathrm{gap}(\IM_m) = E \ (1 \le E \le m)$, and tune the constant learning rate $\eta \le \Theta(\frac{1}{LE})$ where $L=\max_{k \in [n]} L_k$.
	Let $F = \EB \left[f(\x_0) - f(\x^*)\right]$ denote the initial error.
	Then Local SVRG produces a global $\hat{\x}$ satisfying: $\EB \left[f(\hat{\x}) - f(\x^*) \right]=$
	\[
		\OM\left(  \frac{LEB^2}{T} + \frac{F}{S} \right)
	\]
	for convex case $(\mu = 0)$ and
	\[
	\OM \left( \left[   L EB^2 +F\right] \cdot 
	\exp\left( -  \Theta \left(\frac{T}{\max\{ \kappa E,  m \}} \right) \right)\right)
	\]
	for strongly convex case $(\mu > 0)$, no matter whether each machine obtains samples from an identical data distribution or not.
\end{cor}

\subsection{Acceleration}
\begin{algorithm}[tb]
	\caption{Local Accelerated SVRG}
	\label{alg:local_acc_SVRG}
	\begin{algorithmic}
		\STATE {\bfseries Input:} functions $\{f_k\}_{k=1}^n$, initial point $\x_0$, stage number $S$, loop iteration $m$, communication set $\IM_m \subset [m]$ with $\mathrm{gap}(\IM_m) = E (1 \le E \le m )$, (sufficiently small) step size $\eta$, $\delta = 9(E^2-1)\eta^2 L^2 \in [0, 1)$.
		\STATE {\bfseries An auxiliary sequence:}  {If $\mu > 0$, $\theta_{s}\equiv \theta = 2\delta + \sqrt{4\delta^2 + \eta\mu m} \in (2\delta, 1+\delta)$ for all $s \ge 0$.}
		\STATE {If $\mu =0$, let $\theta_{0}  = 1 - \frac{2\eta L}{1-\eta L} \in (2\delta, 1+\delta)$ and $\theta_{s+1}  = \sqrt{ \frac{1+\delta}{1-\delta}\theta_{s}^2 + \frac{\theta_{s}^4}{4(1-\delta)^2}} - \frac{\theta_{s}^2}{2(1-\delta)} $.}
		\STATE {\bfseries Initialization:} let $\ttx_0 = \x_{0,0}^{(k)} = \x_0$ for all $k$.
		\FOR{$s=0$ {\bfseries to} $S-1$}
		\STATE {$\tth_{s} \gets \frac{1}{n} \sum_{k=1}^n \nabla f_k(\ttx_{s})$ }
		\FOR{$t=0$ {\bfseries to} $m-1$}
		\FOR {each device $k=1$ {\bfseries to} $n$}
		\STATE  $\g_{t, s}^{ (k)} \gets \nabla f_k(\x_{t,s}^{(k)}; \xi_{t, s}^{(k)}) - \nabla f_k(\ttx_s; \xi_{t,s}^{(k)}) + \tth_s$  with  $\xi_{t,s}^{(k)}$ sampled independently on device $k$
		\STATE  $\u_{t+1,s}^{(k)} \gets \u_{t,s}^{(k)} - \frac{\eta}{\theta_s} \cdot \g_{t,s}^{(k)} $  
		\STATE  $\x_{t+1,s}^{(k)} \gets \ttx_{s} + \theta_s(\u_{t+1,s}^{(k)} - \ttx_{s}) $  
		\IF {$(t+1)\in \IM_m$}
		\STATE {$\x_{t+1,s}^{(k)} \gets \frac{1}{n}\sum_{j=1}^n \x_{t+1,s}^{(j)}, \u_{t+1,s}^{(k)} \gets   \frac{1}{n}\sum_{j=1}^n\u_{t+1,s}^{(j)}$ }
		\ENDIF
		\ENDFOR
		\ENDFOR
		\STATE {$\ttx_{s+1} \gets \frac{1}{mn}\sum_{i=0}^{m-1}\sum_{k=1}^n\x_{i,s}^{(k)}$}
		\STATE {$\x_{0, s+1}^{(k)} \gets  \ttx_{s+1}$ and $\u_{0, s+1}^{(k)} \gets   \frac{1}{n} \sum_{j=1}^n  \u_{m,s}^{(j)}$ for all $k$ \quad \# initialization for the next stage}
		\ENDFOR
		\STATE If $\mu > 0$, $\hat{\y} \gets \frac{1}{S}\sum_{s=1}^S \ttx_{s}$; otherwise $\hat{\y} \gets \ttx_{S}$.
		\STATE {\bfseries Return:} $\hat{\y}$.
	\end{algorithmic}
\end{algorithm}

When applying DP-ASVRG to solve the specific distributed  problem~\eqref{eq:constraint}, we obtain a novel distributed algorithm, named Local ASVRG (Algorithm~\ref{alg:local_acc_SVRG}).

\begin{cor}[Local ASVRG]
	\label{cor:local_ASVRG}
	Assume each $f_k(\cdot)$ is $L_k$-smooth, $\mu_k$-strongly convex.
	Start from $\x_0$ that $\|\x_0 - \x^*\| \le B$, run Local ASVRG for $S$ stages, each stages has $m$ iterations with $\mathrm{gap}(\IM_m) = E \ (1 \le E \le m)$, and tune the constant learning rate sufficiently small.
	Then Local ASVRG produces an $\eps$-suboptimal solution in 
	\[
	\widetilde{\OM}\left( \frac{m}{E}\frac{\sqrt{L}B}{\sqrt{\eps}}  + \left(1-\frac{1}{E^2}\right)^{\frac{1}{3}}\left(\frac{m}{E}\right)^{\frac{1}{3}} \frac{ L^\frac{2}{3}B^\frac{4}{3}}{\eps^{\frac{2}{3}}} \right)
	\]
	communication rounds for generally convex case $(\mu = 0)$ and in
	\[
	\widetilde{\OM}\left( \max\left\{ \frac{m}{E}, \sqrt{\frac{m \kappa}{E}} \right\} + \kappa^{\frac{2}{3}}\left(1-\frac{1}{E^2}\right)^{\frac{1}{3}}\left(\frac{m}{E}\right)^{\frac{1}{3}} \right)
	\]
communication rounds for strongly convex case $(\mu > 0)$, where $\kappa = \frac{L}{\mu}$ is the condition number.
\end{cor}

Typically, all the discussion on DP-ASVRG can be paralleled to Local ASVRG.
For example, let's focus on the strongly convex case.
 The smallest round is $\widetilde{\OM}(\sqrt{\kappa})$, achieved by $m=E=1$ for Local ASVRG, which is also achieved by other distributed algorithm like ADMM~\cite{boyd2011distributed} and AIDE~\cite{reddi2016aide}.
 \cite{arjevani2015communication} shows that $\widetilde{\OM}(\sqrt{\kappa})$ is the optimal communication complexity for strongly convex and smooth distributed optimization problems.
If saving computation is the primal goal, one can set $m=\kappa E$ and $E = \max \{ \sqrt{N/\kappa}, 1 \}$ for Local ASVRG and has $\widetilde{\OM}(N + \max \{ \kappa, \sqrt{N \kappa} \})$ gradient computations, which is optimal in the large scale case where $N \ge \kappa$.
At that case, the required communication round is $\widetilde{\OM}(\kappa)$, which is also achieved by Local SVRG and CoCoA~\cite{jaggi2014communication,ma2015adding}.
When we set $m=E=\max\{N^{1/2}/\kappa^{1/6}, 1\}$, the communication complexity for Local ASVRG becomes $\widetilde{\OM}(\kappa^{2/3})$ and its iteration complexity becomes $\widetilde{\OM}(\max \{\kappa^{2/3}, \sqrt{N\kappa }\})$, both smaller than Local SVRG and Local ASVRG with $m=\kappa E$ and $E = \max \{ \sqrt{N/\kappa}, 1 \}$.
Actually, we can see that if $E>1$, the fastest communication rounds for Local ASVRG is $\widetilde{\OM}(\kappa^{2/3})$, which fails the match the $\widetilde{\OM}(\sqrt{\kappa})$ lower bound.
We believe it results from the residual error since several infeasible updates are performed during two consecutive projections.

Many previous algorithms has been shown to enjoy similar or even better communication complexity when some ideal assumptions are made.
We give a brief introduction to those ideal cases.
\begin{enumerate}
	\item \textbf{Simpler models:}~\cite{jaggi2014communication,ma2015adding,zhang2015disco,wang2018giant} assume a linear model, implying their local objective can be written as $f_k(\x) = \frac{1}{n_k} \sum_{i=1}^{n_k} \ell_i(\x^\top\varsigma_i)$.
	The structure makes it easier to solve, for example, by using dual methods~\cite{jaggi2014communication,ma2015adding} or using subsampled Newton methods~\cite{wang2018giant}.
	Here, we don't impose such constraints and consider arbitrary models as long as they satisfy our assumptions.
	\item \textbf{Similar local objectives:} Many works~\cite{shamir2014communication,jaggi2014communication,zhang2015disco,yang2019federated} assume each local objective functions are the same, which can be achieved by assuming each device has access to the global underlying data distribution or all local data are i.i.d. generated.
	Obviously, such i.i.d. assumptions can't apply to FL.
	~\cite{sahu2018convergence,fan2019communication,li2019convergence,haddadpour2019convergence} define some quantities to measure the degree of data heterogeneity and assume it is finite.
	The finite non-i.i.d. assumption shrinks the class of objective functions taken into account.
	By contrast, we allow arbitrary local objective functions.
	We don't use another quantity to measure the non-i.i.d. degree for Local SVRG and Local ASVRG, since, as shown in the last subsection, they can eliminate both the statistical error and residual error, due to the used variance reduction technique.
	\item  \textbf{Extra dataset:}~\cite{lee2017distributed} proposes a distributed version of SVRG that gives each device access to an extra dataset that conforms to the global data distribution $\DM$ to ensure the unbiasedness of stochastic gradients. 
	However, such an extra dataset is impossible in FL; even if a shared dataset can be obtained voluntarily, it is not easy to ensure it comes from $\DM$.
	Therefore, it is not practical in FL, even though it has admirable communication complexity $\widetilde{\OM}(1 + \frac{\kappa}{\alpha N})$, where $\alpha N$ is the size of the extra dataset.
\end{enumerate}
DP-ASVRG doesn't require any of the three impractical assumptions.
It is the first accelerated algorithm in federated optimization that is able to eliminate the data heterogeneity.

\section{Experiments}

\subsection{Linear Equiality Constrained Logistic Regression}
\begin{figure*}[tp]
	\centering
	\hspace{-1in}
	\subfigure[DP-methods with $E=10$]{\includegraphics[height=65mm] {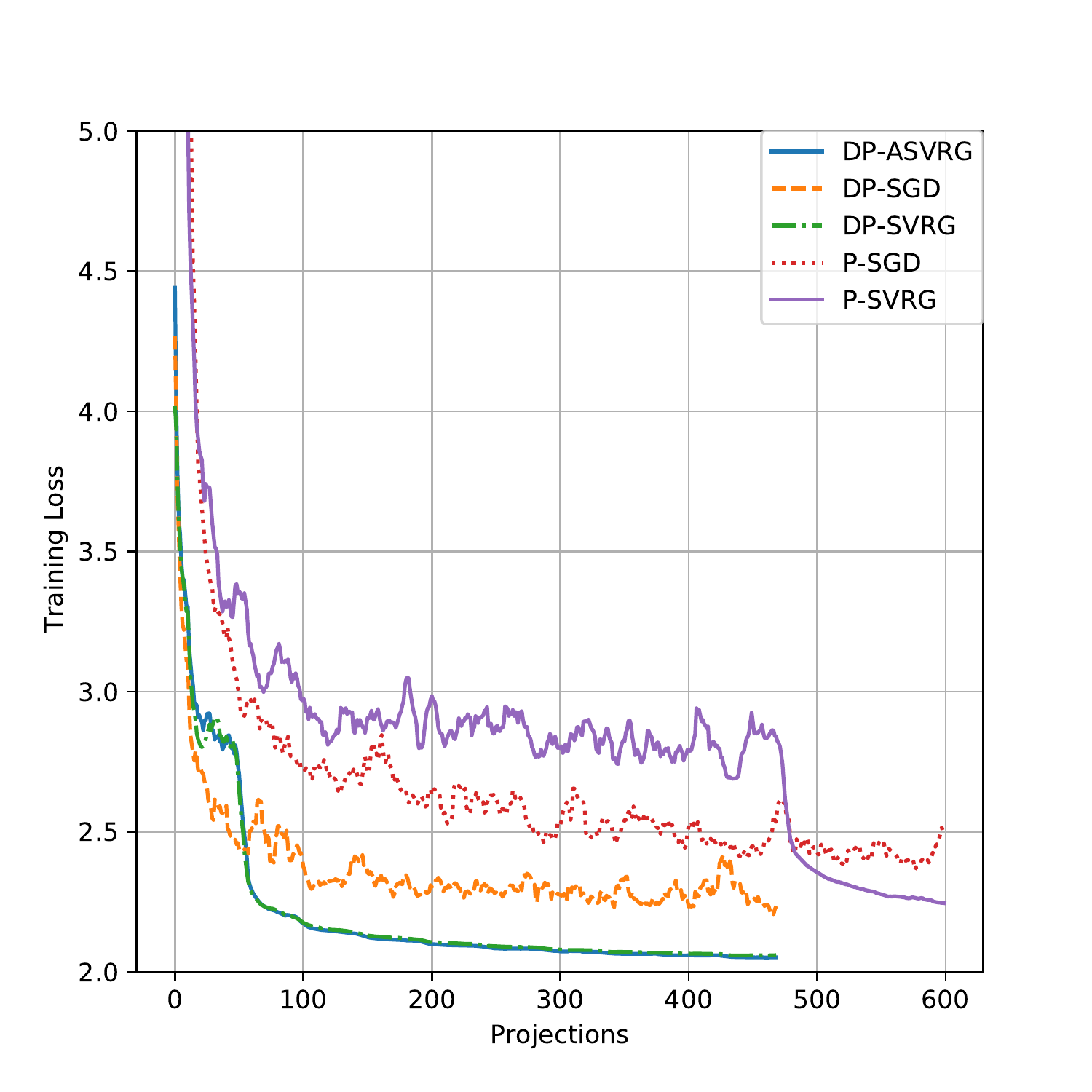} \label{fig:dp_a}}
	\hspace{-0.4in}
	\subfigure[DP-ASVRG with different $\theta$]{\includegraphics[height=65mm] {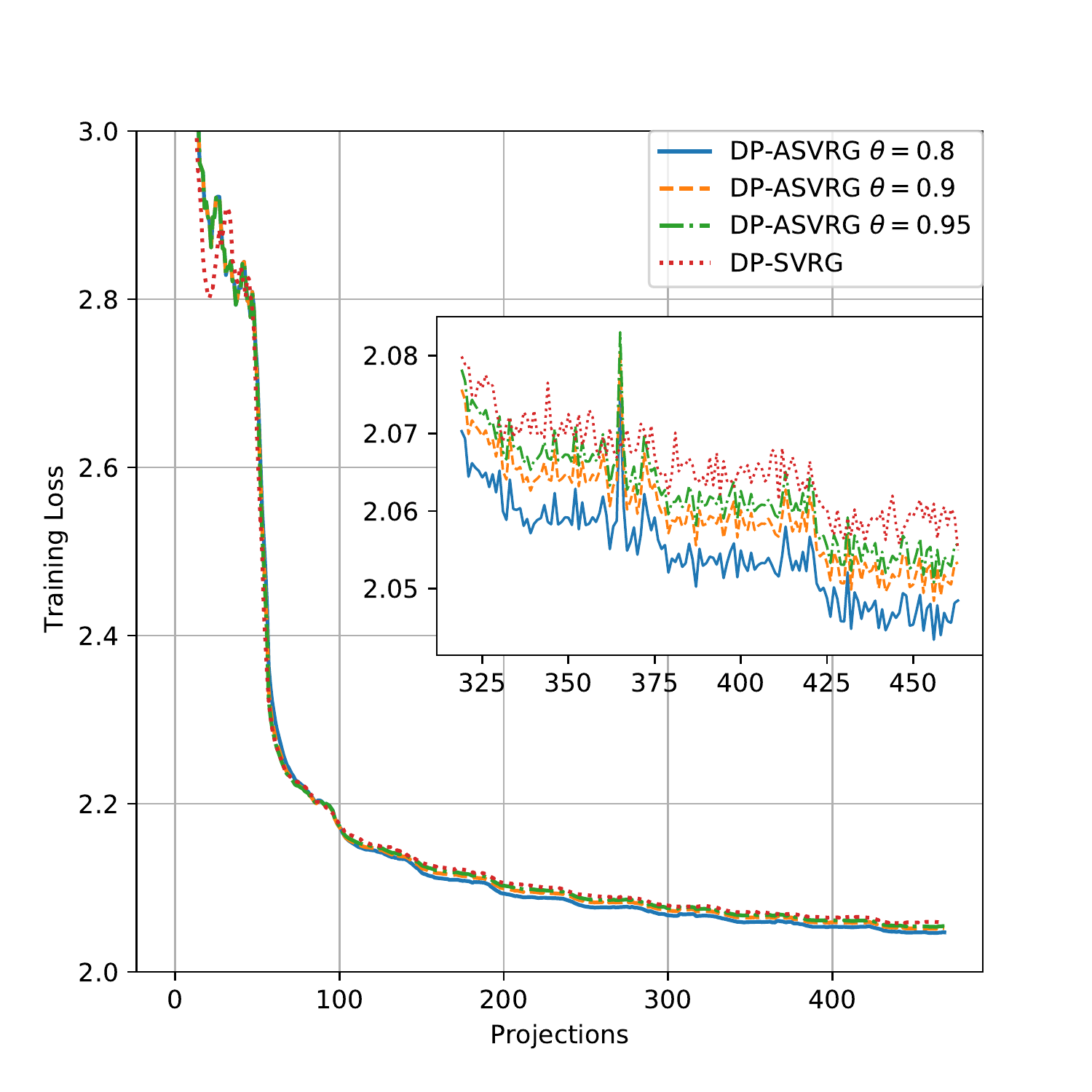} \label{fig:dp_b}}
	\hspace{-0.4in}
	\subfigure[Different $E$.]{\includegraphics[height=65mm] {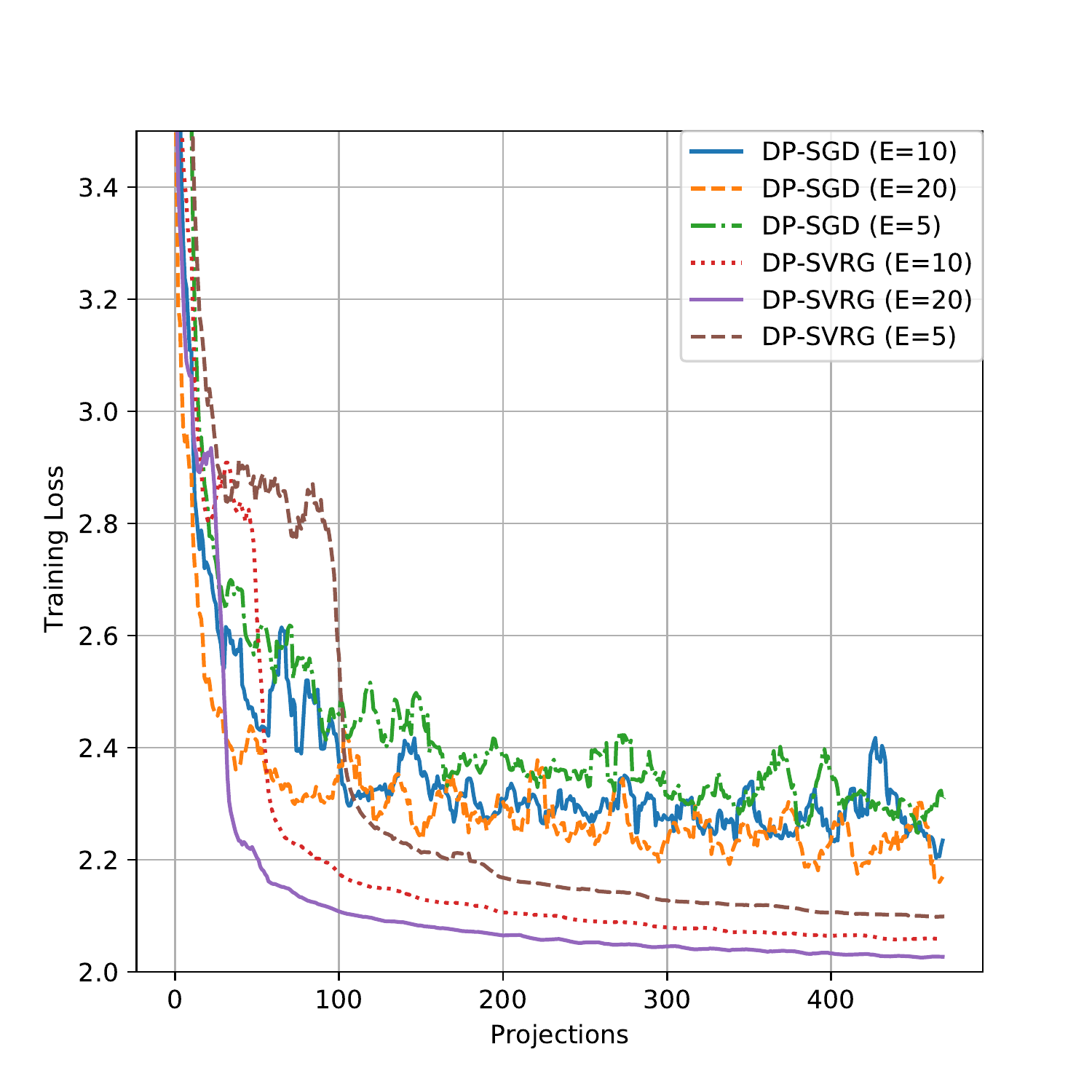} \label{fig:dp_c}} 
		\hspace{-1in}
	\caption{ Comparative results of methods with the delayed projection technique.}
	\label{fig:dp}
\end{figure*}

First, we consider a linear equality constrained logistic regression problem on MNIST dataset~\cite{lecun1998gradient}.
It requires us to classify $N=50,000$ handwriting number images into $10$ corresponding classes.
Here $p=7850, \eta = 0.1$ and weight decay is used to ensure the strongly convexity with $\mu = 10^{-4}$.
To impose restriction $\A^\top\x =\0$, we generate a $7850 \times 200$ matrix $\A$ with each entry generated as independent normal random variables.
To fasten computation, we then conduct Gram-Schmidt orthogonalization on $\A$ such that columns of $\A$ are mutually orthogonal.
Such $\A$ will not necessarily ensure Assumption~\ref{asmp:g_unconstrained}.

We compare three proposed methods with two baseline methods, Projected SGD (P-SGD) and Projected SVRG (P-SVRG), in projection complexity.
All methods start from the same initial point.
The batch size $b$ is set as 128, $\theta$ in DP-ASVRG is set as 0.9, and $m=N/b$ for all variance reduced methods.
The comparative results of training loss v.s. projection is shown in Figure~\ref{fig:dp_a}.
We find all methods with delayed projections are more efficient than the baselines since they reduce the training loss more given any budget of projection.
Variance reduced methods obtain smaller losses and have less fluctuation.
One interesting observation is DP-SVRG and its accelerated variant converges slowly than DP-SGD at the beginning and then decline more rapidly, reaching a smaller training loss.
This is because (DP-)SVRG is a multi-stage scheme algorithm.
At the first stage, little progress could be made due to the inaccurate snapshot model.
It also explains why P-SVRG is much slower than P-SGD; actually, after about 500 projections are performed (i.e., at the second stage), P-SVRG declines to a smaller loss than P-SGD finally fluctuates above.

From Figure~\ref{fig:dp_a}, DP-SVRG and DP-ASVRG with $\theta=0.9$ have similar convergence behaviors.
Indeed,  if $\theta = 1$, DP-SVRG is reduced to DP-SVRG except for the slight difference in the choice of stage-initial points.
It also implies that DP-ASVRG is not sensitive to the value of $\theta$, as suggested by Figure~\ref{fig:dp_b}.
Varying $\theta$ from $0.8$ to $0.95$, the convergence behaviors almost stay unchanged. 

We then explore how different projection interval $E$'s affect convergence.
Since DP-SVRG and DP-ASVRG have similar convergence behaviors, we only show the result of the formal for simplicity.
Figure~\ref{fig:dp_c} shows that large $E$ typically fastens convergence in terms of projections.
The observation is fitted well with established theories.
From Table~\ref{table:1}, the projection complexity of DP-SGD is $\OM(\frac{\sigmaat}{E \mu \eps})$ and that of DP-SVRG is $\widetilde{\OM}(\max\{ \kappa, \frac{N}{E} \})$, both not positively correlated with $E$.
It means increasing projection frequency indeed improves projection efficiency.

\subsection{Federated Logistic Regression}
\begin{figure*}[tp]
	\centering
	\hspace{-1in}
	\subfigure[Local methods with $E=10$]{\includegraphics[height=65mm] {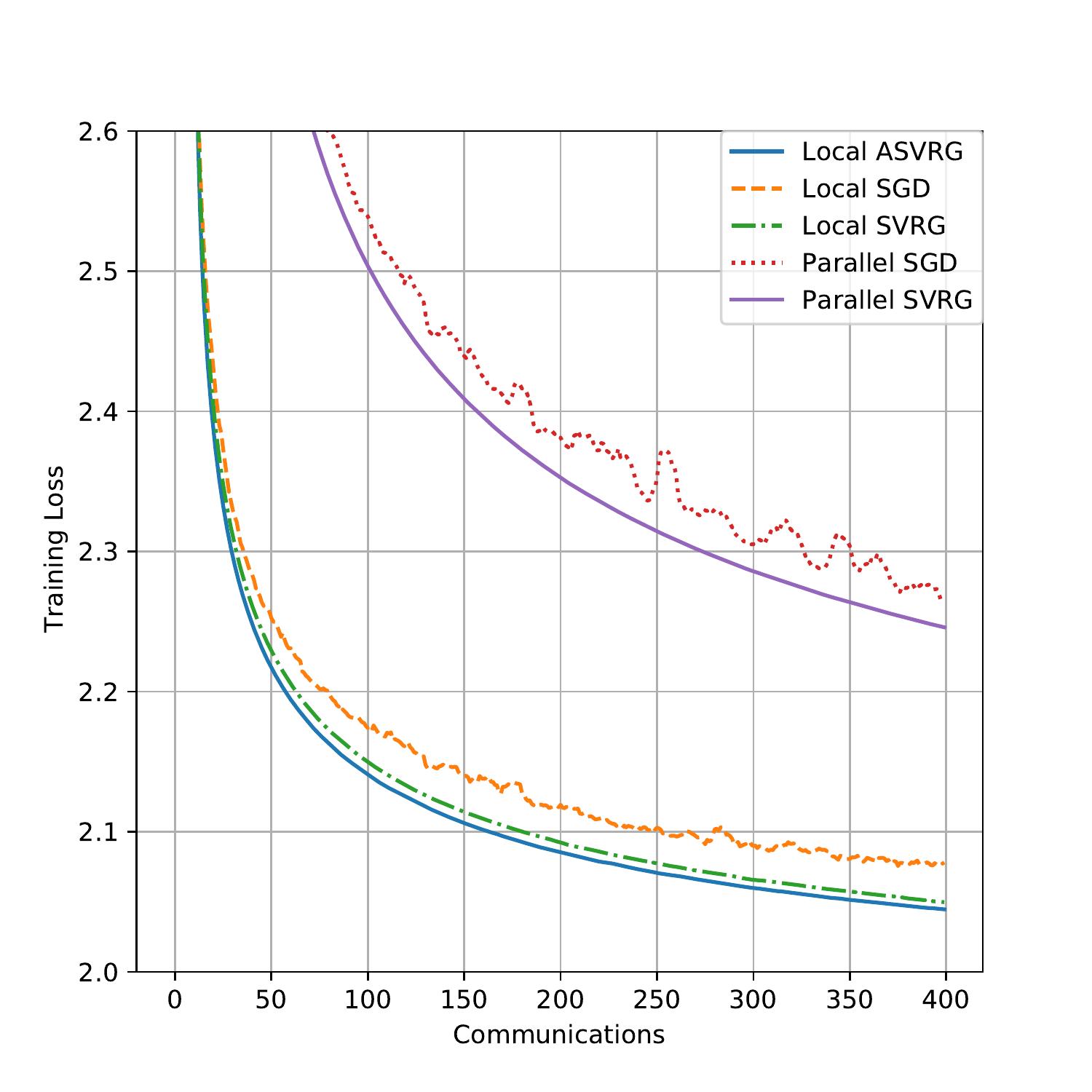} \label{fig:local_m_a}}
	\hspace{-0.4in}
	\subfigure[Local ASVRG with different $\theta$]{\includegraphics[height=65mm] {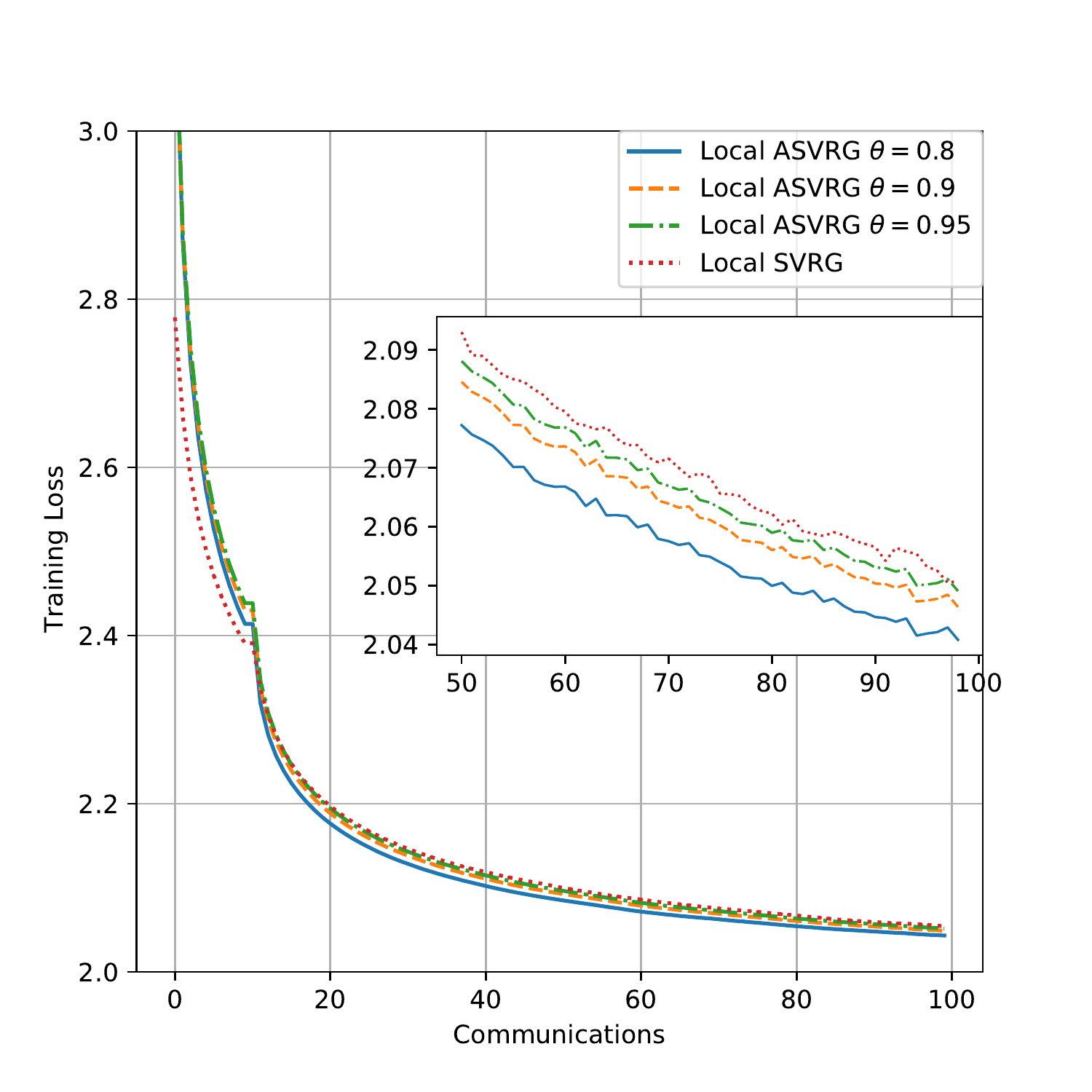} \label{fig:local_m_b}}
	\hspace{-0.4in}
	\subfigure[Different $E$.]{\includegraphics[height=65mm] {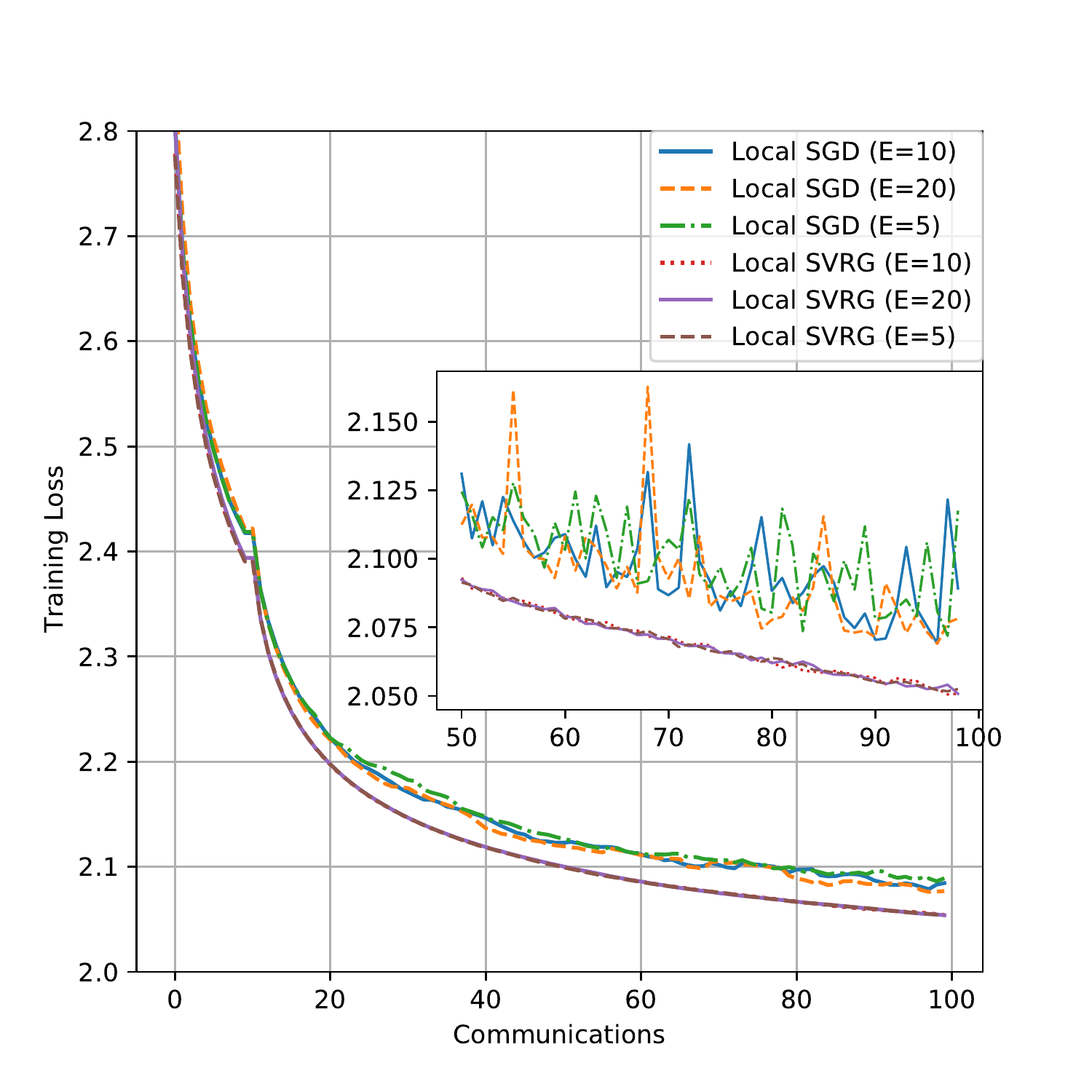} \label{fig:local_m_c}} 
	\hspace{-1in}
	\caption{ Comparative results of local methods with $m=40$.}
	\label{fig:local-m40}
\end{figure*}

\begin{figure*}[tp]
	\centering
	\hspace{-1in}
	\subfigure[Local methods with $E=10$]{\includegraphics[height=65mm] {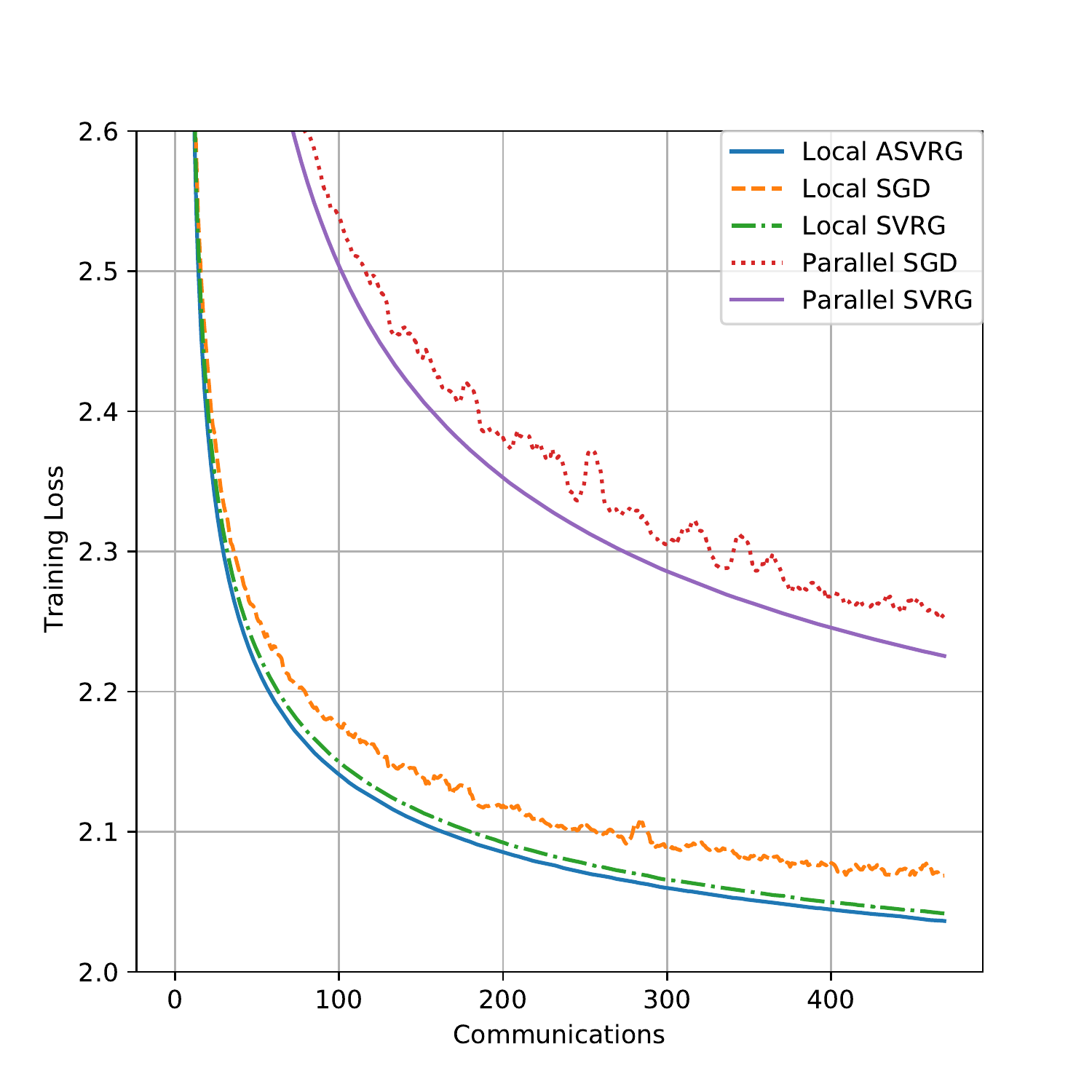} \label{fig:local_a}}
	\hspace{-0.4in}
	\subfigure[Local ASVRG with different $\theta$]{\includegraphics[height=65mm] {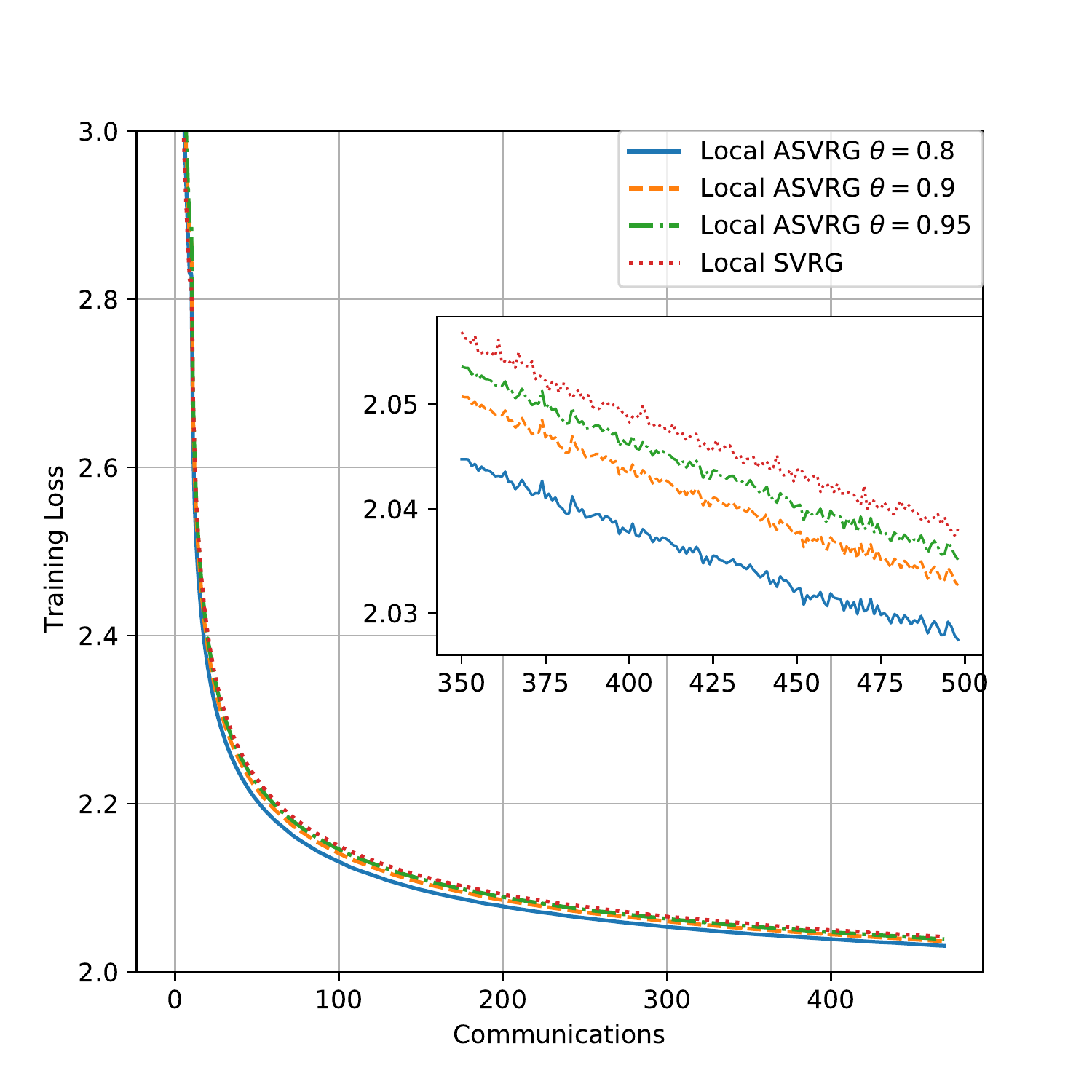} \label{fig:local_b}}
	\hspace{-0.4in}
	\subfigure[Different $E$.]{\includegraphics[height=65mm] {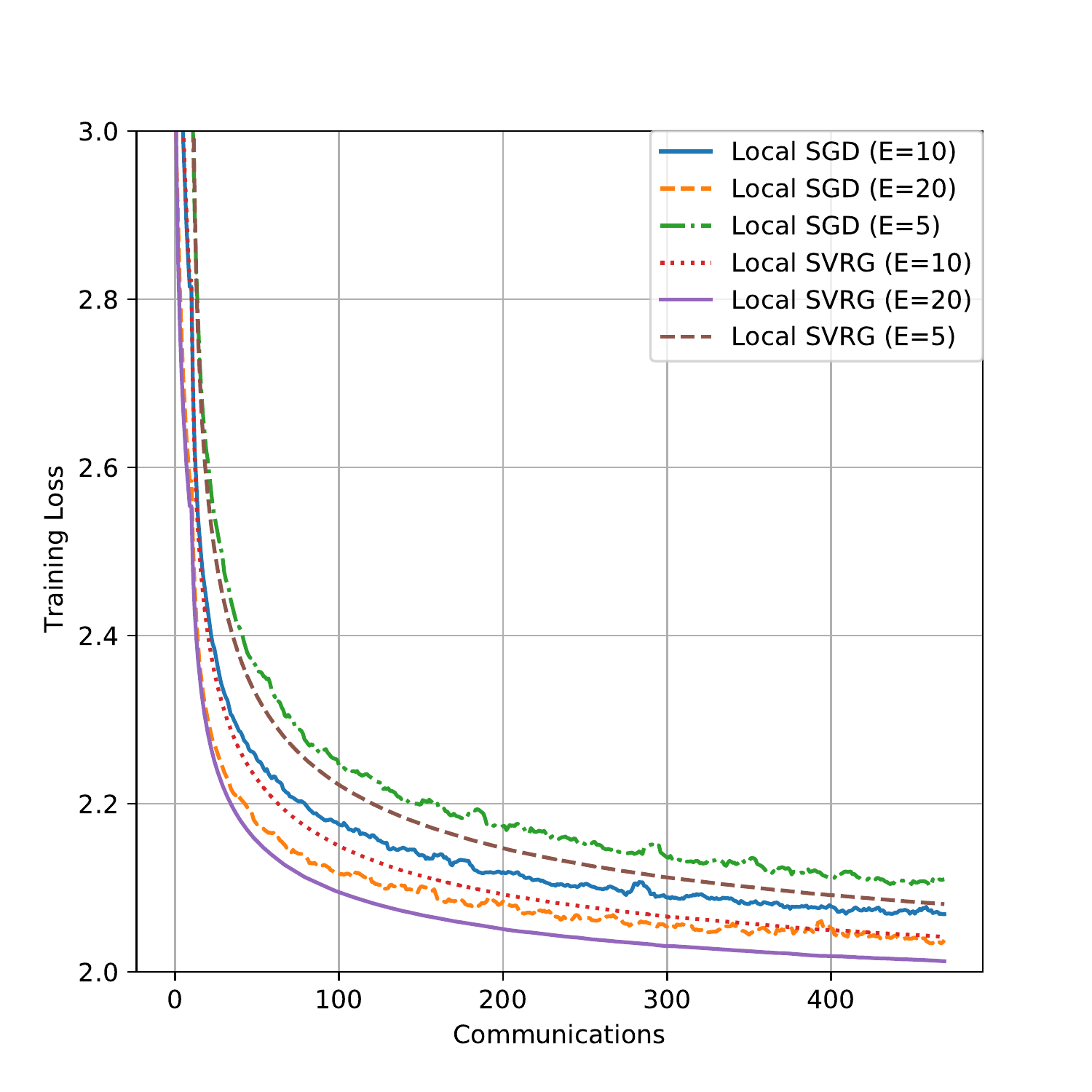} \label{fig:local_c}} 
	\hspace{-1in}
	\caption{ Comparative results of local methods with $m=E$.}
	\label{fig:local}
\end{figure*}

We conduct a logistic regression in a federated setting where the $N=50,000$ training data is evenly distributed into $n=10$ nodes with each node $N/n=5000$ samples.
To protect privacy, only intermediate variables are allowed to communicate.
The model dimension $d=7850$, learning rate $\eta = 0.1$, the batch size $b=128$, and weight decay is set as $\mu = 10^{-4}$ to ensure the strongly convexity.

We first fix $m=40$, which is about $\frac{N}{n b}$, vary $E$, and show the result in Figure~\ref{fig:local-m40}.
We then set $m=E$, an extreme case where nodes communicate with each other only at the end of each stage and show the results in Figure~\ref{fig:local}
We observe a similar convergence pattern in the last subsection, no matter in the fixed $m$ case or $m=E$ case.
Figure~\ref{fig:local_m_a} and~\ref{fig:local_a} show local methods are more efficient than the parallel baselines (which can be viewed as instances of local methods with $E=1$).
Figure~\ref{fig:local_m_b} and~\ref{fig:local_b} imply Local ASVRG is not sensitive to the value of $\theta$, though Local ASVRG converges slightly faster than Local SVRG, no matter what value $\theta$ is.
Figure~\ref{fig:local_m_c} and~\ref{fig:local_c} shows that large communication interval $E$ typically fastens convergence in terms of communication.
An obvious difference is that curves of $m=E$ distinguish from each other more than that of $m=40$.
We speculate the small differences between curves of $m=40$ is caused by too many inner loops.
Besides, we find that local update still fastens the convergence, even though theories derived in Section~\ref{sec:FL} imply communication complexity has nothing to do with $E$ when $m=E$.
It indicates increasing projection frequency improves projection efficiency more than the theories predict.

\section{Conclusion and Discussions}
In this work, we propose delayed projected SGD and two variance reduced variants for linearly constrained problems (LCPs).
We theoretically show it is possible to lower projection frequency and improve projection efficiency simultaneously.
Our analysis is simple and unified and can be extended to other delayed projected algorithms.
An important and natural question is how to extend delayed projected methods to more general cases.

\paragraph{Other feasible regimes.}
An important open problem is whether delayed projection techniques can work for other feasible regimes.
In our work, we main focus on the case where the domain is defined to be a linear space, i.e., $\mathcal{R}(\A^\perp) = \{\x: \A^\top\x = 0 \}$.
The main reason is the three nice properties given by projections into linear space, namely linearity, non-expansiveness, and orthogonality (see Proposition~\ref{prop:proj}).
All of them will be used frequently in our analysis, with linearity the most important.
A natural question is whether there are other regimes into which the projection preserves linearity.
Unfortunately, the following proposition states there is no other regime except a linear space satisfying the condition, implying the algorithm and technique proposed in our paper work only for linear spaces.
Hence, new methods and new techniques are required in order to design delayed projected methods for other regimes like the simplex or polygons.
Optimal transport~\cite{villani2008optimal} and reinforcement learning~\cite{sutton2018reinforcement} has the simplex constraint, while generalized lasso~\cite{gaines2018algorithms} has the polygon domain constraint.
Hence, it is an interesting and important open problem that whether delayed projection techniques work for inequality constrained problems.
\begin{prop}
	\label{prop:linear}
	Let $\AM \subset \RB^p$ be a region and $\PM_{\AM}(\x) = \argmin_{\y \in \AM}\|\y-\x\|^2$ be the projection onto it.
	If for any $\x, \y \in \RB^p$, $\PM_{\AM}(\x  + \y) = \PM_{\AM}(\x) + \PM_{\AM}(\y)$, then $\AM$ must be a linear (sub)space of $ \RB^p$.
\end{prop}

\paragraph{Other descent rules.}
In the three proposed algorithm, we only consider the stochastic gradient descent, i.e., $\x_{t+1} = \x_t - \eta_t \g_t$ where $\g_t$ is an unbiased estimator of the objective function evaluated at $\x_t$.
Due to the required smooth assumption, it could not be applied to lasso and its variant (see Example 2 in the introduction).
A remedy is to modify the update rule to $\x_{t+1} = \mathrm{prox}_{\eta_t h}(\x_t - \eta_t \g_t)$ where $h(\cdot)$ is some convex and non-smooth function like the $\ell_1$-norm in lasso.
Though our analysis can't apply anymore, but we believe the proof idea will still work that analyzing the incremental errors for $\EB\|\PAT(\x_t) - \xc\|^2$ and $\EB\|\PA(\x_t) \|^2$ and recurring the system they create.
We leave the combination of proximal operator and delayed projection techniques as future work.

Another extension is motivated by FedAvg~\cite{mcmahan2016communication,li2019convergence,konevcny2017stochastic}, a variant of LocalSGD that randomly activates a small portion of devices instead of all of them at beginning of each communication round.
One can check that its counterpart algorithm in LCPs is randomized block coordinate gradient descent (RBCGD)~\cite{richtarik2014iteration}.
We believe our technique can provide convergence analysis for the delayed projected variant of RBCGD.

\paragraph{Faster accelerated delayed projected methods.}
We have shown that when $E=1$, DP-ASVRG is reduced to P-ASVRG with $\widetilde{\OM}(\sqrt{\kappa})$ projection complexity.
We can show that the lower bound of projection complexity for a smooth strongly convex problem is $\widetilde{\Omega}(\sqrt{\kappa})$.
This follows by reducing distributed optimization to an instance of LCP and paralleling its theory.
In distributed optimization,~\cite{arjevani2015communication} provides a lower bound of communication rounds $\widetilde{\Omega}(\sqrt{\kappa})$ for smooth and $\mu$-strongly convex functions and $\Omega(\sqrt{\frac{L}{\eps}})$ for smooth convex functions.


However, once $E>1$, the best achievable projection complexity of DP-ASVRG deteriorates to $\widetilde{\OM}(\kappa^{2/3})$, based on Theorem~\ref{thm:acc_svrg_strong}.
We speculate this is caused by delayed projections since unprojected updates are often biased which might slow down the convergence rate.
However, the current lower bound is $\widetilde{\Omega}(\sqrt{\kappa})$ as argued.
Can we design a delayed projected method that both lets $E>1$ and achieves this lower bound of projection numbers?
We left it as an open problem.

%

\section*{Acknowledgement}
The authors want to thank Guangzeng Xie, Wenhao Yang, and Dachao Lin for helpful discussion on some inequalities and thank Prof. Weijie Su and Prof. Shusen Wang for helpful suggestions on an earlier version of this manuscript.

\bibliography{bib/optimization,bib/federated,bib/personalization,bib/distributed}
\bibliographystyle{plain}
\appendix 
\begin{appendix}
	\onecolumn
	\begin{center}
		{\huge \textbf{Appendix}}
	\end{center}

\section{Useful Lemmas}

\begin{lem}
	\label{lem:ineq1}
	For any vectors $\a, \bb$ and positive number $\gamma > 0$, it follows that
	\[
	2\langle \a, \bb \rangle \le \gamma \|\a\|^2 + \frac{1}{\gamma}\|\bb\|^2.
	\]
\end{lem}

\begin{lem}
	\label{lem:j}
	For any set of (random) vectors $\x_1, \x_2, \cdots, \x_n \in \EB^p$, it follows that
	\[  \big\| \sum_{i=1}^n \x_i\big\|^2 \le n  \sum_{i=1}^n \big\| \x_i\big\|^2.
	\] 
\end{lem}

\begin{lem}
	\label{lem:F}
	Let $F(\x)$ satisfy Assumption~\ref{asmp:smooth} and~\ref{asmp:strong}, then
	\begin{enumerate}
		\item (Jensen's inequality) For any $\x_1, \cdots, \x_E$, it follows that $F\left(\frac{1}{E}\sum_{t=1}^E \x_t\right) \le \frac{1}{E}\sum_{t=1}^E F\left( \x_t \right)$.
		\item For any $\x, \y \in \RB^p$, 
		$ \|\nabla F(\x;\xi) - \nabla F(\y; \xi) \|^2 \le 2 L \left[F(\x; \xi) - F(\y; \xi) -\langle\nabla F(\y; \xi), \x-\y \rangle \right] $;
		\item For $\x \in \mathcal{R}(\A^\perp)$, $\EB_{\xi} \|\nabla F(\x;\xi) - \nabla F(\xc; \xi) \|^2 \le 2L \EB_{\xi}\left[ F(\x) - F(\xc)\right]$
		where $\xc = \argmin_{\A^\top \x = \0}F(\x)$.
		\item $\langle \nabla F(\x) - \nabla F(\y), \x -\y \rangle \ge \mu \|\x -\y\|^2$.
	\end{enumerate}
\end{lem}
\begin{proof}
	\begin{enumerate}
		\item See Theorem 2.1.2 of Nesterov~\cite{nesterov2013introductory} for a proof of case $E=2$. Induction for $E \ge 2$.
		\item See Theorem 2.1.5 of the textbook of Nesterov~\cite{nesterov2013introductory}.
		\item Since $F(\x) = \EB_{\xi} \nabla F(\x;\xi)$, using results of the first item, we only need to prove $\EB \langle\nabla F(\xc; \xi), \x-\xc \rangle = 0$.
		Note that when $\x \in \mathcal{R}(\A^\perp)$, $\x-\xc \in \mathcal{R}(\A^\perp)$.
		By the optimality of $\xc$, $\EB_{\xi}\nabla F(\xc; \xi) = \nabla F(\xc) \in \mathcal{R}(\A)$.
		Then the result follows.
		\item By the $\mu$-strongly convexity of $F(\cdot)$, it follows that
		\begin{gather*}
		F(\x) - F(\y) - \langle  \nabla F(\y), \x - \y \rangle \ge \frac{\mu}{2}\|\x -\y\|^2\\
		F(\y) - F(\x) - \langle  \nabla F(\x), \y - \x \rangle \ge \frac{\mu}{2}\|\x -\y\|^2.
		\end{gather*}
		Adding the above two inequalities gives the result.
	\end{enumerate}
\end{proof}

\section{Appendix for Section~\ref{sec:notation} and Discussion}
The following lemma proves Proposition~\ref{prop:proj} and Proposition~\ref{prop:linear} together.
\begin{lem}[Property of projection]
	Let $\AM \subset \RB^p$ be some closed convex region, $\PM_{\AM}(\x) = \argmin_{\y \in \AM}\|\y-\x\|^2$ and $\PM_{\AM^\perp}(\x) = \x - \PM_{\AM}(\x)$.
	\begin{enumerate}
		\item  $\PM_{\AM}$ is idempotent: $\PM_{\AM}(\PM_{\AM}(\x)) = \PM_{\AM}(\x)$ for all $\x \in \RB^p$;
		\item $\PM_{\AM}$ is non-expansive: for $\x \in \RB^p$ and $\y \in \AM$, $\|\PM_{\AM}(\x)-\x\|^2 + \|\PM_{\AM}(\x)-\y\|^2 \le \|\x-\y\|^2$.
		In particular, if $\0 \in \AM$, $\max\{\|\PM_{\AM}(\x)\|, \|\PM_{\AM^\perp}(\x)\|  \} \le \|\x\|$.
		\item If $\AM$ is a linear space, for example $\AM = \mathcal{R}(\A)$, then  $\PM_{\AM}$ is linear in $\x$: $\PM_{\AM}(\alpha\x + \beta\y) = \alpha\PM_{\AM}(\x) + \beta\PM_{\AM}(\y)$ for any $\x, \y \in \RB^p$ and $\alpha, \beta \in \RB$; and $ \langle\PM_{\AM}(\x),  \PM_{\AM^\perp}(\x)\rangle = 0 $ for any $\x \in \RB^p$. 
		\item
		If $\PM_{\AM}(\x  + \y) = \PM_{\AM}(\x) + \PM_{\AM}(\y)$ for any $\x, \y \in \RB^p$, then $\AM$ must be a linear (sub)space in $\RB^p$.
	\end{enumerate}
\end{lem}
\begin{proof}
		\begin{enumerate}
			\item This follows directly from definition.
			\item Since $\PM_{\AM}(\x) = \argmin_{\y \in \AM}\|\y-\x\|^2$, the first order optimality condition gives $\langle \PM_{\AM}(\x) - \x, \PM_{\AM}(\x) - \y \rangle \le 0$ for any $\y \in \AM$.
			Hence,
			\[
			\|\x-\y\|^2-\|\PM_{\AM}(\x)-\y\|^2 -\|\PM_{\AM}(\x)-\x\|^2
			=2\left\langle \x - \PM_{\AM}(\x), \PM_{\AM}(\x) - \y \right\rangle \ge 0.
			\]
			\item If $\AM$ is a linear space, $\PM_{\AM}(\x) = \A(\A^\top \A)^{\dag}\A^\top \x$ and thus the linearity follows.
			Besides, $\PM_{\AM^\perp}(\x) = (\I-\A(\A^\top \A)^{\dag}\A^\top) \x$ and $(\I-\A(\A^\top \A)^{\dag}\A^\top)\A(\A^\top \A)^{\dag}\A^\top = \0 $ account for the rest.
			\item Obviously, letting $\x=\y=\0$, we have $\PM_{\AM}(\0) = \0$ and thus $\0 \in \AM$.
			So by the second item, $\PM_{\AM}(\cdot)$ is continuous.
			By reduction, we know that $\PM_{\AM}(k\x)= k \PM_{\AM}(\x)$ for any positive integer $k$ and $\x \in \RB^p$.
			Thus, we have $\PM_{\AM}(q\x)= q\PM_{\AM}(\x)$ for any positive rational number $q$ and $\x \in \RB^p$.
			Since $\PM_{\AM}(q\x) + \PM_{\AM}(-q\x) = 0$, $\PM_{\AM}(q\x)= q\PM_{\AM}(\x)$ for any rational number $q$ and $\x \in \RB^p$.
			By continuity of $\PM_{\AM}(\cdot)$, $\PM_{\AM}(\alpha \x)= \alpha \PM_{\AM}(\x)$ for any real number $\alpha$ and $\x \in \RB^p$.
			Hence, for any $\x, \y \in \AM$, we have $\PM_{\AM}(\alpha\x + \beta\y) 
			= \PM_{\AM}(\alpha\x) + \PM_{\AM}(\beta\y)  = \alpha\PM_{\AM}(\x) + \beta\PM_{\AM}(\y) = \alpha \x + \beta \y \in \AM$ for any $\alpha, \beta \in \RB$, which implies that $\AM$ is a linear (sub)space in $\RB^p$.
		\end{enumerate}
\end{proof}

\begin{proof}[Proof of Corollary~\ref{cor:xc}]
	If there are two $\x^*, \y^* \in \mathcal{R}(\A^\perp)$ that both minimize $F(\cdot)$ within the linear constraint $\A^\top \x^* = \A^\top \y^* = \0$, then we assert that we have $(\A^{\perp})^\top \nabla F(\x^*)=\0$, which implies $\PAT (\nabla F(\x^*))=\0$.
	To see this, let $\xc$ be any solution of the constrained problem, it is often the case that $\nabla F(\xc) \neq \0$. 
	Since $\xc \in  \mathcal{R}(\A^\perp)$, then $\xc = \A^\perp\z^*$ and $\z^* = \argmin_{\z} F(\A^\perp\z)$.
	By the first order condition of $\z^*$, we have $(\A^{\perp})^\top \nabla F(\x^*)=\0$.
	
	So $\y^*-\x^* \in \mathcal{R}(\A^\perp)$ and $\nabla F(\x^*) \in \mathcal{R}(\A)$.
	Then from the $\mu$-strongly convexity, $0 = F(\y^*) - F(\x^*) \ge \langle\nabla F(\x^*), \y^*-\x^* \rangle + \frac{\mu}{2} \| \x^*-\y^*\|^2 = \frac{\mu}{2} \| \x^*-\y^*\|^2 \ge 0$, indicating that $\x^* = \y^*$.
\end{proof}

\section{Proof of Delayed Projected SGD}

\subsection{Descent Lemma}
\begin{lem}[Bounded gradient variance]
\label{lem:var_localsgd}
Under Assumption~\ref{asmp:smooth}, ~\ref{asmp:strong} and~\ref{assum:bounded},
\begin{equation}
\label{eq:second_y}
\EB \| \PAT(\nabla F(\x_{t}; \xi_t)- \nabla F(\x_t))\|^2  \le 3L^2\EB \| \z_t \|^2 + 6 L\EB\left[ F(\y_t) - F(\xc)\right] + 3\sigmaat.
\end{equation}
\end{lem}
\begin{proof}
	\begin{align*}
	&\EB \| \PAT(\nabla F(\x_{t}; \xi_t))-\PAT( \nabla F(\x_t))\|^2 \nonumber \\
	\overset{(a)}{=}&\EB \| \PAT(\nabla F(\x_{t}; \xi_t))- \nabla F(\x_t)) \|^2\nonumber  \\
	=& \EB \|  \PAT \left(\left[ \nabla F(\x_{t}; \xi_t)-\nabla F(\y_{t}; \xi_t) - \left( \nabla F(\x_t) -\nabla F(\y_t) \right)  \right] \right) \nonumber \\
	&+\PAT \left( \left[ \nabla F(\y_{t}; \xi_t) - \nabla F(\xc; \xi_t)  - \left(\nabla F(\y_{t})-\nabla F(\xc)    \right)  \right]\right)  \nonumber \\
	&+\PAT \left( \nabla F(\xc; \xi_t)  - \nabla F(\xc)  \right) \|^2 \nonumber  \\
	\overset{(b)}{\le}& 
	3\EB \|  \nabla F(\x_{t}; \xi_t)-\nabla F(\y_{t}; \xi_t) 
	- \left( \nabla F(\x_t) -\nabla F(\y_t) \right)   \|^2 \nonumber \\
	&+3 \EB \| \nabla F(\y_{t}; \xi_t) - \nabla F(\xc; \xi_t)  - \left(\nabla F(\y_{t})-\nabla F(\xc)    \right) \|^2\nonumber \\
	& + 3 \EB \| \PAT(\nabla F(\xc; \xi_t)  - \nabla F(\xc))\|^2 \nonumber \\
	\overset{(c)}{\le}& 3\EB \|  \nabla F(\x_{t}; \xi_t)-\nabla F(\y_{t}; \xi_t) \|^2 + 3 \EB \| \nabla F(\y_{t}; \xi_t) - \nabla F(\xc; \xi_t)\|^2 \nonumber \\
	&+ 3 \EB \|\PAT(\nabla F(\xc; \xi_t))\|^2 \nonumber \\
	\overset{(d)}{\le}& 3L^2\EB \| \z_t \|^2 + 6 L\EB\left[ F(\y_t) - F(\xc)\right] + 3\sigmaat
	\end{align*}
	where (a) follows from the linearity of $\PA$; (b) is due to Lemma~\ref{lem:j} with $n=3$ and $\|\PAT(\x)-\PAT(\y)\| \le \|\x-\y\|$; 
	(c) follows by noting $\EB\| X - \EB X\|^2 \le \EB \|X\|^2$ and that $\PAT (\nabla F(\x^*))=\0$ from Corollary~\ref{cor:xc}; and (d) follows from Assumption~\ref{asmp:smooth} and~\ref{assum:bounded}.
\end{proof}

\begin{lem}[General descent lemma on $\mathcal{R}(\A^\perp)$]
	\label{lem:y_descent_general}
	Let $\y_t = \PAT (\x_t)$ be the projection onto $\mathcal{R}(\A^\perp)$ and $\z_t = \PA(\x_t)$ the projection onto $\mathcal{R}(\A)$.
	Let $\x_{t+1} = \x_t - \eta_t \g_t$ be the update rule with $\EB_{\xi_t} \PAT(\g_t) = \PAT(\nabla F(\x_t))$ and $\g_t$ is independent with all randomness before iteration $t$, then 
\begin{align}
	\EB \|\y_{t+1}-\xc\|^2 
		&\le (1-\mu \eta_t)  \EB \| \y_t-\xc\|^2  +  (4L\eta_t^2-2\eta_t) \EB\left[ F(\y_t) - F(\xc)\right]   + (L\eta_t + 2L^2\eta_t^2)\EB \| \z_t \|^2 
		\nonumber \\
& \qquad \qquad +\eta_t^2  \EB \| \PAT(\g_t) -\PAT( \nabla F(\x_t))\|^2.
\end{align}
\end{lem}
\begin{proof}
	Since the optimal constrained solution $\xc$ lies in $\mathcal{R}(A^\perp)$, $\PAT(\xc) = \xc$ and $\PA(\xc) = \0$.
	Let $\y_t = \PAT (\x_t)$ be the projection onto $\mathcal{R}(A^\perp)$ and $\z_t = \PA(\x_t)$ the projection onto $\mathcal{R}(A)$, then $\x_t = \y_t + \z_t$ and $\langle \y_t, \z_t \rangle = 0$.
	Taking expectation with respect to $\x_1, \cdots, \x_t$, we have
	\begin{align}
	\label{eq:y_zhankai}
	\EB \|\y_{t+1}-\xc\|^2
	&= \EB \| \PAT (\x_{t}-\eta_t \g_t)-\xc\|^2 \nonumber \\
	&= \EB \| \y_t-\eta_t \PAT (\nabla F(\x_{t}) )-\xc\|^2 + \eta_t^2 \EB \| \PAT(\g_t)-\PAT(\nabla F(\x_t))\|^2.
	\end{align}

	For the first term of~\eqref{eq:y_zhankai},
	\begin{align}
	\label{eq:dy_zhankai}
	\EB \| \y_t-\eta_t \PAT (\nabla F(\x_{t}) )-\xc\|^2 
	&=\EB \| \y_t-\xc\|^2  -2\eta_t \langle\PAT (\nabla F(\x_{t}) ), \y_t-\xc  \rangle  + \eta_t^2 \EB \| \PAT (\nabla F(\x_{t}) )\|^2 \nonumber \\
	&=\EB \| \y_t-\xc\|^2  -2\eta_t \langle \nabla F(\x_{t}), \y_t-\xc  \rangle  + \eta_t^2 \EB \| \PAT (\nabla F(\x_{t}) )\|^2.
	\end{align}
	For the second term of~\eqref{eq:dy_zhankai},
	\begin{align}
	\label{eq:second_dy}
	- \langle \nabla F(\x_{t}), \y_t-\xc  \rangle 
	&= -  \langle \nabla F(\x_{t}), \x_t-\xc  \rangle  -  \langle \nabla F(\x_{t}), \y_t-\x_t  \rangle \nonumber \\
	&\overset{(a)}{\le} - \left[  F(\x_t) - F(\xc) + \frac{\mu}{2} \| \x_t - \xc\|^2 \right] +  \langle \nabla F(\x_{t}), \x_t -\y_t \rangle \nonumber \\ 
	&\overset{(b)}{\le} - \left[  F(\x_t) - F(\xc) + \frac{\mu}{2} \| \x_t - \xc\|^2 \right] + 
	\left[ F(\x_t) - F(\y_t) + \frac{L}{2} \|\y_t - \x_t\|^2 \right] \nonumber \\
	&\overset{(c)}{=} - \left[  F(\y_t) - F(\xc) + \frac{\mu}{2} \| \x_t - \xc\|^2 \right] + \frac{L}{2} \|\z_t\|^2 \nonumber \\
	&\overset{(d)}{=} - \left[  F(\y_t) - F(\xc) + \frac{\mu}{2} \| \y_t - \xc\|^2 \right] + \frac{L-\mu}{2} \| \z_t\|^2
	\end{align}
	where (a) follows from the $\mu$-strongly convexity; (b) follows from the $L$-smoothness; (c) is due to arrangement and the decomposition $\x_t =\PAT(\x_t) + \PA(\x_t) =\y_t + \z_t$; and (d) holds since $\x_t - \xc = \PA(\x_t - \xc) + \PAT(\x_t - \xc) =\z_t +  (\y_t- \xc)$ and $\|\x_t - \xc\|^2 = \|\z_t \|^2+  \|\y_t- \xc\|^2$.
	For the third term of~\eqref{eq:dy_zhankai},
	\begin{align}
	\label{eq:third_dy}
	\EB \| \PAT (\nabla F(\x_{t}) )\|^2&
	=\EB \| \PAT (\nabla F(\x_{t}) ) - \PAT (\nabla F(\y_t)) + \PAT (\nabla F(\y_t))- \PAT (\nabla F(\xc) )\|^2 \nonumber \\
	& \le 2\EB \| \PAT (\nabla F(\x_{t}) -\nabla F(\y_t)) \|^2+2\EB\| \PAT (\nabla F(\y_t) - \nabla F(\xc) )\|^2\nonumber\\
	& \le 2\EB \|  \nabla F(\x_{t}) -\nabla F(\y_t) \|^2+2\EB\| \nabla F(\y_t) - \nabla F(\xc)  \|^2 \nonumber \\
	&\le 2 L^2 \EB \| \z_t\|^2  + 4L \left[ F(\y_t) - F(\xc)  \right].
	\end{align}
	
	By substituting all the inequalities, we have
	\begin{align*}
	\EB &\|\y_{t+1}-\xc\|^2\\
	&\overset{\eqref{eq:y_zhankai}}{=} \EB \| \y_t-\eta_t \PAT (\nabla F(\x_{t}) )-\xc\|^2 + \eta_t^2 \EB \| \PAT(\nabla F(\x_{t}; \xi_t))-\PAT(\nabla F(\x_t))\|^2 \\
	&\overset{\eqref{eq:dy_zhankai}}{=} \EB \| \y_t-\xc\|^2  -2\eta_t \langle \nabla F(\x_{t}), \y_t-\xc  \rangle  + \eta_t^2 \EB \| \PAT (\nabla F(\x_{t}) )\|^2 \\
	& \qquad + \eta_t^2 \EB \| \PAT(\nabla F(\x_{t}; \xi_t))-\PAT(\nabla F(\x_t))\|^2 \\
	&\overset{\eqref{eq:second_dy},\eqref{eq:third_dy}}{\le}
	\EB \| \y_t-\xc\|^2  -2\eta_t\EB \left[  F(\y_t) - F(\xc) + \frac{\mu}{2} \| \y_t - \xc\|^2 \right] + \eta_t(L-\mu) \EB \| \z_t\|^2 \\
	& \qquad \qquad + \eta_t^2 \left[ 2 L^2 \EB \| \z_t \|^2  + 4L \EB\left[ F(\y_t) - F(\xc)  \right]  \right]  \\
	& \qquad \qquad + \eta_t^2  \EB \| \PAT(\nabla F(\x_{t}; \xi_t))-\PAT(\nabla F(\x_t))\|^2\\
	&\le (1-\mu \eta_t)  \EB \| \y_t-\xc\|^2  +  (4L\eta_t^2-2\eta_t) \EB\left[ F(\y_t) - F(\xc)\right]   + (L\eta_t + 2L^2\eta_t^2)\EB \| \z_t \|^2\\
	& \qquad \qquad +\eta_t^2  \EB \| \PAT(\g_t)-\PAT(\nabla F(\x_t))\|^2
	\end{align*}
\end{proof}

We are now ready to prove the decent lemma  for Algorithm~\ref{alg:multi}.
\begin{proof}[Proof of Lemma~\ref{lem:y_descent}]
	Lemma~\ref{lem:y_descent_general} with $\g_t = \nabla F(\x_t; \xi_{t})$ gives that 
	\begin{align*}
		\EB \|\y_{t+1}-\xc\|^2
		&\le (1-\mu \eta_t)  \EB \| \y_t-\xc\|^2  +  (4L\eta_t^2-2\eta_t) \EB\left[ F(\y_t) - F(\xc)\right]   + (L\eta_t + 2L^2\eta_t^2)\EB \| \z_t \|^2\\
	& \qquad \qquad +\eta_t^2  \EB \| \PAT(\nabla F(\x_{t}; \xi_t))-\PAT(\nabla F(\x_t))\|^2.
	\end{align*}
	Lemma~\ref{lem:var_localsgd} gives
	\[
	\EB \| \PAT(\nabla F(\x_{t}; \xi_t)- \nabla F(\x_t))\|^2  \le 3L^2\EB \| \z_t \|^2 + 6 L\EB\left[ F(\y_t) - F(\xc)\right] + 3\sigmaat.
	\]
	Combing the last two inequalities, we get
	\begin{align*}
	\EB &\|\y_{t+1}-\xc\|^2\\
	&\le (1-\mu \eta_t)  \EB \| \y_t-\xc\|^2 + (10L\eta_t^2-2\eta_t) \EB\left[ F(\y_t) - F(\xc)\right]  + 3  \eta_t^2 \sigmaat + (L\eta_t + 5L^2\eta_t^2)\EB \| \z_t \|^2\\
	&\le (1-\mu \eta_t)  \EB \| \y_t-\xc\|^2 -\eta_t \EB\left[ F(\y_t) - F(\xc)\right]  + 3  \eta_t^2 \sigmaat + 2L\eta_t\EB \| \z_t\|^2
	\end{align*}
		where for the final line we used that $\eta_t \le \frac{1}{10L}$.
\end{proof}

\subsection{Residual Lemma}
\begin{proof}
	Note that
	\begin{align}
	\label{eq:z_zhankai}
	\EB \|\z_{t+1}\|^2 
	&= \EB \big\| \PA(\x_t -  \eta_t \nabla F(\x_{t}; \xi_{t})) \big\|^2 \nonumber \\
	&= \EB \| \PA(\x_t) \|^2  -2\eta_t \EB \langle\PA(\x_t),  \PA( \nabla F(\x_{t};\xi_{t}))\rangle + \eta_t^2  \EB \| \PA(\nabla F(\x_{t}; \xi_{t})) \|^2 \nonumber \\
	&= \EB \| \z_t \|^2  -2\eta_t \EB \langle\z_t, \PA( \nabla F(\x_{t}))\rangle + \eta_t^2  \EB \| \PA(\nabla F(\x_{t}; \xi_{t})) \|^2.
	\end{align}
	where we use $\z_t = \PA(\x_t)$ in the last equality and $\EB \PA(\nabla F(\x_{t}; \xi_{t}))  = \PA(\nabla F(\xi_{t}))$.

	With Assumption~\ref{asmp:g_unconstrained}, by the $\mu$-strongly convexity of $F(\cdot)$, we have for the second term of~\eqref{eq:z_zhankai} 
	\begin{align}
	\label{eq:second_z}
	- \langle \z_t,  \PA( \nabla F(\x_{t}))\rangle
	&=- \langle \PA( \z_t),  \nabla F(\x_{t})\rangle
	=- \langle \z_t,  \nabla F(\x_{t})\rangle \nonumber \\
	&= - \langle \x_t - \y_t,  \nabla F(\x_{t})\rangle \nonumber \\
	&\le - \left[  F(\x_t)  - F(\y_t) + \frac{\mu}{2} \| \x_t - \y_t\|^2 \right]
	\end{align}
	For the third term of~\eqref{eq:z_zhankai}, if we have Assumption~\ref{asmp:g_unconstrained},
	\begin{align}
	\label{eq:third_z}
	\EB \|\PA( \nabla F(\x_{t}; \xi_{t})) \|^2
	&\le 2\EB \|\PA( \nabla F(\x_{t}; \xi_{t}) - \nabla F(\xc; \xi_{t})) \|^2
	+  2\EB \|\PA( \nabla F(\xc; \xi_{t})) \|^2 \nonumber \\
	&\le  2\EB \|\nabla F(\x_{t}; \xi_{t}) - \nabla F(\xc; \xi_{t}) \|^2
	+ 2\sigmaa \nonumber \\
	&\le 4L \EB \left[ F(\x_t;\xi_{t}) - F(\xc;\xi_{t})  - \langle  \nabla F(\xc;\xi_{t}) , \x_t - \xc \rangle \right] + 2\sigmaa \nonumber \\
	&\le 4L\EB\left[ F(\x_t) - F(\xc) \right]+ 2\sigmaa
	\end{align}
	where in the last inequality we use $\EB  \nabla F(\xc;\xi_{t})  = \0$ derived from Assumption~\ref{asmp:g_unconstrained}.
	Then, based on~\eqref{eq:z_zhankai},~\eqref{eq:second_z} and~\eqref{eq:third_z}, when $\eta_t \le \frac{1}{2L}$, we have
	\begin{align}
	\label{eq:z_descent_00}
	\EB \|\z_{t+1}\|^2 &\le
	(1-\mu \eta_t) \EB \|\z_{t}\|^2 + 2\eta_t^2 \sigmaa + 4L\eta_t^2\EB\left[ F(\x_t) - F(\xc) \right] + 2\eta_t\EB\left[  F(\y_t)  - F(\x_t) \right] \nonumber \\
	&\le
	(1-\mu \eta_t) \EB \|\z_{t}\|^2 + 2\eta_t^2 \sigmaa +  2\eta_t\EB\left[  F(\y_t)  - F(\xc) \right]. 
	\end{align}
	Here we also use $\EB\left[ F(\x_t) - F(\xc) \right] \ge \frac{1}{2L} \EB \|\nabla F(\x_{t}; \xi_{t}) - \nabla F(\xc; \xi_{t}) \|^2 \ge 0$.

	Without Assumption~\ref{asmp:g_unconstrained}, then for the second term of~\eqref{eq:z_zhankai}, we have
	\begin{align}
	\label{eq:second_z_no}
	- \langle \z_t,  \PA( \nabla F(\x_{t}))\rangle
	&= - \langle \x_t - \y_t,  \nabla F(\x_{t})\rangle \nonumber \\
	&= - \langle \x_t - \y_t,  \nabla F(\x_{t})-\nabla F(\y_{t})\rangle  - \langle \x_t - \y_t, \nabla F(\y_{t})\rangle \nonumber \\
	&\overset{(a)}{\le} - \mu \|\z_t\|^2 + | \langle \z_t, \nabla F(\y_{t})\rangle|  \nonumber \\
	&\overset{(b)}{\le} - \mu \|\z_t\|^2 +  \frac{\gamma}{2}\|\z_t\|^2 + \frac{1}{2\gamma} \|\nabla F(\y_t)\|^2  \nonumber\\
	&\overset{(c)}{=}\left(-\frac{\mu}{2} + \frac{1-\mu\eta}{2E\eta}\right) \|\z_t\|^2  + \frac{1}{2\left(\mu + \frac{1-\mu\eta}{E\eta}\right)} \|\nabla F(\y_t)\|^2  \nonumber\\
	&\le \left(-\frac{\mu}{2} + \frac{1-\mu\eta}{2E\eta}\right) \|\z_t\|^2  + \frac{E\eta}{2} \|\nabla F(\y_t)\|^2 
	\end{align}
	where (a) uses the $\mu$-strongly convexity of $F(\cdot)$ that implies $\langle \nabla F(\x_t) - \nabla F(\y_t), \x_t -\y_t \rangle \ge \mu \|\x_t -\y_t\|^2$ (Lemma~\ref{lem:F}); (b) uses Lemma~\ref{lem:ineq1}; (c) uses $\gamma = \mu + \frac{1-\mu\eta}{E\eta}$.
	
	Besides, we also have
	\begin{align}
	\label{eq:nableF}
	\|\nabla F(\y_t)\|^2
	&\le 2\|\nabla F(\y_t)-\nabla F(\xc)\|^2 + 2\|\nabla F(\xc)\|^2   &  \text{By \ Lemma}~\ref{lem:j} \nonumber \\
	&\le 4L\left[ F(\y_t) - F(\xc)\right]+ 2\|\nabla F(\xc)\|^2.   &\text{By \ Lemma}~\ref{lem:F}
	\end{align}
	For the third term of~\eqref{eq:z_zhankai}, we instead have
	\begin{align}
	\label{eq:third_z_1}
	\EB \|\PA( \nabla F(\x_{t}; \xi_{t})) \|^2&\le 3\EB \|\PA( \nabla F(\x_{t}; \xi_{t}) - \nabla F(\y_t; \xi_{t})) \|^2
	+ 3\EB \|\PA( \nabla F(\y_{t}; \xi_{t}) - \nabla F(\xc; \xi_{t})) \|^2 \nonumber \\
	& \qquad+  3\EB \|\PA( \nabla F(\xc; \xi_{t})) \|^2 \nonumber \\
	&\le 3\EB \| \nabla F(\x_{t}; \xi_{t}) - \nabla F(\y_t; \xi_{t}) \|^2
	+ 3\EB \|\nabla F(\y_{t}; \xi_{t}) - \nabla F(\xc; \xi_{t}) \|^2 + 3 \sigmaa \nonumber \\
	&\le 3L^2 \EB \|\x_t - \y_t\|^2 + 6L \EB \left[ F(\y_t) - F(\xc) \right]+ 3 \sigmaa.
	\end{align}
	Then, based on~\eqref{eq:z_zhankai},~\eqref{eq:second_z_no},~\eqref{eq:nableF} and~\eqref{eq:third_z_1}, when $\eta_t \le \frac{1}{L(3+2E)}$, we have:
	\begin{align}
	\label{eq:z_descent_11}
	\EB \|\z_{t+1}\|^2
	&\le\left(1 -\mu \eta_t + \frac{1-\mu \eta_t}{E}+3L^2\eta_t^2\right) \EB \|\z_{t}\|^2 + 3\eta_t^2 \sigmaa \nonumber \\
	& \qquad  \qquad + \left(6+ 4E\right)L\eta_t^2\EB\left[ F(\y_t) - F(\xc) \right] + 2E\eta_t^2 \|\nabla F(\xc)\|^2 \nonumber \\
		&\le \left(1 -\mu \eta_t + \frac{1-\mu \eta_t}{E}+L\eta_t\right) \EB \|\z_{t}\|^2 + 3\eta_t^2 \sigmaa \nonumber \\
		& \qquad  \qquad + 2\eta_t\EB\left[ F(\y_t) - F(\xc) \right] + 2E\eta_t^2 \|\nabla F(\xc)\|^2 
	\end{align}
\end{proof}

\subsection{Other Helper Lemmas}
\begin{lem}[Error Propagation]
	\label{lem:error_prop}
	Let $L_t = \left(\begin{matrix}
	r_t \\ u_t
	\end{matrix}\right)$ be a two dimensional error vector with non-negative entries.
	Assume it satisfies the following relation: for all $t \ge 0$, there exist a non-negative sequence $\{\delta_t\}_{t \ge 0}$, a upper triangular matrix $A = \left(
	\begin{matrix}
	a_1  & a_2 \\
	0  & a_3
	\end{matrix}
	\right) \in \RB^{2 \times 2}$, and two vectors $\bb=\left(
	\begin{matrix}
	b_1\\
	-b_2
	\end{matrix}
	\right), 
	\c=\left(
	\begin{matrix}
	c_1\\
	c_2
	\end{matrix}
	\right) \in \RB^{2}$ such that
	\begin{equation}
	\label{eq:error_relation}
	L_{t+1} \le A L_{t} -\eta \delta_t\bb + \eta^2 \c
	\end{equation}
	where the inequality holds element-by-element and $\{a_i\}_{i=1}^3, \{b_i\}_{i=1}^2, \{c_i\}_{i=1}^2$ are all non-negative real numbers.
	Then it follows that for any $0 \le t_0 \le t_1 \le t_0 + E$, once  $a_1^i b_1 \ge 2a_2b_2 \kappa_i$ for all $0 \le i \le E-1$,
	\begin{equation}
	\label{eq:error_r}
	r_{t_1} \le \left[ a_1^{s_0}r_{t_0} + \kappa_{s_0}a_2  u_{t_0}\right]
	-\frac{\eta}{2} \sum_{i=0}^{s_k-1} \delta_{t_{1}-i-1} a_1^ib_1
	+\eta^2 \sum_{i=0}^{s_0-1} \left[a_1^i c_1 +\kappa_i a_2c_2 \right].
	\end{equation}
	where $\kappa_i \le i \max\{a_1, a_3\}^{i-1}$ and
	\begin{equation}
	\label{eq:kappa}
	\kappa_i =\begin{cases}
	ia_1^{i-1}& \text{if} \ a_1=a_3 \\
	\frac{a_1^i - a_3^i}{a_1-a_3}& \text{if} \ a_1 \neq a_3
	\end{cases} 
	\end{equation}
\end{lem}
\begin{proof}
	Starting from $t_0$ and recurring~\eqref{eq:error_relation} for $s_0=t_1-t_0$ times, we have
	\begin{equation}
	\label{eq:error_L0}
	L_{t_{1}} \le A^{s_0} L_{t_0} - \eta \sum_{i=0}^{s_0-1}\delta_{t_{1}-i-1} A^i \bb + \eta^2 \sum_{i=0}^{s_0-1} A^i \c
	\end{equation}
	With $\e=(1, 0)^\top$ on the left of both sides of~\eqref{eq:error_L0}, we focus on the first entry of $L_t$:
	\begin{equation}
	\label{eq:error_L1}
	r_{t_1} \le  \e^\top A^{s_0}L_{t_0}-  \eta \sum_{i=0}^{s_0-1} \delta_{t_{1}-i-1} \e^\top A^i \bb + \eta^2 \sum_{i=0}^{s_0-1} \e^\top A^i \c.
	\end{equation}
	To give a clear form of~\eqref{eq:error_L1}, we should give the close form of $A^i$.
	By reduction, we have that 
	\[
	\text{if} \ a_1=a_3, A^i =\left(\begin{matrix}
	a_1^i & ia_1^{i-1} a_2\\ 0 & a_1^i
	\end{matrix}\right); \ \text{otherwise} \ A^i =\left(\begin{matrix}
	a_1^i &  \frac{a_1^i - a_3^i}{a_1-a_3}a_2\\ 0 & a_3^i
	\end{matrix}\right).
	\]
	By defining $\kappa_i$ in~\eqref{eq:kappa}, $A^i$ then has a unified close form: $A^i =\left(\begin{matrix}
	a_1^i & a_2\kappa_i\\ 0 & a_3^i
	\end{matrix}\right).$
	Without loss of generality, we assume $a_3 > a_1$. 
	By the convexity of $x^i \ (i \ge 0)$, $a_3^i - a_1^i \le ia_3^{i-1}\cdot(a_3-a_1)$ giving the upper bound on $\kappa_i$.
	Then, once $a_1^i b_1 \ge 2a_2b_2 \kappa_i$ for all $0 \le i \le E-1$, we have
	\begin{equation}
	\label{eq:1}
	\sum_{i=0}^{s_0-1} \delta_{t_{1}-i-1} \e^\top A^i \bb 
	= \sum_{i=0}^{s_0-1} \delta_{t_{1}-i-1} \left[a_1^i b_1 - a_2b_2 \kappa_i \right] \\
	\ge  \frac{1}{2} \sum_{i=0}^{s_0-1} \delta_{t_{1}-i-1} a_1^ib_1
	\end{equation}
	\begin{equation}
	\sum_{i=0}^{s_0-1} \e^\top A^i \c
	=  \sum_{i=0}^{s_0-1} \left[a_1^i c_1 +\kappa_i a_2 c_2 \right]
	\label{eq:2}
	\end{equation}
	Combining~\eqref{eq:error_L1},~\eqref{eq:1} and~\eqref{eq:2} gives~\eqref{eq:error_r}.
\end{proof}

\begin{lem}[Learning rate choice I]
	\label{lem:lr1}
	For any $c_1, c_2 \ge 0$ and $r_0, d, T > 0$, we can always find $0<\eta \le \frac{1}{d}$ such that
	\[
	\Phi(\eta) = \frac{r_0}{T\eta} + \eta c_1 + \eta^2 c_2 =
	\OM\left(\frac{dr_0}{T} + \sqrt{\frac{r_0c_1}{T}} +  \sqrt[3]{\frac{r_0^2c_2}{T^2}}\right).
	\]
\end{lem}
\begin{proof}
	By setting $\eta_0 = \min\left\{  \frac{1}{d},\sqrt{\frac{r_0}{c_1 T}}, \sqrt[3]{\frac{r_0}{c_2T}} \right\}$, we have that
	\[
	\Phi(\eta_0) \le \frac{r_0}{T}\max\left\{ d, \sqrt{\frac{c_1 T}{r_0}}, \sqrt[3]{\frac{c_2T}{r_0}} \right\} + \sqrt{\frac{r_0c_1}{T}} +  \sqrt[3]{\frac{r_0^2c_2}{T^2}}
	\le 	\frac{dr_0}{T} + 2\sqrt{\frac{r_0c_1}{T}} +  2\sqrt[3]{\frac{r_0^2c_2}{T^2}}.
	\]
	where the last inequality uses $\max\{a, b, c\}\le a+b+c$ for any non-negative numbers $a, b, c$.
\end{proof}

\begin{lem}[Learning rate choice II]
	\label{lem:lr2}
	For any $r_0, c_1, c_2 \ge 0$ and $d, T, c_0 > 0$, we can always find $0 <\eta \le \frac{1}{d}$ such that
	\[
	\Phi(\eta) = \exp(-\eta c_0 T)\frac{r_0}{\eta} + \eta c_1 + \eta^2 c_2 =
	\widetilde{\OM}\left( dr_0 \exp\left(- \frac{c_0}{d} T\right) + \frac{c_1}{c_0 T} + \frac{c_2}{c_0^2 T^2}  \right).
	\]
\end{lem}
\begin{proof}
	Let us consider three cases:
	\begin{enumerate}
		\item If $\frac{1}{d} \le \frac{1}{c_0T}$, then we can pick $\eta_0=\frac{1}{d}$ and get $\Phi(\eta_0)$ bounded by
		\[
		\Phi(\eta_0) =dr_0 \exp\left(- \frac{c_0}{d} T\right) + \frac{c_1}{d} + \frac{c_2}{d^2}
		={\OM}\left( dr_0 \exp\left(- \frac{c_0}{d} T\right) + \frac{c_1}{c_0 T} + \frac{c_2}{c_0^2 T^2}  \right)
		\]
		\item If $\frac{1}{d} > \frac{1}{c_0T}$ and $c_0^2r_0T^2 \ge c_1$, then we can pick $\eta_0=\frac{1}{c_0T}\ln \frac{c_0^2r_0T^2}{c_1}$ and get $\Phi(\eta_0)$ bounded by
		\[
		\Phi(\eta_0)=\frac{c_1}{c_0 T \ln \frac{c_0^2r_0T^2}{c_1}} + \frac{c_1}{c_0 T } \ln \frac{c_0^2r_0T^2}{c_1} + \frac{c_2}{c_0^2 T^2}\ln^2 \frac{c_0^2r_0T^2}{c_1}
		= \widetilde{\OM}\left( \frac{c_1}{c_0 T} + \frac{c_2}{c_0^2 T^2} \right).
		\]
		\item If $\frac{1}{d} > \frac{1}{c_0T}$ and $c_0^2r_0T^2 < c_1$, then we can pick $\eta_0=\frac{1}{c_0T}$ and get $\Phi(\eta_0)$ bounded by
		\[\Phi(\eta_0) = \exp(-1)r_0c_0T +  \frac{2c_1}{c_0 T} + \frac{c_2}{c_0^2 T^2} 
		<\frac{2c_1}{c_0 T} + \frac{c_2}{c_0^2 T^2} 
		= \OM\left(\frac{c_1}{c_0 T} + \frac{c_2}{c_0^2 T^2} \right).
		\]
	\end{enumerate}
\end{proof}

\subsection{Proof of Theorem~\ref{thm:simple}}
\begin{proof}
We consider a fixed step size, i.e., $\eta_t = \eta  \le \frac{1}{10L}$ for all $t \ge 0$.
From Lemma~\ref{lem:y_descent} and~\ref{lem:z_descent}, when we concatenate~\eqref{eq:y_descent} and~\ref{eq:z_descent_0}, it follows that
\begin{equation*}
\label{eq:error_con}
\left(\begin{matrix}
\EB\|\y_{t+1}-\xc\|^2\\
\EB\|\z_{t+1}\|^2
\end{matrix}\right)
\le \left(
\begin{matrix}
1- \mu \eta & 2L \eta\\
0  & 1- \mu \eta
\end{matrix}
\right)
\left(\begin{matrix}
\EB\|\y_{t}-\xc\|^2\\
\EB\|\z_{t}\|^2
\end{matrix}\right)
-\eta\EB \left[F(\y_{t})-F(\xc)\right]
\left(
\begin{matrix}
1\\-2
\end{matrix}
\right)
+3\eta^2
\left(
\begin{matrix}
\sigmaat \\
\sigmaa
\end{matrix}
\right),
\end{equation*}
which implies it satisfies Lemma~\ref{lem:error_prop} with the following parameters:
$\delta_t =\EB \left[ F(\y_{t})-F(\xc) \right]\ge 0$ and 
\[
L_t = \left(\begin{matrix}
\EB\|\y_{t}-\xc\|^2\\
 \EB\|\z_{t}\|^2
\end{matrix}\right),
A = \left(
\begin{matrix}
1- \mu \eta & 2L \eta\\
0  & 1- \mu \eta
\end{matrix}
\right),
\bb = 
\left(
\begin{matrix}
1\\-2
\end{matrix}
\right),
\c = 
3\left(
\begin{matrix}
\sigmaat \\
\sigmaa
\end{matrix}
\right),
\kappa_i = 2L\eta(1-\mu \eta)^{i-1}i.
\]
such that the following inequality that holds element-by-element: $L_{t+1} \le A L_{t} -\eta \delta_t\bb + \eta^2  \c$.

Let $0=t_0 < t_1 < t_2 < \cdots < t_K = T$ be the elements of $\IM_T$ and denote $s_k = t_{k+1} - t_{k}$ (so that $\sum_{i=0}^{K-1}s_i = T$ and $s_k \le E$).
Since when $t \in \IM_T$, we perform a projection to force $\EB \|\z_{t}\|^2 = 0$, implying the second entry of $L_{t_k} (k \ge 0)$ is zero.
Combing all, Lemma~\ref{lem:error_prop} gives, for all $0\le k\le K-1$, 
\begin{align}
\label{eq:y1}
\EB\|\y_{t_{k+1}} -\xc\|^2 
&\le (1-\mu \eta)^{s_k}\EB\|\y_{t_{k}} -\xc\|^2 
-\frac{\eta}{2} \sum_{i=0}^{s_k-1} (1-\mu \eta)^i  \delta_{t_{k+1}-i-1} \nonumber \\
&+ 3 \eta^2 \sum_{i=0}^{s_k-1} \left[(1-\mu \eta)^i \sigmaat +2L\eta (1-\mu \eta)^{i-1} i \sigmaa \right]. 
\end{align}
Here we require $a_1^i b_1 \ge 2a_2b_2 \kappa_i$ for all $0 \le i \le E-1$, i.e., $a_1b_2 \ge 2(E-1)a_2b_2$ here since $a_1=a_3=1-\mu\eta$.
In this specific case, it is equivalent to $\eta \le \frac{1}{\mu + 8L(E-1)}$.

Recurring~\eqref{eq:y1} from $k=0$ to $K-1$, we obtain
\begin{align}
\label{eq:y2}
\EB\|\y_{T} -\xc\|^2 
&\le (1-\mu \eta)^{T}\EB\|\y_{0} -\xc\|^2 
-\frac{\eta}{2} \sum_{k=0}^{K-1}   (1-\mu \eta)^{T-t_{k+1}} \sum_{i=0}^{s_k-1} (1-\mu \eta)^i \delta_{t_{k+1}-i-1} \nonumber \\
&\qquad + 3 \eta^2 \sum_{k=0}^{K-1}   (1-\mu \eta)^{T-t_{k+1}}\sum_{i=0}^{s_k-1} \left[(1-\mu \eta)^i \sigmaat +2L\eta (1-\mu \eta)^{i-1} i \sigmaa \right] \nonumber \\
&= (1-\mu \eta)^{T}\EB\|\y_{0} -\xc\|^2 
-\frac{\eta}{2}\sum_{k=0}^{K-1}   (1-\mu \eta)^{T-t_{k+1}} \sum_{j=t_k}^{t_{k+1}-1} (1-\mu \eta)^{t_{k+1}-j-1} \delta_{j} \nonumber \\
&\qquad+ 3 \eta^2 \sum_{k=0}^{K-1}   (1-\mu \eta)^{T-t_{k+1}}\sum_{i=0}^{s_k-1} \left[(1-\mu \eta)^i \sigmaat +2L\eta (1-\mu \eta)^{i-1} i \sigmaa \right]\nonumber\\
&= (1-\mu \eta)^{T}\EB\|\y_{0} -\xc\|^2 -\frac{\eta}{2}\sum_{j=0}^{T-1}(1-\mu \eta)^{T-j-1} \delta_{j}
 \nonumber + 3 \eta^2\sum_{j=0}^{T-1}(1-\mu \eta)^{T-j-1} \sigmaat \\
 &\qquad + 6 L\eta^3\sum_{k=0}^{K-1}  \sum_{j=t_k}^{t_{k+1}-2}  (1-\mu \eta)^{T-j-2}(t_{k+1}-j-1) \sigmaa \nonumber \\
 &\le (1-\mu \eta)^{T}\EB\|\y_{0} -\xc\|^2  -\frac{\eta}{2}\sum_{j=0}^{T-1}(1-\mu \eta)^{T-j-1} \delta_{j} +  3 \eta^2\sum_{j=0}^{T-1}(1-\mu \eta)^{T-j-1} \sigmaat \nonumber \\
 & \qquad + 6 L\eta^3(E-1)\sum_{k=0}^{K-1}  \sum_{j=t_k}^{t_{k+1}-2}  (1-\mu \eta)^{T-j-2} \sigmaa 
\end{align}
where the first equality uses change of variable $j=t_{k+1}-i-1$ and the second inequality uses $t_{k+1}-j-1 \le E-1 $ for any $t_k \le j \le t_{k+1}-2$.

Let $\Delta^2 = \EB\|\y_{0} -\xc\|^2 $.
Based on~\eqref{eq:y2}, we can derive convergence rate for constant learning rate $\eta \le \min\{\frac{1}{10L}, \frac{1}{\mu + 8L(E-1)}\}$.
Dividing $W_T = \sum_{j=0}^{T-1}(1-\mu \eta)^{T-j-1}$ on both sides of~\eqref{eq:y2} and rearranging, we obtain
\begin{equation}
\label{eq:y3}
F(\hat{\y})-F(\xc)\le
\frac{1}{W_T}\sum_{j=0}^{T-1}(1-\mu \eta)^{T-j-1} \delta_{j}  \le \frac{2}{\eta W_T}(1-\mu\eta)^T \Delta^2 + 6\eta\sigmaat + 12(E-1) L \eta^2\sigmaa
\end{equation}
where $\hat{\y}=\frac{1}{W_T}\sum_{j=0}^{T-1} (1-\mu \eta)^{T-j-1} \y_j$. 
Then we consider two cases according to $\mu =0$ or not.
\begin{enumerate}
	\item When $\mu = 0, W_T = T, \hat{\y}=\frac{1}{T}\sum_{j=0}^{T-1} \y_j$, so~\eqref{eq:y3} is reduced to
	\begin{equation*}
	\label{eq:y4}
	 \EB \left[F(\hat{\y})-F(\xc)\right]
	\le \frac{2\Delta^2}{\eta T}  + 6\eta \sigmaat + 12(E-1) L \eta^2\sigmaa.
	\end{equation*}
	Using Lemma~\ref{lem:lr1} and setting therein parameters properly (i.e., $d=\frac{1}{10LE}, r_0=\Delta^2, c_1=\sigmaat, c_2=(E-1)L\sigmaa$), we can always find sufficiently small constant learning rate $\eta$ such that
	\[
	\EB \left[F(\hat{\y})-F(\xc)\right]
	\le \OM\left(
	\frac{LE\Delta^2}{T} +\frac{\Delta \sigma_{\A^{\perp},*}}{\sqrt{T}} + \sqrt[3]{\frac{(E-1)L\sigmaa\Delta^4}{T^2}} 
	\right).
	\]
	 \item When $\mu > 0$, $W_T = \frac{1-(1-\mu\eta)^T}{\mu\eta} \ge 1$, so~\eqref{eq:y3} becomes
	 \begin{equation*}
	 \label{eq:y5}
	 	\EB \left[ F(\hat{\y})-F(\xc)\right]
	 \le \exp\left(-\mu\eta T\right) \frac{2\Delta^2}{\eta} + 6\eta \sigmaat + 12(E-1) L \eta^2\sigmaa
	 \end{equation*}
Using Lemma~\ref{lem:lr2} and setting therein parameters properly (i.e., $d=\frac{1}{10LE}, r_0=\Delta^2, c_0 = \mu, c_1=\sigmaat, c_2=(E-1)L\sigmaa$), we can always find an appropriate constant learning rate $\eta$ such that 
 \[
 \EB \left[F(\hat{\y})-F(\xc) \right]= \widetilde{\OM}\left(  	LE\Delta^2 \cdot\exp\left(-\Theta\left(\frac{\mu T}{LE}\right) \right) + \frac{\sigmaat}{\mu T}  + \frac{(E-1)L\sigmaa}{\mu^2 T^2}  \right).
 \]
\end{enumerate}
\end{proof}

\subsection{Proof of Theorem~\ref{thm:complicate}}
\begin{proof}
	In this part, we derive convergence results for Algorithm~\ref{alg:multi} when Assumption~\ref{asmp:g_unconstrained} is absent.
	We use a similar argument inherent in the proof of Theorem~\ref{thm:simple}.
	Here we also consider a fixed step size, i.e., $\eta_t = \eta  \le \frac{1}{10L}$ for all $t \ge 0$.
	From Lemma~\ref{lem:y_descent} and~\ref{lem:z_descent}, when we concatenate~\eqref{eq:y_descent} and~\ref{eq:z_descent_1}, it follows that
	\begin{equation*}
	\label{eq:error_con1}
	\left(\begin{matrix}
	\EB\|\y_{t+1}-\xc\|^2\\
	\EB\|\z_{t+1}\|^2
	\end{matrix}\right)
	\le \left(
	\begin{matrix}
	1- \mu \eta  & 2L \eta\\
	0& \theta
	\end{matrix}
	\right)
	\left(\begin{matrix}
	\EB\|\y_{t}-\xc\|^2\\
	\EB\|\z_{t}\|^2
	\end{matrix}\right)
	-\eta\EB
	\left[ F(\y_t)-F(\xc) \right]\left(
	\begin{matrix}
	 1\\-2
	\end{matrix}
	\right)
	+3\eta^2
	\left(
	\begin{matrix}
	\sigmaat \\
	\widetilde{\sigma}_{A,*}^2
	\end{matrix}
	\right).
	\end{equation*}
	where $\theta = (1 + \frac{1}{E})(1-\mu\eta)+L\eta$ and $\widetilde{\sigma}_{\A,*}^2=\sigmaa + \frac{2E}{3}\|\nabla F(\xc)\|^2$ for simplicity.
	
	The last inequality implies the error propagation satisfies Lemma~\ref{lem:error_prop} with the following parameters: $\delta_t = F(\y_{t})-F(\xc) \ge 0$ and 
	\[
	L_t = \left(\begin{matrix}
	\EB\|\y_{t}-\xc\|^2\\
	\EB\|\z_{t}\|^2
	\end{matrix}\right),
	A = \left(
	\begin{matrix}
	1- \mu \eta & 2L \eta\\
	0 & \theta
	\end{matrix}
	\right),
	\bb = 
	\left(
	\begin{matrix}
	1\\-2
	\end{matrix}
	\right),
	\c = 
	3\left(
	\begin{matrix}
	\sigmaat \\
	\widetilde{\sigma}_{\A,*}^2
	\end{matrix}
	\right),
	\kappa_i = \frac{\theta^i - (1-\mu\eta)^i}{\theta-(1-\mu\eta)}
	\]
	such that the following inequality that holds element-by-element: $	L_{t+1} \le A L_{t} -\eta \delta_t\bb + \eta^2  \c$.
	
	Let $0=t_0 < t_1 < t_2 < \cdots < t_K = T$ be the elements of $\IM_T$ and denote $s_k = t_{k+1} - t_{k}$ (so that $\sum_{i=0}^{K-1}s_i = T$ and $s_k \le E$).
	Since when $t \in \IM_T$, we perform a projection to force $\EB \|\z_{t}\|^2 = 0$, implying the second entry of $L_{t_k} (k \ge 0)$ is zero.
	Combing all, Lemma~\ref{lem:error_prop} gives, for all $0\le k\le K-1$, 
	\begin{align}
	\label{eq:y7}
	\EB\|\y_{t_{k+1}} -\xc\|^2 
	&\le (1-\mu \eta)^{s_k}\EB\|\y_{t_{k}} -\xc\|^2 
	-\frac{\eta}{2} \sum_{i=0}^{s_k-1} (1-\mu \eta)^i  \delta_{t_{k+1}-i-1} \nonumber \\
	&+ 3 \eta^2 \sum_{i=0}^{s_k-1} \left[(1-\mu \eta)^i \sigmaat +2L\eta \kappa_i \widetilde{\sigma}_{A,*}^2 \right]. 
	\end{align}
	Here we require $a_1^i b_1 \ge 2a_2b_2 \kappa_i$ for all $0 \le i \le E-1$.
	Since $\kappa_i \le i \max\{a_1, a_3\}^{i-1}$, a sufficient condition is to require $a_1^i b_1 \ge 2a_2b_2 \cdot i\max\{a_1, a_3\}^{i-1} $ for all $0 \le i \le E-1$.
	Plugging it with $a_1=1-\mu\eta, a_2 = 2L\eta, a_3 = \theta > a_1, b_1=1, b_2=2$, it is equivalent to for all $0 \le i \le E-1$,
	\begin{equation}
	\label{eq:cond}
	(1-\mu \eta)^i  \ge 8L\eta \cdot i \theta^{i-1}.
	\end{equation}
	To make sure~\eqref{eq:cond} holds, we only need to tune $\eta \le \frac{1}{\mu + 25L(E-1)}$.
	Indeed, when $\eta \le \frac{1}{\mu + 25L(E-1)}$, 
	\begin{align}
	\label{eq:lr}
	8L\eta \cdot i \left(\frac{\theta}{1-\mu\eta}\right)^i 
	&\le 	8L\eta \cdot i \left(1 + \frac{1}{E-1}+ \frac{L\eta}{1-\mu\eta}\right)^i \nonumber \\
	&\overset{(a)}{\le} 8L\eta \cdot (E-1) \left(1 + \frac{1}{E-1}+ \frac{L\eta}{1-\mu\eta}\right)^{E-1} \nonumber \\
	&\overset{(b)}{\le }8L\eta \cdot (E-1) \left(1 + \frac{1.04}{E-1}\right)^{E-1} \nonumber \\
	&\overset{(c)}{\le }25L\eta \cdot (E-1) \nonumber \\
	&\overset{(d)}{\le }1-\mu \eta <\theta
	\end{align}
	where (a) holds since $0 \le i \le E-1$; (b) follows by noting $ \frac{3L\eta}{1-\mu\eta}$ increases in $\eta$ and $\eta \le \frac{1}{\mu + 10L(E-1)}$ is upper bounded; (c) uses $8(1+\frac{1.04}{n})^n \le8 \mathrm{e}^{1.04}< 25$ for any non-negative integer $n$; (d) follows since $\eta \le \frac{1}{\mu + 25L(E-1)}$.
	
	Also, using $\kappa_i \le i \max\{a_1, a_3\}^{i-1}$, we have
	\begin{align}
	\label{eq:3}
	2\kappa_i 
	&\le 2i\cdot\theta^{i-1} = 2i(1-\mu\eta)^{i-1}\left(\frac{\theta}{1-\mu\eta}\right)^{i-1} \nonumber \\
	&\le  i(1-\mu\eta)^{i-1} \cdot 2\left(1 + \frac{1.04}{E-1}\right)^{E-1} \nonumber \\
	&\le i(1-\mu\eta)^{i-1} \cdot 2\mathrm{e}^{1.04} \le  6i(1-\mu\eta)^{i-1}.
	\end{align}
	Combing~\eqref{eq:y7} and~\eqref{eq:3}, we obtain
	\begin{align}
	\label{eq:y8}
	\EB\|\y_{t_{k+1}} -\xc\|^2 
	&\le (1-\mu \eta)^{s_k}\EB\|\y_{t_{k}} -\xc\|^2 
	-\frac{\eta}{2} \sum_{i=0}^{s_k-1} (1-\mu \eta)^i  \delta_{t_{k+1}-i-1} \nonumber \\
	&+ 3 \eta^2 \sum_{i=0}^{s_k-1} \left[(1-\mu \eta)^i \sigmaat +6L\eta \cdot i(1-\mu\eta)^{i-1}  \widetilde{\sigma}_{\A,*}^2 \right]. 
	\end{align}
	Recurring~\eqref{eq:y8} from $k=0$ to $K-1$, we have
	\begin{align}
	\label{eq:y9}
	 \EB\|\y_{T} -\xc\|^2 &\le (1-\mu \eta)^{T}\EB\|\y_{0} -\xc\|^2  -\frac{\eta}{2}\sum_{j=0}^{T-1}(1-\mu \eta)^{T-j-1} \delta_{j} +  3 \eta^2\sum_{j=0}^{T-1}(1-\mu \eta)^{T-j-1} \sigmaat \nonumber \\
	& \qquad + 18 L\eta^3(E-1)\sum_{k=0}^{K-1}  \sum_{j=t_k}^{t_{k+1}-2}  (1-\mu \eta)^{T-j-2} \widetilde{\sigma}_{\A,*}^2.
	\end{align}
	Recall that $\Delta^2 = \EB\|\y_{0} -\xc\|^2 $.
	Based on~\eqref{eq:y9}, we can derive convergence rate for constant learning rate $\eta \le \min\{\frac{1}{10L}, \frac{1}{\mu + 25L(E-1)}\}$.
	Dividing $W_T = \sum_{j=0}^{T-1}(1-\mu \eta)^{T-j-1}$ on both sides of~\eqref{eq:y9} and rearranging, we obtain
	\begin{equation*}
	\EB \left[F(\hat{\y})-F(\xc) \right] \le
	\frac{1}{W_T}\sum_{j=0}^{T-1}(1-\mu \eta)^{T-j-1} \delta_{j}  \le \frac{2}{\eta W_T}(1-\mu\eta)^T \Delta^2 + 6\eta\sigmaat + 18(E-1) L \eta^2\widetilde{\sigma}_{\A,*}^2
	\end{equation*}
	where $\hat{\y}=\frac{1}{W_T}\sum_{j=0}^{T-1} (1-\mu \eta)^{T-j-1} \y_j$. 
	Since the last inequality is quite similar to~\eqref{eq:y3}, the proof follows from a similar argument of the poof for Theorem~\ref{thm:simple} by replacing $\sigmaa$ with $\widetilde{\sigma}_{\A,*}^2 $.
\end{proof}

\section{Proof of Delayed Projected SVRG}

\subsection{Descent Lemma}
In this section, we give a decent lemma for Algorithm~\ref{alg:multi_SVRG} that is a counter part of Lemma~\ref{lem:y_descent}.
The main technique is the same except that we use variance-reduced estimators for stochastic gradients.
\begin{lem}[Bounded gradient variance]
	\label{lem:var_svrg}
	Let $\g_t^s = \nabla F(\x_{t}^s; \xi_{t}^s) - \nabla F(\ttx_s; \xi_{t}^s) + \PAT(\nabla F(\ttx_{s}))$.
	 Then, $\PAT(\g_t^s)$ is unbiased for $\PAT(\nabla F(\x_t^s))$ (i.e., $\EB_{\xi_{t}^s}  \PAT(\g_t^s) =  \PAT(\nabla F(\x_t^s))$) and has bounded variance at most:
	\[
	\EB \| \PAT(\g_t^s)- \PAT(\nabla F(\x_{t}^s))\|^2
	\le 3L^2 \EB \|\z_t^s\|^2 + 6L\EB\left[ F(\y_t^s) - F(\xc) + F(\ttx_s) - F(\xc)  \right].
	\]
\end{lem}

\begin{proof}
	Note that $\EB_{\xi_{t}^s}  \PAT(\g_t^s) = \PAT(\nabla F(\x_t^s) - \nabla F(\ttx_{s})) + \PAT(\nabla F(\ttx_{s})) =  \PAT(\nabla F(\x_t^s))$.
	The proof of bounded variance is classical and is analogous to most of the variance reduction
	literature.
	 \begin{align*}
	 \EB_{\xi_{t}^s} &\| \PAT(\g_t^s)- \PAT(\nabla F(\x_{t}^s))\|^2\\
	 &=	 \EB_{\xi_{t}^s} \|  \PAT( \nabla F(\x_{t}^s; \xi_{t}^s) - \nabla F(\ttx_s; \xi_{t}^s) ) -  \PAT( \nabla F(\x_{t}^s) - \nabla F(\ttx_s) )     \|^2\\
	 &\le  \EB_{\xi_{t}^s} \|  \PAT( \nabla F(\x_{t}^s; \xi_{t}^s) - \nabla F(\ttx_s; \xi_{t}^s) ) \|^2\\
	 &=  \EB_{\xi_{t}^s} \|  \PAT( \nabla F(\x_{t}^s; \xi_{t}^s)-\nabla F(\y_t^s; \xi_{t}^s)) 
	 + \PAT( \nabla F(\y_t^s; \xi_{t}^s)-\nabla F(\xc; \xi_{t}^s))  \\
	 & \qquad   - \PAT( \nabla F(\ttx_s; \xi_{t}^s) -\nabla F(\xc; \xi_{t}^s)) \|^2\\
	 &\le 3\EB_{\xi_{t}^s} \|  \PAT( \nabla F(\x_{t}^s; \xi_{t}^s)-\nabla F(\y_t^s; \xi_{t}^s)) \|^2 
	 + 3  \EB_{\xi_{t}^s} \| \PAT( \nabla F(\y_t^s; \xi_{t}^s)-\nabla F(\xc; \xi_{t}^s)) \|^2\\
	 & \qquad  + 3\EB_{\xi_{t}^s} \|\PAT( \nabla F(\ttx_s; \xi_{t}^s) -\nabla F(\xc; \xi_{t}^s)) \|^2\\
	 &\le 3\EB_{\xi_{t}^s} \|  \nabla F(\x_{t}^s; \xi_{t}^s)-\nabla F(\y_t^s; \xi_{t}^s)\|^2 
	+ 3  \EB_{\xi_{t}^s} \| \nabla F(\y_t^s; \xi_{t}^s)-\nabla F(\xc; \xi_{t}^s)\|^2\\
	& \qquad  + 3\EB_{\xi_{t}^s} \| \nabla F(\ttx_s; \xi_{t}^s) -\nabla F(\xc; \xi_{t}^s)\|^2\\
	&\le 3\EB_{\xi_{t}^s} \left[  L^2 \|\z_t^s\|^2 + 2L\left[F(\y_t^s) - F(\xc) \right] + 2L\left[F(\ttx_s) - F(\xc) \right] \right].
	 \end{align*}
\end{proof}

\begin{lem}
\label{lem:y_descent_1}
Under Assumption~\ref{asmp:smooth} and~\ref{asmp:strong},
let $\y_t^s = \PAT (\x_t^s)$ be the projection onto $\mathcal{R}(\A^\perp)$ and $\z_t^s = \PA(\x_t^s)$ the projection onto $\mathcal{R}(\A)$, then for Algorithm~\ref{alg:multi_SVRG}, when $\eta_t^s = \eta \le \frac{1}{10L}$, we have
\begin{equation}
\label{eq:y_descent_svrg}
	\EB \|\y_{t+1}^s-\xc\|^2
	\le (1-\mu \eta)  \EB \| \y_t^s-\xc\|^2 -\eta \EB\left[ F(\y_t^s) - F(\xc)\right]  +6L\eta^2\EB\left[ F(\ttx_s) - F(\xc)\right] + 2L\eta\EB \| \z_t^s\|^2
\end{equation}
\end{lem}
\begin{proof}
	We first fixed any $s \ge 0$.
	Lemma~\ref{lem:var_svrg} shows that $\PAT(\g_t^s) = \PAT(\nabla F(\x_{t}^s; \xi_{t}^s) - \nabla F(\ttx_s; \xi_{t}^s) )+ \PAT(\nabla F(\ttx_{s}))$ is unbiased for $\PAT(\nabla F(\x_t^s))$ and with variance at most
	\[
	\EB \| \PAT(\g_t^s- \nabla F(\x_t))\|^2  \le 3L^2\EB \| \z_t^s \|^2 
	+ 6 L\EB\left[ F(\y_t^s) - F(\xc)\right] + 6L\EB\left[ F(\ttx_s) - F(\xc)\right].
	\]
	Lemma~\ref{lem:y_descent_general} with the choice of $\g_t^s$ and a constant learning rate $\eta_{t}^s = \eta$ gives that 
	\begin{align*}
		\EB \|\y_{t+1}^s-\xc\|^2
	&\le (1-\mu \eta)  \EB \| \y_t-\xc\|^2  +  (4L\eta^2-2\eta) \EB\left[ F(\y_t) - F(\xc)\right]   + (L\eta + 2L^2\eta^2)\EB \| \z_t \|^2\\
	& \qquad \qquad +\eta^2  \EB \| \PAT(\g_t^s)-\PAT(\nabla F(\x_t))\|^2.
	\end{align*}
	Combing the last two inequalities, we get
	\begin{align*}
	\EB &\|\y_{t+1}^s-\xc\|^2\\
	&\le (1-\mu \eta)  \EB \| \y_t^s-\xc\|^2 + (10L\eta^2-2\eta) \EB\left[ F(\y_t^s) - F(\xc)\right]  \\
	&\qquad \qquad+  6L\eta^2\EB\left[  F(\ttx_s) - F(\xc)\right] + (L\eta + 5L^2\eta^2)\EB \| \z_t^s \|^2\\
	&\le (1-\mu \eta)  \EB \| \y_t^s-\xc\|^2 -\eta \EB\left[ F(\y_t^s) - F(\xc)\right]  +6L\eta^2\EB\left[  F(\ttx_s) - F(\xc)\right] + 2L\eta\EB \| \z_t^s\|^2
	\end{align*}
	where for the final line we used that $\eta \le \frac{1}{10L}$.
\end{proof}

\subsection{Residual Lemma}
In this section, we given a residual lemma for Algorithm~\ref{alg:multi_SVRG} that is a counter part of Lemma~\ref{lem:z_descent}.
The main technique is the same except that we use variance-reduced estimators for stochastic gradients.
The main difference from Lemma~\ref{lem:z_descent} is that we replace $\sigma_{\A,*}^2$ with $F(\ttx_s) - F(\xc)$ that decays to zero when $s$ goes to infinity. 

\begin{lem}
	\label{lem:z_descent_1}
	Under Assumption~\ref{asmp:smooth} and~\ref{asmp:strong}, when $\mathrm{gap}(\IM_m) = E$ and $\eta_t^s = \eta \le \frac{1}{L(3+2E)}$, then
	\begin{equation}
	\label{eq:z_descent_svrg_1}
	\EB \|\z_{t+1}^s\|^2
	\le\left(1 -\mu \eta+ \frac{1-\mu \eta}{E}+L\eta\right) \EB \|\z_{t}^s\|^2  +  2\eta \EB \left[ F(\y_t^s) - F(\xc) \right]  +  2\eta\EB\left[ F(\ttx_s) - F(\xc) \right].
	\end{equation}
\end{lem}
\begin{proof}
	Noting that $\g_t^s = \nabla F(\x_{t}^s; \xi_{t}^s) - \nabla F(\ttx_s; \xi_{t}^s) + \PAT(\nabla F(\ttx_{s}))$, $\PA(\g_t^s) = \PA(\nabla F(\x_{t}^s; \xi_{t}^s) - \nabla F(\ttx_s; \xi_{t}^s)) $ and $\EB_{\xi_{t}^s}\PA(\g_t^s) = \PA(\nabla F(\x_{t}^s) - \nabla F(\ttx_{s}))$.
	Then,
	\begin{align}
	\label{eq:z_zhankai_1}
	\EB \|\z_{t+1}^s\|^2 
	&= \EB \big\| \PA(\x_t^s -  \eta \g_t^s) \big\|^2 \nonumber \\
	&= \EB \| \PA(\x_t^s) \|^2  -2\eta \EB \langle\PA(\x_t^s),  \PA(\g_t^s)\rangle + \eta^2  \EB \| \PA(\g_t^s) \|^2  \nonumber \\
	&= \EB \| \z_t^s \|^2  -2\eta \EB \langle \z_t^s,  \PA(\g_t^s)\rangle + \eta^2  \EB \| \PA(\g_t^s) \|^2 
	\end{align}
	where the last inequality uses $\PA(\x_t^s) = \z_t^s$.

	For the second term of~\eqref{eq:z_zhankai_1}, we have
	\begin{align}
	\label{eq:second_z_no_1}
	- \EB_{\xi_{t}^s}\langle \z_t^s,  \PA(\g_t^s)\rangle
	&= - \langle \x_t^s - \y_t^s,  \PA(\nabla F(\x_{t}^s) -   \nabla F(\ttx_{s})) \rangle \nonumber \\
	&= - \langle \x_t^s - \y_t^s,  \nabla F(\x_{t}^s) -   \nabla F(\ttx_{s}) \rangle \nonumber \\
	&= - \langle \x_t^s - \y_t^s,  \nabla F(\x_{t}^s)-\nabla F(\y_{t}^s)\rangle  - \langle \x_t^s - \y_t^s, \nabla F(\y_{t}^s)-   \nabla F(\ttx_{s})\rangle \nonumber \\
	&\overset{(a)}{\le} - \mu \|\z_t^s\|^2 + | \langle \z_t^s, \nabla F(\y_{t}^s)-   \nabla F(\ttx_{s}) \rangle|  \nonumber \\
	&\overset{(b)}{\le} - \mu \|\z_t^s\|^2 +  \frac{\gamma}{2}\|\z_t^s\|^2 + \frac{1}{2\gamma} \|\nabla F(\y_{t}^s)-   \nabla F(\ttx_{s})\|^2 \nonumber \\
	&\overset{(c)}{=}\left(-\frac{\mu}{2} + \frac{1-\mu\eta}{2E\eta}\right) \|\z_t^s\|^2  + \frac{1}{2\left(\mu + \frac{1-\mu\eta}{E\eta}\right)}  \|\nabla F(\y_{t}^s)-   \nabla F(\ttx_{s})\|^2 \nonumber \\
	&\le \left(-\frac{\mu}{2} + \frac{1-\mu\eta}{2E\eta}\right) \|\z_t^s\|^2  + \frac{E\eta}{2}  \|\nabla F(\y_{t}^s)-   \nabla F(\ttx_{s})\|^2
	\end{align}
	where (a) uses the $\mu$-strongly convexity of $F(\cdot)$ that implies $\langle \nabla F(\x_t^s) - F(\y_t^s), \x_t^s -\y_t^s \rangle \ge \mu \|\x_t^s -\y_t^s\|^2$ (Lemma~\ref{lem:F}); and (b) uses Lemma~\ref{lem:ineq1} with $\gamma = \mu + \frac{1-\mu\eta}{E\eta}$.
	Besides, we also have
	\begin{align}
	\label{eq:nableF_1}
	\|\nabla F(\y_{t}^s)-   \nabla F(\ttx_{s})\|^2
	&\le 2\|\nabla F(\y_t^s)-\nabla F(\xc)\|^2 + 2\|\nabla F(\ttx_{s})-\nabla F(\xc)\|^2   &  \text{By \ Lemma}~\ref{lem:j} \nonumber \\
	&\le 4 L\left[ F(\y_t^s) - F(\xc)  + F(\ttx_{s}) - F(\xc) \right].   &\text{By \ Lemma}~\ref{lem:F}
	\end{align}
	For the third term of~\eqref{eq:z_zhankai_1}, we have
	\begin{align}
	\label{eq:third_z_11}
	\EB \|\PA( \g_t^s) \|^2
	&=	\EB \|\PA(\nabla F(\x_{t}^s; \xi_{t}^s) - \nabla F(\ttx_s; \xi_{t}^s)) \|^2 \nonumber \\
	&\le	\EB \|\nabla F(\x_{t}^s; \xi_{t}^s) - \nabla F(\ttx_s; \xi_{t}^s) \|^2 \nonumber \\
	&\le 3\EB \|\nabla F(\x_{t}^s; \xi_{t}^s) - \nabla F(\y_t^s; \xi_{t}^s) \|^2
	+ 3\EB \|\nabla F(\y_{t}^s; \xi_{t}^s) - \nabla F(\ttx_s; \xi_{t}^s)\|^2 \nonumber \\
	& \qquad+  3\EB \|\nabla F(\ttx_s; \xi_{t}^s)- \nabla F(\xc; \xi_{t}^s) \|^2 \nonumber \\
	&\le 3L^2 \EB \|\z_t^s\|^2 + 6L \EB \left[ F(\y_t^s) - F(\xc)+F(\ttx_s) - F(\xc) \right]
	\end{align}
	Then, based on~\eqref{eq:z_zhankai_1},~\eqref{eq:second_z_no_1},~\eqref{eq:nableF_1} and~\eqref{eq:third_z_11}, when $\eta \le \frac{1}{L(3+2E)}$, we have:
	\begin{align*}
	\EB &\|\z_{t+1}^s\|^2  \\
	&\le\left(1-\mu \eta + \frac{1-\mu \eta}{E}+3L^2\eta^2\right) \EB \|\z_{t}^s\|^2  + (4E+6)L\eta^2  \EB \left[ F(\y_t^s) - F(\xc)+F(\ttx_s) - F(\xc) \right]\nonumber \\
	&\le\left(1 -\mu \eta+ \frac{1-\mu \eta}{E}+L\eta\right) \EB \|\z_{t}^s\|^2  +  2\eta \EB \left[ F(\y_t^s) - F(\xc) \right]  +  2\eta\EB\left[ F(\ttx_s) - F(\xc) \right].
	\end{align*}
\end{proof}

\subsection{Other Helper Lemmas}
\begin{lem}[Stage-wise error propagation]
	\label{lem:error_svrg_stage}
	Under Assumption~\ref{asmp:smooth} and~\ref{asmp:strong}, when 
	\begin{equation}
	\label{eq:lr_svrg_lem}
	\eta \le \min\left\{ \frac{1}{\mu+25L(E-1)}, \frac{1}{10L}, \frac{1}{L(3+2E)} \right\},
	\end{equation}
	then in a stage, DP-SVRG (Algorithm~\ref{alg:multi_SVRG}) holds that
	\[
	\EB \left[F(\ttx_{s+1}) - F(\xc)\right] \le 
	\frac{2}{\eta\Gamma} \left[ \rho \EB\|\y_{0}^s -\xc\|^2 - \EB\|\y_{0}^{s+1} -\xc\|^2 \right] + 
	\frac{12L\eta (2E-1) }{1-\mu \eta}\EB[F(\ttx_s) - F(\xc)]
	\]
	where 
	$\rho = (1-\mu\eta)^m$ and $\Gamma = \sum_{i=0}^{m-1} (1-\mu \eta)^i $.
\end{lem}

\begin{proof}
	In this part, we derive convergence results for Algorithm~\ref{alg:multi_SVRG}.
	We use a similar argument inherent in the proof of multi-step projected SGD.
	We consider a fixed step size, i.e., $\eta_t^s = \eta  \le \min\{\frac{1}{10L}, \frac{1}{L(3+2E)}\}$ for all $t, s\ge 0$ that guarantees the establishment of Lemma~\ref{lem:y_descent_1} and~\ref{lem:z_descent_1}.
	
	We consider a fixed stage $s$ first.
	From Lemma~\ref{lem:y_descent_1} and~\ref{lem:z_descent_1}, concatenating~\eqref{eq:y_descent_svrg} and~\ref{eq:z_descent_svrg_1} gives
	\begin{align*}
	\label{eq:error_con_svrg}
	\left(\begin{matrix}
	\EB\|\y_{t+1}^s-\xc\|^2\\
	\EB\|\z_{t+1}^s\|^2
	\end{matrix}\right)
	\le \left(
	\begin{matrix}
	1- \mu \eta  & 2L \eta\\
	0& \theta
	\end{matrix}
	\right)
	\left(\begin{matrix}
	\EB\|\y_{t}^s-\xc\|^2\\
	\EB\|\z_{t}^s\|^2
	\end{matrix}\right)
	&-\eta
	\left[ F(\y_t^s)-F(\xc) \right]\left(
	\begin{matrix}
	1\\-2
	\end{matrix}
	\right)\\
	&+2\eta^2\left[ F(\ttx_s)-F(\xc) \right]
	\left(
	\begin{matrix}
	3L \\
	\frac{1}{\eta}
	\end{matrix}
	\right).
	\end{align*}
	where $\theta = 1- \mu \eta+ \frac{1- \mu \eta}{E} +L\eta$ for simplicity.
	The last inequality implies the error propagation satisfies Lemma~\ref{lem:error_prop} with the following parameters: $\delta_t = \EB \left[F(\y_{t}^s)-F(\xc)\right] \ge 0$, $\kappa_i = \frac{\theta^i - (1-\mu\eta)^i}{\theta-(1-\mu\eta)}$, and
	\[
	L_t = \left(\begin{matrix}
	\EB\|\y_{t}^s-\xc\|^2\\
	\EB\|\z_{t}^s\|^2
	\end{matrix}\right),
	A = \left(
	\begin{matrix}
	1- \mu \eta & 2L \eta\\
	0 & \theta
	\end{matrix}
	\right),
	\bb = 
	\left(
	\begin{matrix}
	1\\-2
	\end{matrix}
	\right),
	\c = 
	2\left[ F(\ttx_s)-F(\xc) \right]\left(
	\begin{matrix}
	3L \\
	\frac{1}{\eta}
	\end{matrix}
	\right)
	\]
	such that the following inequality that holds element-by-element: $	L_{t+1} \le A L_{t} -\eta \delta_t\bb + \eta^2  \c$.
	
	Let $0=t_0 < t_1 < t_2 < \cdots < t_K = m$ be the elements of $\IM_m$ and denote $s_k = t_{k+1} - t_{k}$ (so that $\sum_{i=0}^{K-1}s_i = m$ and $s_k \le E$).
	Since when $t \in \IM_m$, we perform a projection to force $\EB \|\z_{t}^s\|^2 = 0$, implying the second entry of $L_{t_k} (k \ge 0)$ is zero.
	Combing all, Lemma~\ref{lem:error_prop} gives, for all $0\le k\le K-1$, 
	\begin{align}
	\label{eq:y10}
	\EB\|\y_{t_{k+1}}^s -\xc\|^2 
	&\le (1-\mu \eta)^{s_k}\EB\|\y_{t_k}^s -\xc\|^2 
	-\frac{\eta}{2} \sum_{i=0}^{s_k-1} (1-\mu \eta)^i  \delta_{t_{k+1}-i-1} \nonumber \\
	&+ 6L\eta^2  \EB[F(\ttx_s) - F(\xc)]\sum_{i=0}^{s_k-1} \left[(1-\mu \eta)^i  +2 i(1-\mu\eta)^{i-1}\right]. 
	\end{align}
	To ensure $a_1^i b_1 \ge 2a_2b_2 \kappa_i$ for all $0 \le i \le E-1$, we only need to tune $\eta \le \frac{1}{\mu + 25L(E-1)}$.
	The reason here is the same in~\eqref{eq:lr}.
	We also use $2\kappa_i \le 6i(1-\mu\eta)^{i-1}$ that is already derived in~\eqref{eq:3}.
	
	Recurring~\eqref{eq:y10} from $k=0$ to $K-1$, we have
\begin{align}
\label{eq:y10.5}
\EB\|\y_{m}^s -\xc\|^2 
&\le (1-\mu \eta)^{m}\EB\|\y_{0}^s -\xc\|^2 
-\frac{\eta}{2} \sum_{k=0}^{K-1}   (1-\mu \eta)^{m-t_{k+1}} \sum_{i=0}^{s_k-1} (1-\mu \eta)^i \delta_{t_{k+1}-i-1} \nonumber \\
&\qquad + 
6L\eta^2  \EB[F(\ttx_s) - F(\xc)] \sum_{k=0}^{K-1}   (1-\mu \eta)^{T-t_{k+1}}\sum_{i=0}^{s_k-1} \left[(1-\mu \eta)^i  +2 i(1-\mu\eta)^{i-1}\right]\nonumber \\
&= (1-\mu \eta)^{m}\EB\|\y_{0}^s -\xc\|^2 
-\frac{\eta}{2}\sum_{k=0}^{K-1}   (1-\mu \eta)^{m-t_{k+1}} \sum_{j=t_k}^{t_{k+1}-1} (1-\mu \eta)^{t_{k+1}-j-1} \delta_{j} \nonumber \\
&\qquad+ 
6L\eta^2  \EB[F(\ttx_s) - F(\xc)] \sum_{k=0}^{K-1}   (1-\mu \eta)^{T-t_{k+1}}\sum_{i=0}^{s_k-1} 
(1-\mu \eta)^i  \left(1+\frac{2i}{1-\mu \eta}\right) \nonumber \\
&\le (1-\mu \eta)^{m}\EB\|\y_{0}^s -\xc\|^2 -\frac{\eta}{2}\sum_{j=0}^{m-1}(1-\mu \eta)^{m-j-1} \delta_{j}  \nonumber \\
&\qquad + 6 L\eta^2 \EB[F(\ttx_s) - F(\xc)] 
\sum_{j=0}^{m-1}(1-\mu \eta)^{j}  \left(1+\frac{2(E-1)}{1-\mu \eta}\right)
\end{align}
where the equality uses change of variable $j=t_{k+1}-i-1$ and the second inequality uses $s_k \le E$ for any $0 \le k \le K-1$.
	
	By the way we generate $\ttx_{t+1}$, we have
	\begin{align}
	\label{eq:mid_term}
	\sum_{j=0}^{m-1}(1-\mu \eta)^{m-j-1} \delta_{j} 
	&=\sum_{j=0}^{m-1}(1-\mu \eta)^{m-j-1} \left[ F(\y_{j}^s) -F(\xc) \right] \nonumber \\
	&\ge \sum_{i=0}^{m-1} (1-\mu \eta)^i   \left[ F(\ttx_{s+1}) -F(\xc) \right].
	\end{align}
	According to the algorithm, we set $\x_{0}^{s+1} = \PAT(\x_m^s)$ and thus $\y_0^{s+1}=\y_m^s$.
	Plugging~\eqref{eq:mid_term} into~\eqref{eq:y10.5} and rearranging complete the proof.
\end{proof}

\subsection{Proof of the generally convex case in Theorem~\ref{thm:svrg}}
\begin{proof}
	Letting $\mu =0$, under the conditions, Lemma~\ref{lem:error_svrg_stage} gives
	\begin{align*}
	\EB \left[F(\ttx_{s+1}) - F(\xc)\right] \le 
	\frac{2}{\eta m} \left[ \EB\|\y_{0}^s -\xc\|^2 - \EB\|\y_{0}^{s+1} -\xc\|^2 \right] + 
	12L\eta (2E-1)\EB[F(\ttx_s) - F(\xc)].
	\end{align*}
	Summing~the last inequality over $s=0, \cdots S-1$ and telescoping, we obtain
	\begin{align*}
	\sum_{s=0}^{S-1}  \EB\left[F(\ttx_{s+1}) - F(\xc)\right]
	&\le \frac{2}{\eta m} \EB\|\y_{0}^s -\xc\|^2 + 12L\eta (2E-1) \sum_{s=0}^{S-1}   \EB[F(\ttx_s) - F(\xc)] \nonumber\\
	&\le \frac{2}{\eta m} \EB\|\y_{0}^s -\xc\|^2 + 12L\eta (2E-1)  \EB[F(\ttx_0) - F(\xc)] \nonumber \\
	&\qquad + 12L\eta (2E-1) \sum_{s=1}^{S}   \EB[F(\ttx_s) - F(\xc)] \nonumber \\
		&\le \frac{2}{\eta m} \EB\|\y_{0}^s -\xc\|^2 + \frac{1}{2}  \EB[F(\ttx_0) - F(\xc)] + \frac{1}{2} \sum_{s=1}^{S}   \EB[F(\ttx_s) - F(\xc)].
	\end{align*}
	where the last inequality requires $12L\eta (2E-1) \le \frac{1}{2}$.
	Dividing $S$ on the both sides of the last inequality  and rearranging, we obtain
	\[
		\EB\left[F(\hat{\y}) - F(\xc)\right] 
	\le  \frac{1}{S} \sum_{s=0}^{S-1}  \EB\left[F(\ttx_{s+1}) - F(\xc)\right] \le \frac{4}{\eta m S} \EB\|\x_{0} -\xc\|^2 + \frac{1}{S} \EB[F(\x_0) - F(\xc)]
	\]
	where we use $\hat{\y} = \frac{1}{S}\sum_{s=1}^{S} \ttx_{s}$ and the convexity of $F(\cdot)$.
	By setting
	\[
	\eta = \min\left\{ \frac{1}{\mu+25L(E-1)}, \frac{1}{10L}, \frac{1}{L(3+2E)} , \frac{1}{24L(2E-1)}\right\},
	\]
	we have $\eta = \Theta(\frac{1}{LE})$, then
	\[
		\EB\left[F(\hat{\y}) - F(\xc)\right]  = \OM\left(  \frac{LE \Delta^2}{T}  + \frac{\EB[F(\x_0) - F(\xc)]}{S}\right)
	\]
	where $\Delta^2 = \EB \|\x_0 - \xc\|^2$ and $T=mS$.
\end{proof}

\subsection{Proof of the strongly convex case in Theorem~\ref{thm:svrg}}
\begin{proof}
For simplicity, let $F_s = \EB[F(\ttx_s) - F(\xc)]$, $\delta = \frac{12L\eta (2E-1) }{1-\mu \eta}$ and $y_s = \EB\|\y_{0}^s -\xc\|^2$.
Recall that $\rho = (1-\mu\eta)^m$ and $\Gamma = \sum_{i=0}^{m-1} (1-\mu \eta)^i$.

When $\mu > 0$, under the conditions, Lemma~\ref{lem:error_svrg_stage} shows
\begin{equation*}
F_{s+1} \le  \frac{2}{\eta \Gamma} \left[ \rho y_s - y_{s+1} \right] +  \delta F_s.
\end{equation*}
Let $ \theta= \max\{\delta, \rho\}$.
Multiplying the both sides of the last inequality by $\theta^{S-s-1}$ and recurring the result  from $s=0$ to $S-1$,  we obtain
	\begin{align*}
F_S 
&\le \theta^S F_0 +  \frac{2}{\eta\Gamma} \sum_{s=0}^{S-1} \theta^{S-s-1} \left[ \rho y_s - y_{s+1} \right] \\
&= \theta^S F_0 + 
\frac{2}{\eta\Gamma} \left[    \theta^{S-1} \rho y_0  -y_S +  \sum_{s=1}^{S-1}  \theta^{S-s-1}  y_s \left(\rho - \theta\right) \right]\\
&\le \theta^S F_0  + \frac{2}{\eta\Gamma}  \theta^{S} y_0\\
&\le \left(F_0 + \frac{2}{\eta} y_0 \right) \theta^S
\end{align*}
where the last inequality uses $\Gamma \ge 1$.

By letting 
\[
	\eta = \min\left\{ \frac{1}{\mu+25L(E-1)}, \frac{1}{10L}, \frac{1}{L(3+2E)} , \frac{1}{\mu + 24L(2E-1)}\right\},
\]
we have $\eta = \Theta(\frac{1}{LE})$ and $\delta \le 0.5$. Then
\begin{align*}
\theta 
&= \max\{\delta, \rho\}
\le  \max \left\{ (1-\mu\eta)^m, 0.5  \right\}\\
&\le \max \left\{ \exp(-m\mu\eta), \exp (- 0.5 ) \right\}
=\exp\left( - \min\left\{ m\mu \eta, 0.5 \right\} \right),
\end{align*}
we have
\begin{align*}
\EB[F(\ttx_S) - F(\xc)]
&\le \left[   \frac{2}{\eta}\Delta^2 + \EB[F(\x_0) - F(\xc)] \right]  \cdot \exp\left( - \min\left\{ m\mu \eta, 0.5 \right\} S \right) \\
&= \OM \left( \left[   L E\Delta^2 + \EB[F(\x_0) - F(\xc)] \right] \cdot 
\exp\left( -  \Theta \left(\frac{T}{\max\{ \kappa E,  m \}} \right) \right)\right)
\end{align*}
where we use $\kappa = \frac{L}{\mu}$ and $T = mS$.

\end{proof}

\section{Proof of Delayed Projected Accelerated SVRG}
\subsection{Descent Lemma}
In this section, we give a decent lemma for Algorithm~\ref{alg:multi_acc_SVRG} that is a counter part of Lemma~\ref{lem:y_descent_1}.
However, since we incorporate the acceleration technique, we can't directly apply Lemma~\ref{lem:y_descent_general}.
\begin{lem}[Three-points lemma]
	\label{lem:three-point}
	Let $L(\cdot)$ is a proper convex (not necessarily differentiable) function, whose domain is an open set containing $\CM$.
	Let $\z^*$ be the minimizer of the following problem
	\[
	\z^*  = \argmin_{\z \in \CM} \left[  L(\z) + \frac{1}{2} \| \z - \z_0\|^2 \right].
	\]
	Then, for any point $\z \in \CM$, we have
	\[
	L(\z)+ \frac{1}{2}\|\z - \z_0\|^2 \ge  L(\z^*) +  \frac{1}{2}\|\z^* - \z_0\|^2  + \frac{1}{2}\|\z - \z^*\|^2.
	\]
\end{lem}
\begin{proof}
	By the first order condition of $\z^*$, there must be a subgradient $\g \in \partial L(\z^*)$ such that 
	\begin{equation}
	\label{eq:first-order}
	\langle \g + \z^*  - \z_0, \z - \z^*  \rangle \ge 0 \ \text{for} \ \z \in \CM.
	\end{equation}
	Therefore, using the property of subgradient, we have
	\[
	L(\z) 
	\ge L(\z^*) + \langle \g, \z - \z^* \rangle 
	\overset{\eqref{eq:first-order}}{\ge}  L(\z^*) + 	\langle \z_0 - \z^* , \z - \z^*  \rangle
	= L(\z^*) +  \frac{1}{2}  \left[ \|\z^* - \z_0\|^2  + \|\z - \z^*\|^2 -\|\z - \z_0\|^2 \right].
	\]
\end{proof}

\begin{lem}[Bounded gradient variance]
	\label{lem:var_acc_svrg}
	Let $\g_t^s = \nabla F(\x_{t}^s; \xi_{t}^s) - \nabla F(\ttx_s; \xi_{t}^s) + \PAT(\nabla F(\ttx_{s}))$.
Then, $\PAT(\g_t^s)$ is unbiased for $\PAT(\nabla F(\x_t^s))$ (i.e., $\EB_{\xi_{t}^s}  \PAT(\g_t^s) =  \PAT(\nabla F(\x_t^s))$) and has bounded variance at most:
\[
\EB_{\xi_{t}^s} \| \PAT(\g_t^s- \nabla F(\x_{t}^s))\|^2
\le  2L \left[ F(\ttx_s) - F( \x_t^s) - \langle \nabla F( \x_t^s), \ttx_s -   \x_t^s \rangle\right].
\]
\end{lem}
\begin{proof}
	The unbiasedness is obvious. 
	For bounded variance, we have
	\begin{align*}
	\EB_{\xi_{t}^s}\| \PAT(\g_t^s- \nabla F(\x_{t}^s))\|^2
	&=\EB_{\xi_{t}^s} \| \PAT\left(\nabla F(\x_{t}^s; \xi_{t}^s) - \nabla F(\ttx_s; \xi_{t}^s) \right) - \PAT\left(\nabla F(\x_{t}^s) - \nabla F(\ttx_s) \right)\|^2\\
	&\overset{(a)}{\le} \EB_{\xi_{t}^s} \| \PAT\left(\nabla F(\x_{t}^s; \xi_{t}^s) - \nabla F(\ttx_s; \xi_{t}^s) \right) \|^2\\
	&\le \EB_{\xi_{t}^s} \| \nabla F(\x_{t}^s; \xi_{t}^s) - \nabla F(\ttx_s; \xi_{t}^s) \|^2 \\
	&\overset{(b)}{\le}  2L \left[ F(\ttx_s) - F( \x_t^s) - \langle \nabla F( \x_t^s), \ttx_s -   \x_t^s \rangle\right].
	\end{align*}
	where (a) uses $\EB\|X-\EB X\|^2 \le \EB \|X\|^2$ for any random vector $X$; and (b) uses Lemma~\ref{lem:F}.
\end{proof}

\begin{lem}
	\label{lem:y_descent_2}
	Under Assumption~\ref{asmp:smooth} and~\ref{asmp:strong}, for Algorithm~\ref{alg:multi_acc_SVRG}, if the learning rate satisfies 
	\begin{equation}
	\label{eq:lr_acc_svrg}
	0 < \eta L \le \frac{1}{2}, 1-\theta_s \ge \frac{\eta L}{1- \eta L},
	\end{equation}
	then for Algorithm~\ref{alg:multi_acc_SVRG} it follows that
			\begin{align*}
	\EB  \left[F(\PAT(\x_{t+1}^s)) - F(\xc)\right] 
&\le  (1-\theta_s) \EB  \left[F(\ttx_{s}) - F(\xc)\right]  + \frac{3L}{2}  \EB \|\PA(\x_t^s)\|^2 \\
&\qquad + \frac{\theta_s^2}{2\eta} \EB\left[\| \PAT(\u_{t}^s) - \xc \|^2 - \| \PAT(\u_{t+1}^s) - \xc \|^2 \right] .
	\end{align*}
\end{lem}
\begin{proof}
	By $L$-smoothness of $F(\cdot)$ and setting $\beta_1 - \beta_2 = 1$ where $\beta_1 = \frac{1}{\eta L}$ and $\eta$ is sufficiently small (guaranteed by~\eqref{eq:lr_acc_svrg}) such that $\beta_2 > 0$, we have
	\begin{align}
	\label{eq:ay1}
	F(\PAT(\x_{t+1}^s))
	&\le F(\PAT(\x_{t}^s)) + \langle \nabla F(\PAT(\x_{t}^s)), \PAT(\x_{t+1}^s-\x_{t}^s) \rangle + \frac{L}{2} \| \PAT(\x_{t+1}^s-\x_{t}^s) \|^2 \nonumber \\
	&= F(\PAT(\x_{t}^s)) + \langle  \PAT(\nabla F(\PAT(\x_{t}^s))), \PAT(\x_{t+1}^s-\x_{t}^s) \rangle + \frac{L}{2} \| \PAT(\x_{t+1}^s-\x_{t}^s) \|^2 \nonumber \\
	&= F(\PAT(\x_{t}^s))  +\underbrace{\langle \PAT(\g_t^s), \PAT(\x_{t+1}^s-\x_{t}^s) \rangle +\frac{L\beta_1}{2} \| \PAT(\x_{t+1}^s-\x_{t}^s) \|^2 }_{I} \nonumber \\
	&\qquad +\underbrace{\langle \PAT( \nabla F(\PAT(\x_{t}^s))-\g_t^s), \PAT(\x_{t+1}^s-\x_{t}^s) \rangle -\frac{L\beta_2}{2} \| \PAT(\x_{t+1}^s-\x_{t}^s) \|^2 }_{II}.
	\end{align}

	We will use $\EB_{\xi_{t}^s}(\cdot)$ to denote that we condition on all randomness before $\xi_t^s$ and take expectation with respect to $\xi_t^s$.
	To bound the $I$ term, we have that
	\begin{align}
	\label{eq:ay2}
 	\EB_{\xi_{t}^s}I &:= \EB_{\xi_{t}^s} \left[\langle \PAT( \g_t^s), \PAT(\x_{t+1}^s-\x_{t}^s) \rangle +\frac{L\beta_1}{2} \| \PAT(\x_{t+1}^s-\x_{t}^s) \|^2 \right]\nonumber \\
	&\overset{(a)}{=}\EB_{\xi_{t}^s} \left[\theta_s\langle \PAT (\g_t^s), \PAT(\u_{t+1}^s-\u_{t}^s) \rangle +\frac{L\beta_1\theta_s^2}{2} \| \PAT(\u_{t+1}^s-\u_{t}^s) \|^2\right] \nonumber \\
	&\overset{(b)}{\le}\EB_{\xi_{t}^s} \left[ \theta_s \langle  \PAT(\g_t^s), \xc - \PAT(\u_{t}^s) \rangle +\frac{L\beta_1\theta_s^2}{2} \left[ \| \PAT(\u_{t}^s) - \xc \|^2 - \| \PAT(\u_{t+1}^s) - \xc \|^2 \right]\right]   \nonumber \\
	&\overset{(c)}{=} \theta_s \langle  \PAT(\nabla F(\x_t^s)), \xc - \PAT(\u_{t}^s) \rangle +\frac{L\beta_1\theta_s^2}{2}\EB_{\xi_{t}^s} \left[ \| \PAT(\u_{t}^s) - \xc \|^2 - \| \PAT(\u_{t+1}^s) - \xc \|^2 \right] \nonumber \\
	&= \theta_s \langle \nabla F(\x_t^s), \xc - \PAT(\u_{t}^s) \rangle +\frac{L\beta_1\theta_s^2}{2}\EB_{\xi_{t}^s} \left[ \| \PAT(\u_{t}^s) - \xc \|^2 - \| \PAT(\u_{t+1}^s) - \xc \|^2 \right].
	\end{align}
	Here (a) uses $\x_{t+1}^s - \x_t^s = \theta_s(\u_{t+1}^s - \u_t^s)$; 
	(b) uses Lemma~\ref{lem:three-point} by setting $L(\z) = \frac{1}{L\beta_1\theta_s} \langle \PAT(\g_t^s) , \z - \PAT(\u_t^s) \rangle, \z_0 = \PAT(\u_t^s)$ and $\z = \PAT(\xc) = \xc$ (in this case we can find that $\z^* = \PAT(\u_{t+1}^s)  =  \PAT(\u_{t}^s - \frac{\eta}{\theta_s} \cdot \g_{t}^s)= \z_0 - \frac{1}{L\beta_1\theta_s}  \PAT(\g_t^s)$ due to $\eta = \frac{1}{L\beta_1}$); and (c) uses $\PAT(\g_t^s)$ is a unbiased estimator for $\PAT(\nabla F(\x_t^s))$ and is independent with $\u_t^s$.
	
	To bound the $II$ term, we have that
	\begin{align}
	\label{eq:ay3}
	\EB_{\xi_{t}^s}	II &:=  \EB_{\xi_{t}^s} \left[\langle \PAT( \nabla F(\PAT(\x_{t}^s))-\g_t^s), \PAT(\x_{t+1}^s-\x_{t}^s) \rangle -\frac{L\beta_2}{2} \| \PAT(\x_{t+1}^s-\x_{t}^s) \|^2\right] \nonumber \\
	&\overset{(a)}{\le}  \EB_{\xi_{t}^s} \left[ \frac{1}{2L\beta_2}\|\PAT(\nabla F(\PAT(\x_{t}^s))-\g_t^s)\|^2  + \frac{L\beta_2}{2}   \| \PAT(\x_{t+1}^s-\x_{t}^s) \|^2  - \frac{L\beta_2}{2} \| \PAT(\x_{t+1}^s-\x_{t}^s) \|^2  \right]\nonumber \\
	&= \frac{1}{2L\beta_2} \EB_{\xi_{t}^s}\|\PAT(\nabla F(\PAT(\x_{t}^s))-\g_t^s)\|^2 \nonumber \\
	&\overset{(b)}{\le} \frac{1}{L\beta_2}  \EB_{\xi_{t}^s} \left[ \|\PAT(\nabla F(\x_{t}^s))-\g_t^s)\|^2 + \|\PAT(\nabla F(\PAT(\x_{t}^s)-\nabla F(\x_{t}^s))\|^2     \right]\nonumber \\
	&\overset{(c)}{\le}  \frac{2}{\beta_2} \left[  F(\ttx_s) - F( \x_t^s)) - \langle \nabla F(\x_t^s), \ttx_s -   \x_t^s \rangle \right]  + \frac{1}{L\beta_2}  \|\nabla F(\PAT(\x_{t}^s))-\nabla F(\x_{t}^s)\|^2   \nonumber \\
	&\overset{(d)}{\le} \frac{2}{\beta_2} \left[  F(\ttx_s) - F( \x_t^s)- \langle \nabla F( \x_t^s), \ttx_s -   \x_t^s \rangle \right] + \frac{L}{\beta_2}\|\PA(\x_{t}^s)\|^2.
	\end{align}
	Here (a) uses Lemma~\ref{lem:ineq1}; (b) uses Lemma~\ref{lem:j} with $n=2$; (c) uses Lemma~\ref{lem:var_acc_svrg} and (d) uses $L$-smoothness.

	Combing~\eqref{eq:ay1},~\eqref{eq:ay2} and~\eqref{eq:ay3} and  taking expectation with respect to all randomness, we have
	\begin{align}
	\label{eq:ay4}
	\EB F(\PAT(\x_{t+1}^s)) 
	&\le \EB F(\PAT(\x_{t}^s)) + \theta_s\EB\langle \nabla F(\x_t^s), \xc - \PAT(\u_{t}^s) \rangle + \frac{L}{\beta_2}\|\PA(\x_{t}^s)\|^2 \nonumber  \\
	& \qquad +\frac{L\beta_1\theta_s^2}{2} \EB\left[ \| \PAT(\u_{t}^s) - \xc \|^2 - \| \PAT(\u_{t+1}^s) - \xc \|^2 \right] \nonumber \\
	&\qquad + \frac{2}{\beta_2} \left[  F(\ttx_s) - F( \x_t^s) - \langle \nabla F( \x_t^s), \ttx_s -   \x_t^s \rangle \right]  \nonumber \\
	&= \EB F(\PAT(\x_{t}^s)) +\frac{L\beta_1\theta_s^2}{2} \EB\left[ \| \PAT(\u_{t}^s) - \xc \|^2 - \| \PAT(\u_{t+1}^s) - \xc \|^2 \right]  \nonumber  \\
	& \qquad +  \frac{2}{\beta_2} \left[  F(\ttx_s) - F( \x_t^s) \right]+ \frac{L}{\beta_2}\|\PA(\x_{t}^s)\|^2  \nonumber  \\
	& \qquad
	+  \EB\underbrace{\left\langle \nabla F(\x_t^s), \theta_s(\xc - \PAT(\u_{t}^s)) - \frac{2}{\beta_2} (\ttx_s -   \x_t^s)\right\rangle}_{III}
	\end{align}
	
	To bound the $III$ term, we have that
	\begin{align}
	\label{eq:ay5}
	III&:=\left\langle \nabla F(\x_t^s), \theta_s(\xc - \PAT(\u_{t}^s)) - \frac{2}{\beta_2}( \ttx_s -  \x_t^s) \right\rangle \nonumber \\
	&\overset{(a)}{=}  \left\langle \nabla F( \x_t^s), 	\left[\theta_s\xc + \left(1-\theta_s - \frac{2}{\beta_2} \right)  \ttx_s  + \frac{2}{\beta_2}\x_t^s\right] - \PAT(\x_t^s)\right\rangle \nonumber\\
	&=  \left\langle \nabla F( \x_t^s), 	\left[\theta_s\xc + \left(1-\theta_s - \frac{2}{\beta_2} \right)  \ttx_s  + \frac{2}{\beta_2}\x_t^s\right] - \x_t^s\right\rangle \nonumber\\
	& \qquad +  \left\langle \nabla F( \x_t^s), 	\x_t^s - \PAT(\x_t^s) \right\rangle \nonumber\\
	&\overset{(b)}{\le}  F\left(\theta_s\xc + \left(1-\theta_s - \frac{2}{\beta_2} \right)  \ttx_s  + \frac{2}{\beta_2}\x_t^s\right) - F(\x_t^s) \nonumber \\
	& \qquad +  \left\langle \nabla F( \x_t^s), 	\x_t^s - \PAT(\x_t^s) \right\rangle \nonumber\\
	&\overset{(c)}{\le}  F\left(\theta_s\xc + \left(1-\theta_s - \frac{2}{\beta_2} \right)  \ttx_s  + \frac{2}{\beta_2}\x_t^s\right) - F(\x_t^s) \nonumber \\
	& \qquad +  \left[F( \x_t^s) - F(\PAT(\x_t^s)) + \frac{L}{2} \|\x_t^s - \PAT(\x_t^s)\|^2 \right]  \nonumber\\	
	&=  F\left(\theta_s\xc + \left(1-\theta_s - \frac{2}{\beta_2} \right)  \ttx_s  + \frac{2}{\beta_2}\x_t^s\right) - F(\PAT(\x_t^s))  + \frac{L}{2} \|\PA(\x_t^s)\|^2  \nonumber\\	
	&\overset{(d)}{\le}\theta_s F(\xc) + \left(1-\theta_s - \frac{2}{\beta_2} \right) F(\ttx_{s}) + \frac{2}{\beta_2} F(\x_t^s) -F(\PAT(\x_t^s)) + \frac{L}{2} \|\PA(\x_t^s)\|^2 .
	\end{align}
	Here (a) uses $ \x_t^s = (1-\theta_s)\ttx_{s} + \theta_s \u_t^s$ that implies (note that $\ttx_s, \xc \in \mathcal{R}(A^{\perp})$)
	\[
	\theta_s(\xc - \PAT(\u_{t}^s)) - \frac{2}{\beta_2} \left( \ttx_s -   \x_t^s\right) = 
	\left[\theta_s\xc + \left(1-\theta_s - \frac{2}{\beta_2} \right)  \ttx_s  + \frac{2}{\beta_2}\x_t^s\right] - \PAT(\x_t^s);
	\]
	(b) uses the convexity of $F(\cdot)$; (c) uses the $L$-smoothness of $F(\cdot)$; (d) uses Jensen's inequality for convex functions where we ensure $1-\theta_s \ge \frac{2}{\beta_2}$ by~\eqref{eq:lr_acc_svrg}.

	Substituting~\eqref{eq:ay5} into~\eqref{eq:ay4} yields
	\begin{align*}
	\EB  \left[F(\PAT(\x_{t+1}^s)) - F(\xc)\right] 
	&\le  (1-\theta_s) \EB  \left[F(\ttx_{s}) - F(\xc)\right]  + \left( \frac{L}{2} + \frac{L}{\beta_2} \right)  \EB \|\PA(\x_t^s)\|^2 \\
	&\qquad + \frac{L\beta_1\theta_s^2}{2} \EB\left[\| \PAT(\u_{t}^s) - \xc \|^2 - \| \PAT(\u_{t+1}^s) - \xc \|^2 \right] .
	\end{align*}
	Setting $\eta$ as in~\eqref{eq:lr_acc_svrg} means $\eta L \le \frac{1}{2}$ and thus
	\[
	 \frac{L}{2} + \frac{L}{\beta_2}  =  \frac{L}{2} + \frac{L}{1-\eta L}  \eta L \le \frac{3L}{2}.
	\]
	Plugging the expression of $\beta_1= \frac{1}{\eta L}$ completes the proof.
\end{proof}

\subsection{Residual Lemma}
In this section, we provide a residual lemma for the accelerated algorithm.
Instead of using incremental analysis that upper bounds $\EB \|\PA(\x_{t+1}^s)\|^2$ in terms of $\EB \|\PA(\x_{t}^s)\|^2$ (for example, see Lemma~\ref{lem:z_descent_1}), we make use of a decomposition analysis that splits $\PA(\x_{t+1}^s)$ into a sum of previous stochastic gradients and then bound each variance term.
As a result, we focus the the average residual error, i.e., $\frac{1}{E}\sum_{t=1}^{E} \EB\|\PA(\x_t^s)\|^2$.

\begin{lem}
	\label{lem:z_descent_1.9}
	Let $0=t_0 < t_1 < t_2 < \cdots < t_K = m$ be the elements of $\IM_m$ and denote $s_k = t_{k+1} - t_{k}$ (so that $\sum_{i=0}^{K-1}s_i = m$ and $s_k \le E$).
	Under Assumption~\ref{asmp:smooth} and~\ref{asmp:strong}, when $E(E-1)\eta^2L^2 \le \frac{1}{3}$, we have for any $0 \le k \le K-1$, 
	\[
	\frac{1}{s_k}\sum_{t=t_k+1}^{t_{k+1}} \EB\|\PA(\x_t^s)\|^2
	\le 6E(E-1)\eta^2L \left[ \frac{1}{s_k}\sum_{t=t_k}^{t_{k+1}-1}\EB  \left[F(\PAT(\x_{t}^s)) - F(\xc) \right] + F(\tx_s) - F(\xc) \right].
	\]
\end{lem}
\begin{proof}
	Without loss of generality, we assume $0=t_0 \le t \le t_1-1$.
	Recall $\g_t^s = \nabla F(\x_{t}^s; \xi_{t}^s) - \nabla F(\ttx_s; \xi_{t}^s) + \PAT(\nabla F(\ttx_{s}))$. 
	Then it follows that
	\begin{align}
	\label{eq:g_a}
	\EB\|\PA(\g_t^s)\|^2 
	&=\EB\|\PA(    \nabla F(\x_{t}^s; \xi_{t}^s) - \nabla F(\ttx_s; \xi_{t}^s)  )\|^2   \nonumber \\
	&\le \EB\| \nabla F(\x_{t}^s; \xi_{t}^s) - \nabla F(\ttx_s; \xi_{t}^s) \|^2  \nonumber \\
	&\le 3 \EB\| \nabla F(\x_{t}^s; \xi_{t}^s) - \nabla F(\PAT(\x_{t}^s); \xi_{t}^s) \|^2 + 3 \EB\| \nabla F(\PAT(\x_{t}^s); \xi_{t}^s) - \nabla F(\xc; \xi_{t}^s) \|^2  \nonumber \\
	&\qquad +3 \EB\| \nabla F(\ttx_s; \xi_{t}^s) - \nabla F(\xc; \xi_{t}^s) \|^2  \nonumber \\
	&\le 3L^2\EB\|\PA(\x_t^s)\|^2 + 6L\EB \left[ F(\PAT(\x_{t}^s)) - F(\xc) \right] +  6L\EB \left[ F(\tx_s) - F(\xc) \right]
	\end{align}
	where the last inequality uses $L$-smoothness of $F(\cdot)$ and Lemma~\ref{lem:F}.
	
	Note that $\x_{t+1}^s - \x_t^s = \theta_s(\u_{t+1}^s - \u_t^s) =  - \eta \g_t^s$, which implies the accelerated algorithm has the same one-step update rule as the non-accelerated one and thus we can make use of Lemma~\ref{lem:z_descent_1} to give a residual lemma here.
	However, we apply a decomposition analysis here since the descent lemma is quite different.
	In particular, recursion gives 
	\[
	\PA(\x_{t+1}^s) 
	= \PA(\x_t^s - \eta \g_t^s) 
	= \PA\left(\x_0^s-\eta\sum_{\tau = 0}^t \g_\tau^s\right)
	= - \eta\sum_{\tau = 0}^t \PA(\g_\tau^s),
	\]
	by which we decompose $\PA(\x_{t+1}^s)$ into a sum of previous stochastic gradients.
	Summing $\EB\|\PA(\x_{t+1}^s) \|^2$ over $t=0, \cdots, E-1$ and using the decomposition give
	\begin{align}
	\label{eq:PAx}
	\sum_{t=0}^{t_1-1}\EB\|\PA(\x_{t+1}^s) \|^2
	&=\sum_{t=0}^{t_1-1}\EB\bigg\|\eta\sum_{\tau = 0}^t \PA(\g_\tau^s)\bigg\|^2 \nonumber \\
	&\overset{(a)}{\le}\eta^2 \sum_{t=0}^{t_1-1} (t+1)\sum_{\tau = 0}^t  \EB \| \PA(\g_\tau^s)\|^2 \nonumber \\
	&=\eta^2  \sum_{\tau=0}^{t_1-1}\sum_{t=\tau}^{t_1-1}(t+1) \EB \| \PA(\g_\tau^s)\|^2 \nonumber \\
	&\overset{(b)}{\le}  \frac{E(E-1)\eta^2}{2} \sum_{\tau=0}^{t_1-1} \EB \| \PA(\g_\tau^s)\|^2
	\end{align}
	where (a) uses Lemma~\ref{lem:j}; and (b) uses $\sum_{t=\tau}^{t_1-1}(t+1) \le \sum_{t=0}^{t_1-1}(t+1) = \frac{t_1(t_1-1)}{2} \le \frac{E(E-1)}{2}$ for any $0 \le \tau \le E-1$.
	Plugging~\eqref{eq:g_a} into~\eqref{eq:PAx} gives
	\begin{align*}
	\sum_{t=0}^{t_1-1} &\EB\|\PA(\x_{t+1}^s) \|^2\\
	&\le \frac{3E(E-1)\eta^2}{2}\sum_{t=0}^{t_1-1}
	\left[  L^2\EB\|\PA(\x_t^s)\|^2 + 2L\EB \left[ F(\PAT(\x_{t}^s)) - F(\xc) + F(\tx_s) - F(\xc) \right]  \right]\\
	&\overset{(a)}{\le} \frac{3E(E-1)\eta^2L^2}{2}\sum_{t=1}^{t_1} \EB\|\PA(\x_t^s)\|^2 +
	3E(E-1)\eta^2L\sum_{t=0}^{t_1-1}\EB \left[ F(\PAT(\x_{t}^s)) - F(\xc) + F(\tx_s) - F(\xc) \right]\\
	&\overset{(b)}{\le} \frac{1}{2}\sum_{t=1}^{t_1} \EB\|\PA(\x_t^s)\|^2 +
	3E(E-1)\eta^2L\sum_{t=0}^{t_1-1}\EB \left[ F(\PAT(\x_{t}^s)) - F(\xc) + F(\tx_s) - F(\xc) \right]
	\end{align*}
	where (a) uses $\PA(\x_0^s) = \0$; and (b) follows by setting $E(E-1)\eta^2L^2 \le \frac{1}{3}$.
	
	Finally we complete the proof of $k=0$ by arranging the last inequality.
	It is natural to extend the argument to the case that $1 \le k \le K-1$.
\end{proof}

\begin{lem}
\label{lem:z_descent_2}
Under Assumption~\ref{asmp:smooth} and~\ref{asmp:strong}, when $E(E-1)\eta^2L^2 \le \frac{1}{3}$ and $m \ge E$, we have
\[
\frac{1}{m}\sum_{t=1}^{m} \EB\|\PA(\x_t^s)\|^2
\le 6(E^2-1)\eta^2L \left[ \frac{1}{m}\sum_{t=0}^{m-1}\EB  \left[F(\PAT(\x_{t}^s)) - F(\xc) \right] + F(\tx_s) - F(\xc) \right].
\]
\end{lem}
\begin{proof}
	Let $0=t_0 < t_1 < t_2 < \cdots < t_K = m$ be the elements of $\IM_m$ and denote $s_k = t_{k+1} - t_{k}$ (so that $\sum_{i=0}^{K-1}s_i = m$ and $s_k \le E$).
	By Lemma~\ref{lem:z_descent_1.9}, for $0 \le k \le K-1$, 
	\[
	\frac{1}{s_k}\sum_{t=t_k+1}^{t_{k+1}} \EB\|\PA(\x_t^s)\|^2
	\le 6E(E-1)\eta^2L \left[ \frac{1}{s_k}\sum_{t=t_k}^{t_{k+1}-1}\EB  \left[F(\PAT(\x_{t}^s)) - F(\xc) \right] + F(\tx_s) - F(\xc) \right].
	\]
	Then, weighting the $K$ inequality correspondingly, 
	\begin{align*}
	\frac{1}{m}\sum_{t=1}^{m} \EB\|\PA(\x_t^s)\|^2 
	&= \sum_{k=0}^{K-1}  \frac{s_k}{m}  \cdot \frac{1}{s_k}\sum_{t=t_k+1}^{t_{k+1}} \EB\|\PA(\x_t^s)\|^2\\
	&\le 6E(E-1)\eta^2L \sum_{k=0}^{K-1}  \frac{s_k}{m} \left[ \frac{1}{s_k}\sum_{t=t_k}^{t_{k+1}-1}\EB  \left[F(\PAT(\x_{t}^s)) - F(\xc) \right] + F(\tx_s) - F(\xc) \right]\\
	&=6E(E-1)\eta^2L \left[ \frac{1}{m}\sum_{t=0}^{m-1}\EB  \left[F(\PAT(\x_{t}^s)) - F(\xc) \right] + F(\tx_s) - F(\xc) \right]\\
	&\overset{(a)}{\le}6E(E-1)\eta^2L \left[ \frac{1}{m}\sum_{t=1}^{m}\EB  \left[F(\PAT(\x_{t}^s)) - F(\xc) \right] + \left(1+\frac{1}{m}\right) \cdot F(\tx_s) - F(\xc) \right]\\
	&\le 6E(E-1)\eta^2L\left(1+\frac{1}{m}\right)  \left[ \frac{1}{m}\sum_{t=1}^{m}\EB  \left[F(\PAT(\x_{t}^s)) - F(\xc) \right] + F(\tx_s) - F(\xc) \right] \\
	&\overset{(b)}{\le} 6(E^2-1)\eta^2L  \left[ \frac{1}{m}\sum_{t=1}^{m}\EB  \left[F(\PAT(\x_{t}^s)) - F(\xc) \right] + F(\tx_s) - F(\xc) \right]
	\end{align*}
where (a) uses  $\x_0^s = \ttx_{s}$ and (b) uses $m \ge E$.

\end{proof}

\subsection{Other Helper Lemmas}

\begin{lem}[A useful auxiliary sequence]
	\label{lem:theta}
	Given a positive number $\delta \in [0, 1)$, we define a positive sequence $\{\theta_s\}_{s=0}^{\infty}$. 
	The initial point is set as $2 \delta \le \theta_0 \le 1+\delta$.
	Given $\theta_{s-1}$, $\theta_{s}$ is generated by
	\[
	 \frac{1-\theta_s+\delta}{1-\delta} \cdot \frac{1}{\theta_s ^2} = \frac{1}{\theta_{s-1}^2}.
	\]
	If the quadratic has two different roots, we define $\theta_{s}$ as the larger one.
	Then it follows that
\begin{enumerate}
	\item (Boundedness) For any $s \ge 0$, $2\delta \le \theta_s \le 1+\delta$;
	\item (Monotonicity) For any $s \ge 0$, $ \theta_{s+1} \le \theta_s$;
	\item (Contraction) For any $s \ge 0$, $0 \le \theta_{s+1} - 2\delta \le (1-\delta)(\theta_{s} - 2\delta)$.
	\item (Convergence rate) For $s \ge 0$, if $\delta = \frac{\ln(s+1)}{s+1}$, then $\theta_{s}= \widetilde{\OM}\left(\frac{1}{s}\right)$; if  $\delta = 0$, then $\theta_{s} \le \frac{2}{2+s}  = \OM(\frac{1}{s})$.
\end{enumerate}
\end{lem}
\begin{proof}
	\begin{enumerate}
		\item 	$\theta_{s}$ can be obtained by solving the quadratic $\frac{\theta_s^2}{\theta_{s-1}^2}  + \frac{\theta_s}{1-\delta}  - \frac{1+\delta}{1-\delta}=0$, so
		\[
		 \theta_s = \sqrt{ \frac{1+\delta}{1-\delta}\theta_{s-1}^2 + \frac{\theta_{s-1}^4}{4(1-\delta)^2}} - \frac{\theta_{s-1}^2}{2(1-\delta)} 
		 = \frac{2(1+\delta)}{1 + \sqrt{1 + \frac{4(1-\delta^2)}{\theta_{s-1}^2}}}
		\]
		where we eliminate the negative root since we define $\theta_{s}$ as the larger root.
		
		From the last expression, we have $\theta_{s} \le 1 + \delta$. Besides, the sign of 
		\begin{align}
		\label{eq:theta}
		\theta_{s} - 2\delta 
		&= \frac{2(1+\delta)}{1 + \sqrt{1 + \frac{4(1-\delta^2)}{\theta_{s-1}^2}}} - 2\delta
		=2\cdot \frac{1 -\delta \sqrt{1 + \frac{4(1-\delta^2)}{\theta_{s-1}^2}}}{1 + \sqrt{1 + \frac{4(1-\delta^2)}{\theta_{s-1}^2}}}  \nonumber \\
		&=2\cdot \frac{(1-\delta^2)\left(  1- \frac{4\delta^2}{\theta_{s-1}^2} \right)}{ \left[ 1 + \sqrt{1 + \frac{4(1-\delta^2)}{\theta_{s-1}^2}} \right] \left[  1 +\delta \sqrt{1 + \frac{4(1-\delta^2)}{\theta_{s-1}^2}} \right] } 
		\end{align}
		is determined by the sign of $1- \frac{4\delta^2}{\theta_{s-1}^2} $. Given $\theta_{s-1} \ge 2\delta$, we have $\theta_{s} \ge 2\delta$.
		\item Noting $\theta_{s} \ge 2\delta$, the item directly follows from
		\[
		\frac{\theta_s^2}{\theta_{s-1}^2}  = \frac{1+\delta-\theta_s}{1-\delta} \le 1.
		\]
		\item Rearranging~\eqref{eq:theta} gives 
		\begin{align*}
		\theta_{s} - 2\delta 
		&=
		 2\cdot \frac{(1-\delta^2)\left( \theta_{s-1}^2 - 4\delta^2 \right)}{ \left[ \theta_{s-1} + \sqrt{\theta_{s-1}^2 + 4(1-\delta^2)} \right] \left[ \theta_{s-1} + \delta\sqrt{\theta_{s-1}^2 + 4(1-\delta^2)} \right]} \\
		 		&= \frac{2(1+\delta)}{ \theta_{s-1} + \sqrt{\theta_{s-1}^2 + 4(1-\delta^2)}} \cdot 
		 		\frac{\theta_{s-1} + 2\delta}{\theta_{s-1} + \delta\sqrt{\theta_{s-1}^2 + 4(1-\delta^2)} } \cdot(1-\delta) (\theta_{s-1} - 2\delta)\\
		&\le \frac{2(1+\delta)}{ 2\delta + \sqrt{(2\delta)^2 + 4(1-\delta^2)}} \cdot 
		\frac{\theta_{s-1} + 2\delta}{\theta_{s-1} + \delta\sqrt{(2\delta)^2 + 4(1-\delta^2)} } \cdot(1-\delta)  (\theta_{s-1} - 2\delta)\\
		&=(1-\delta)  (\theta_{s-1} - 2\delta)
		\end{align*}
where the inequality uses $\theta_{s-1} \ge 2 \delta$.
\item If $\delta > 0$, by using the contraction property and recursion, we have 
\[
\theta_{s} 
\le 2\delta + (1-\delta)^s(\theta_0 - 2\delta)
\le 2\delta + (1-\delta)^{s+1}
\le 2\delta + \exp(-(s+1)\delta).
\]
Then we choose $\delta$ to minimize the RHS for a given $s$.
To that end, we let $\delta = \frac{\ln(s+1)}{s+1}$, so 
\[
\theta_{s}  \le 2 \frac{\ln(s+1)}{s+1} + \frac{1}{s+1} =\widetilde{\OM}\left(\frac{1}{s}\right).
\]

If $\delta = 0$, we prove $\theta_s \le \frac{2}{s+2}$ by induction. 
The case of $s=0$ follows from the boundedness.
Suppose we have $\theta_{s-1} \le \frac{2}{s+1}$ already, then by~\eqref{eq:theta},
\[
\theta_{s} = \frac{2}{1 + \sqrt{1 + \frac{4}{\theta_{s-1}^2}} }
\le \frac{2}{1 + \sqrt{1 + (s+1)^2  }} \le \frac{2}{s+2} = \OM\left(\frac{1}{s}\right).
\]
	\end{enumerate}
\end{proof}

\begin{lem}[Stage-wise error propagation]
	\label{lem:error_acc_svrg_stage}
Under Assumption~\ref{asmp:smooth} and~\ref{asmp:strong}, when $m \ge E$, letting $\delta = 9(E^2-1)\eta^2L^2 < 1$ and~\eqref{eq:lr_acc_svrg} holds and setting $\u_0^{s+1} = \PAT(\u_E^{s})$, we have
\begin{align*}
\EB &\left[F(\ttx_{s+1}) - F(\xc)\right]\\
&\le \frac{1-\theta_s + \delta }{1-\delta} \left[F(\ttx_{s}) - F(\xc)\right]  +  \frac{\theta_s^2}{2\eta m(1-\delta)} \left[ \EB \| \PAT(\u_{0}^s) - \xc \|^2  - \EB \| \PAT(\u_{0}^{s+1}) - \xc \|^2 \right].
\end{align*}
\end{lem}
\begin{proof}
For notation simplicity, let $\delta = 9(E^2-1)\eta^2L^2 < 1$ that makes the condition of Lemma~\ref{lem:z_descent_2} holds.
Condition~\eqref{eq:lr_acc_svrg} ensure Lemma~\ref{lem:y_descent_2} holds.
Averaging the result of Lemma~\ref{lem:y_descent_2} over $t=0, \cdots, m-1$ gives
\begin{align}
\label{eq:y_acc_0}
\frac{1}{m} &\sum_{t=0}^{m-1}	\EB  \left[F(\PAT(\x_{t+1}^s)) - F(\xc)\right]  \nonumber \\
&\le  (1-\theta_s) \EB  \left[F(\ttx_{s}) - F(\xc)\right]  + \frac{3L}{2} \frac{1}{m} \sum_{t=0}^{m-1} \EB \|\PA(\x_t^s)\|^2  \nonumber \\
&\qquad + \frac{\theta_s^2}{2\eta} \frac{1}{m} \sum_{t=0}^{m-1}  \EB\left[\| \PAT(\u_{t}^s) - \xc \|^2 - \| \PAT(\u_{t+1}^s) - \xc \|^2 \right]  \nonumber \\ 
&\overset{(a)}{=}  (1-\theta_s) \EB  \left[F(\ttx_{s}) - F(\xc)\right]  + \frac{3L}{2} \frac{1}{m} \sum_{t=1}^{m} \EB \|\PA(\x_t^s)\|^2  \nonumber \\
&\qquad + \frac{\theta_s^2}{2\eta m}  \EB\left[\| \PAT(\u_{0}^s) - \xc \|^2 - \| \PAT(\u_{m}^s) - \xc \|^2 \right]  \nonumber \\
&\overset{(b)}{\le}  (1-\theta_s) \EB  \left[F(\ttx_{s}) - F(\xc)\right]  +  \frac{\theta_s^2}{2\eta m}  \EB\left[\| \PAT(\u_{0}^s) - \xc \|^2 - \| \PAT(\u_{m}^s) - \xc \|^2 \right]  \nonumber \\
&\qquad + 9(E^2-1)\eta^2L^2 \left[ F(\tx_s) - F(\xc) + \frac{1}{m}\sum_{t=1}^{m}\EB  \left[F(\PAT(\x_{t}^s)) - F(\xc) \right]  \right]
\end{align}
where (a) uses telescoping and $\PAT(\u_0^s) = \PAT(\u_m^{s-1})$ and (b) uses Lemma~\ref{lem:z_descent_2}.

Recall that $\ttx_{s+1} = \frac{1}{m}\sum_{t=1}^m\PAT(\x_t^s)$.
Arranging~\eqref{eq:y_acc_0} and using Jensen's inequality give
\begin{align*}
\EB &\left[F(\ttx_{s+1}) - F(\xc)\right]
\le \frac{1}{m} \sum_{t=0}^{m-1}	\EB  \left[F(\PAT(\x_{t+1}^s)) - F(\xc)\right] \\
&\le \frac{1-\theta_s + \delta }{1-\delta} \left[F(\ttx_{s}) - F(\xc)\right]  +  \frac{\theta_s^2}{2\eta m(1-\delta)} \left[ \EB \| \PAT(\u_{0}^s) - \xc \|^2  - \EB \| \PAT(\u_{m}^s) - \xc \|^2 \right].
\end{align*}
Finally noting $\u_0^{s+1} = \PAT(\u_m^s)$ finishes the proof.
\end{proof}

\subsection{Proof of Theorem~\ref{thm:acc_svrg_strong}}
\begin{proof}
For the strongly convex case $(\mu > 0)$, by Lemma~\ref{lem:error_acc_svrg_stage} and setting $\theta_{s} \equiv \theta$, we have
\begin{align*}
\EB &\left[F(\ttx_{s+1}) - F(\xc)\right]\\
&\le \frac{1-\theta + \delta }{1-\delta} \EB\left[F(\ttx_{s}) - F(\xc)\right]  +  \frac{\theta^2}{2\eta m(1-\delta)} \left[ \EB \| \PAT(\u_{0}^s) - \xc \|^2  - \EB \| \PAT(\u_{0}^{s+1}) - \xc \|^2 \right].
\end{align*}
where $\delta = 9(E^2-1)\eta^2L^2$.
By setting $1+\delta>\theta > 2\delta$, we can let
\[
\rho := 1 - \frac{1-\theta + \delta }{1-\delta}  = \frac{\theta - 2\delta }{1-\delta} \in (0, 1].
\]
Denote by $F_{s} = \EB\left[F(\ttx_{s}) - F(\xc)\right] $.
Hence, the last inequality becomes
\[
F_{s+1}
\le  (1-\rho) F_s +  \frac{\theta^2}{2\eta m(1-\delta)} \left[ \EB \| \PAT(\u_{0}^s) - \xc \|^2  - \EB \| \PAT(\u_{0}^{s+1}) - \xc \|^2 \right].
\]
Subtracting $(1-\rho)F_{s+1}$ to both sides of the above inequality, we arrive at
\[
F_{s+1}
\le  \frac{1-\rho}{\rho} (F_s-F_{s+1}) +  \frac{\theta^2}{2\rho\eta m(1-\delta)} \left[ \EB \| \PAT(\u_{0}^s) - \xc \|^2  - \EB \| \PAT(\u_{0}^{s+1}) - \xc \|^2 \right].
\]
Then in $S$ stages, by summing the above inequality overt $s = 0, \cdots, S-1$, we have
\[
\sum_{s=0}^{S-1}F_{s+1}
\le \frac{1-\rho}{\rho}  F_0 +\frac{\theta^2}{2\rho\eta m(1-\delta)}\EB \| \PAT(\u_{0}^0) - \xc \|^2 \le \left[ \frac{1-\rho}{\rho}  + \frac{\theta^2}{\rho\eta\mu m(1-\delta)} \right] F_0.
\]
Choosing the initial vector as $\x_0^{\text{new}} = \frac{1}{S}\sum_{s=1}^{S} \tx_s$ for the restart, we have
\[
\EB \left[ F(\x_0^{\text{new}}) - F(\xc) \right]  \le \frac{\frac{1-\rho}{\rho}  + \frac{\theta^2}{\rho\eta\mu m(1-\delta)}}{S} \cdot \EB \left[ F(\PAT(\x_0)) - F(\xc) \right].
\]
By setting
\begin{equation}
\label{eq:S}
S = 2\left[ \frac{1-\rho}{\rho}  + \frac{\theta^2}{\rho\eta\mu m(1-\delta)} \right] 
= 2\left[ \frac{1-\delta}{\theta - 2\delta} + \frac{\theta^2}{(\theta-2\delta)\eta\mu m} -1\right],
\end{equation}
we have that $\EB \left[ F(\PAT(\x_0)) - F(\xc) \right]$ decreases by a factor of $1/2$ every $S$ epochs.

Then we analyze the upper bound for $S$.
We set $\theta = 2\delta + \sqrt{4\delta^2 + \eta\mu m}$, which is the root of $ \theta^2-4\delta \theta - \eta \mu m = 0$.
Then~\eqref{eq:S} becomes
\begin{align}
\label{eq:new_S}
S
&\le 2\left[ \frac{1}{\theta - 2\delta} + \frac{\theta^2}{(\theta-2\delta)\eta\mu m} \right]
= \frac{4\theta}{\eta\mu m}  \nonumber \\
&= \frac{4 	\left[ 2\delta + \sqrt{4\delta^2 + \eta\mu m} \right]}{\eta\mu m} 
\overset{(a)}{\le}	\frac{4 	\left[ 4\delta + \sqrt{\eta\mu m} \right]}{\eta\mu m}  \nonumber \\
&	\overset{(b)}{=}	144  \frac{E^2-1}{m} \kappa \cdot  \eta L  + \frac{4\sqrt{\kappa}}{\sqrt{\eta L m}}
\end{align}
where (a) uses $\sqrt{a+b} \le \sqrt{a} + \sqrt{b}$ for $a, b \ge 0$ and (b) uses $\delta = 9(E^2-1)\eta^2L^2$.

Obviously $\theta > 2 \delta$. 
To ensure $\theta \le 1 + \delta$, we only need to let $\delta \le \frac{1-\mu\eta m}{1 + \sqrt{1+3(1-\mu \eta m)}}$ that is guaranteed by $3\delta + \mu \eta m \le 1$, i.e.,
\[
 27 (E^2 - 1)\eta^2L^2 + \frac{m}{\kappa} \cdot \eta L \le 1
\implies
\eta L  \le   \frac{2}{\frac{m}{\kappa}+\sqrt{\left(\frac{m}{\kappa}\right)^2 + 108(E^2-1)}}.
\]
To ensure~\eqref{eq:lr_acc_svrg}, we only need to let 
\[
2\eta L + 27 (E^2 - 1)\eta^2L^2 \le 1
\implies
\eta L  \le   \frac{1}{1 + \sqrt{1+27(E^2-1)}}.
\]

To minimize the RHS of~\eqref{eq:new_S}, we let
\[
\eta L = \min \left\{ \frac{2}{\frac{m}{\kappa}+\sqrt{\left(\frac{m}{\kappa}\right)^2 + 108(E^2-1)}}, \frac{1}{1 + \sqrt{1+27(E^2-1)}},    \sqrt[3]{\frac{m}{(E^2-1)^2} \frac{1}{\kappa}}    \right\},
\]
then using $\frac{1}{\min\{a, b\}} = \max\{ \frac{1}{a}, \frac{1}{b} \}=\OM(\frac{1}{a} + \frac{1}{b})$ for $a, b > 0$, we have
\[
S = \OM\left(  \max\left\{ 1, \sqrt{\frac{\kappa E}{m}} \right\} + \kappa^{\frac{2}{3}} \sqrt[3]{\frac{E^2-1}{m^2}}   \right).
\]
Hence, the final projection complexity to obtain an $\eps$-optimal solution is
\[
S \log_2\frac{\EB \left[ F(\PAT(\x_0)) - F(\xc) \right]}{\eps} 
= \OM\left(  \left[  \sqrt{\kappa}  +  \kappa^{\frac{2}{3}} \sqrt[3]{\frac{E^2-1}{m^2}}   \right] \log_2\frac{\EB \left[ F(\PAT(\x_0)) - F(\xc) \right]}{\eps}   \right),
\]
no matter what the value of $E$ is.
\end{proof}

\subsection{Proof of Theorem~\ref{thm:acc_svrg_general}}
\begin{proof}
	For the general convex case $(\mu = 0)$, by Lemma~\ref{lem:error_acc_svrg_stage}, we have
	\begin{align*}
	\frac{1}{\theta_{s}^2}
	\EB &\left[F(\ttx_{s+1}) - F(\xc)\right]\\
	&\le \frac{1-\theta_s + \delta }{1-\delta}\frac{1}{\theta_{s}^2}  \left[F(\ttx_{s}) - F(\xc)\right]  +  \frac{1}{2\eta m(1-\delta)} \left[ \EB \| \PAT(\u_{0}^s) - \xc \|^2  - \EB \| \PAT(\u_{0}^{s+1}) - \xc \|^2 \right].
	\end{align*}
	Setting the sequence $\{\theta_{s}\}_{s=0}^{\infty}$ as that defined in Lemma~\ref{lem:theta}, so $	\frac{1-\theta_s+\delta}{1-\delta} \cdot \frac{1}{\theta_s ^2} = \frac{1}{\theta_{s-1}^2}$ and
	\begin{align*}
	\frac{1}{\theta_{s}^2}
	\EB &\left[F(\ttx_{s+1}) - F(\xc)\right]\\
	&\le \frac{1}{\theta_{s-1}^2}\EB \left[F(\ttx_{s}) - F(\xc)\right]  +  \frac{1}{2\eta m(1-\delta)} \left[ \EB \| \PAT(\u_{0}^s) - \xc \|^2  - \EB \| \PAT(\u_{0}^{s+1}) - \xc \|^2 \right].
	\end{align*}
	By telescoping, we arrive at
	\[
	\EB \left[F(\ttx_{S}) - F(\xc)\right] \le  \left[\frac{1}{\theta_0^2} \EB\left[F(\ttx_{0}) - F(\xc)\right] + \frac{1}{2\eta m(1-\delta)}\EB \| \PAT(\u_{0}^0) - \xc \|^2\right] \cdot \theta_{S}^2.
	\]
	By Lemma~\ref{lem:theta}, 
	if $E>1$, setting $\delta := 9(E^2-1)\eta^2L^2 = \frac{\ln(S)}{S} (S \ge 1)$ results to $\theta_{S}^2 = \widetilde{\OM}(\frac{1}{S^2})$.
	If $E=1$ (thus $\delta = 0$), we have $\theta_{S}^2 ={\OM}(\frac{1}{S^2})$.
	In both cases, $\theta_{s}$ decreases in $s$.
	Hence, in order to ensure~\eqref{eq:lr_acc_svrg} holds, we only need to make it solid when $s=0$ and thus we set
	\[
	1-\theta_{0}  = \frac{\eta L}{1-\eta L} \implies  \theta_{0}  = 1 - \frac{\eta L}{1-\eta L}
	\]
	where we requires $\eta L \le \frac{1}{2}$ to ensure the positiveness of $\theta_{0}$.
	Obviously $\theta_{0} \le 1 < 1 + \delta$.
	To ensure  $\theta_{0} \ge 2\delta$, we only need to let
	\[
	2\eta L + 18(E^2 - 1)\eta^2L^2 \le 1
	\implies
	\eta L  \le   \frac{1}{1 + \sqrt{1+18(E^2-1)}}.
	\]
	Hence, by setting 
	\[
	\eta L = \min\left\{ \frac{1}{1 + \sqrt{1+18(E^2-1)}}, \frac{1}{3\sqrt{E^2-1}} \sqrt{\frac{\ln(S)}{S}}  \right\},
	\]
	and using $\frac{1}{\min\{a, b\}} = \max\{ \frac{1}{a}, \frac{1}{b} \}=\OM(\frac{1}{a} + \frac{1}{b})$ for $a, b > 0$, we have
	\[
	\EB \left[F(\ttx_{S}) - F(\xc)\right] 
	= \widetilde{\OM}\left(   \frac{F(\x_{0}) - F(\xc)}{S^2} + \frac{LE\Delta^2}{mS^2} + \frac{\sqrt{E^2-1}L\Delta^2}{mS^{1.5}}  \right)
	\]
	where $\Delta^2 =  \EB \|\PAT(\u_{0}^0) - \xc \|^2=\EB \|\y_0 - \xc \|^2$.
\end{proof}

\section{Appendix for Section~\ref{sec:FL}}
\begin{proof}[Proof of Lemma~\ref{lem:sigma}]
	By definition,
	\begin{align*}
	\sigmaat &= \EB_{\xi}\| \PAT (\nabla F(\xc; \xi) )\|^2  &\\
	&=\EB_{\xi}\| \PAT (\nabla F(\xc; \xi) - \nabla F(\xc))\|^2&\text{By \ Remark}~\ref{rem:sigma} \\
	&= n  \EB_{\xi}\bigg\| \frac{1}{n} \sum_{k=1}^n \left[ \nabla f(\x^*; \xi_k) - \nabla f_k(\x^* )\right] \bigg\|^2&\text{By \ Lemma}~\ref{lem:proj} \\
	&= \frac{1}{n} \sum_{k=1}^n\EB_{\xi_k}\big\| \nabla f(\x^*; \xi_k) - \nabla f_k(\x^* ) \big\|^2&\text{Since \ }\xi_i \perp \xi_j \\
	&= \sigma_*^2,
	\end{align*} 
	and
	\begin{align*}
	\sigmaa 
	&= \EB_{\xi}\|\PA (\nabla F(\xc; \xi) )\|^2  \\ 
	&= \EB_{\xi}\|\nabla F(\xc; \xi)\|^2 - \EB_{\xi}\| \PAT (\nabla F(\xc; \xi) )\|^2  \\
	&=\|\nabla F(\xc) \|^2 +  \EB_{\xi}\|\nabla F(\xc; \xi) - \nabla F(\xc) \|^2 - \EB_{\xi}\| \PAT (\nabla F(\xc; \xi) )\|^2  \\
	&= n\zeta_*^2 + n \sigma_*^2 - \sigmaat\\
	&=n 	\zeta_*^2 + (n-1) \sigma_*^2.
	\end{align*} 
\end{proof}

\begin{proof}[Proof of Corollary~\ref{cor:local_sgd}]
	From the discussion above Corollary~\ref{cor:local_sgd}, $F(\x) = \sum_{k=1}^n f_k(\x_k)$ the objective function of  $\x = [\x_1^\top, \cdots, \x_n^\top]^\top \in \RB^{nd}$ and satisfies all the assumptions.
	Then we have the bound~\eqref{eq:y_mu0} and~\eqref{eq:y_mu>} hold for this $F(\x)$.
	By dividing $n$ on both sides of~\eqref{eq:y_mu0} and~\eqref{eq:y_mu>}, we obtain bounds for $\frac{1}{n} \sum_{k=1}^n f_k(\hat{\x})  - \frac{1}{n} \sum_{k=1}^n f_k(\x^*)$.
	
	The rest is to replace corresponding parameters. 
	From Lemma~\ref{lem:sigma}, $\sigmaat  =   \sigma_*^2$ and $\sigmaa = n 	\zeta_*^2 + (n-1) \sigma_*^2$, so from~\eqref{eq:new_sigmaa},
	\[\widetilde{\sigma}_{\A,*}^2 = \sigmaa + (E-1)\|\nabla F(\xc)\|^2 =  \sigmaa + (E-1)n\zeta_*^2=nE\zeta_*^2 + (n-1) \sigma_*^2.\]
	Besides, we have $\Delta^2 = n B^2$.
	Then the conclusion follows.
\end{proof}

\end{appendix}

\end{document}